\documentclass[11pt]{article}
\usepackage{latexsym}
\usepackage{amsthm}
\usepackage{amsmath,amssymb,amsfonts,amsthm}
\usepackage{mdframed}
\usepackage{graphicx}
\usepackage{stmaryrd}
\usepackage{gensymb}
\usepackage{bm}
\usepackage{mathtools}
\usepackage{fontspec}
\usepackage{multirow}

\usepackage{tikz}
\usetikzlibrary{matrix,arrows}
\usetikzlibrary{positioning}
\usetikzlibrary{fit}
\usetikzlibrary{patterns}

\usepackage[colorlinks,
linkcolor=blue,
anchorcolor=blue,
citecolor=blue
]{hyperref}
\usepackage[margin=2.9cm]{geometry}
\DeclareMathOperator{\h}{\mbox{\sf{h}}}
\DeclareMathOperator{\crr}{\mbox{\sf{cr}}}

\DeclareMathOperator{\nr}{\mbox{\sf{nr}}}
\DeclareMathOperator{\bl}{\mbox{\sf{bl}}}
\DeclareMathOperator{\mcr}{\mbox{\sf{mcr}}}
\DeclareMathOperator{\smc}{\mbox{\sf{smc}}}

\DeclareMathOperator{\magg}{\mbox{\sf{mag}}}
\DeclareMathOperator{\wid}{\mbox{\sf{w}}}
\DeclareMathOperator{\minn}{\mbox{\sf{min}}}
\DeclareMathOperator{\ssd}{\mbox{\sf{ssd}}}

\DeclareMathOperator{\asc}{\mbox{\sf{asc}}}

\DeclareMathOperator{\idr}{\mbox{\sf{idr}}}
\DeclareMathOperator{\lmax}{\mbox{\sf{lmax}}}

\DeclareMathOperator{\rmin}{\mbox{\sf{rmin}}}
\DeclareMathOperator{\zero}{\mbox{\sf{zero}}}

\DeclareMathOperator{\fcr}{\mbox{\sf{fcr}}}

\newcommand{\etal}{et~al.}

\usepackage[normalem]{ulem}

\newtheorem{thm}{Theorem}[section]

\newtheorem{lem}[thm]{Lemma}

\newtheorem{coro}[thm]{Corollary}
\newtheorem{conj}[thm]{Conjecture}

\newtheorem{remark}[thm]{Remark}

\newcommand{\A}{\mathcal{A}}      

\newcommand{\PP}{{\mathcal P}}
\newcommand{\Asc}[1]{\A_{#1}}
\newcommand{\D}{\mathcal{D}}
\newcommand{\F}{\mathcal{F}}
\newcommand{\M}{\mathcal{M}}

\newcommand{\V}{\mathcal{V}}

\newcommand{\red}{\mbox{\sf{redarc}}}

\newcommand{\T}{\mathcal{T}}

\begin{document}
	
	\begin{center}
		{\large \bf Catalan structures arising from pattern-avoiding Stoimenow matchings and other Fishburn objects}
	\end{center}
	
	\begingroup
	\renewcommand{\thefootnote}{\fnsymbol{footnote}}
	\begin{center}
		Shuzhen Lv$^{a}$, Sergey Kitaev$^{b}$ and Philip B. Zhang$^{c}$
		\\[6pt]
		
		$^{a,c}$College of Mathematical Sciences \& Institute of Mathematics and Interdisciplinary Sciences, Tianjin Normal University, \\ Tianjin  300387, P. R. China\\[6pt]
		
		$^{b}$Department of Mathematics and Statistics, University of Strathclyde, \\ 26 Richmond Street, Glasgow G1 1XH, UK\\[6pt]
		
		Email:  $^{a}${\tt  lvshuzhenlsz@yeah.net},
		$^{b}${\tt sergey.kitaev@strath.ac.uk},
		$^{c}${\tt zhang@tjnu.edu.cn}
	\end{center}

	\endgroup

	\noindent\textbf{Abstract.}
	In connection with Vassiliev's knot invariants, Stoimenow introduced in 1998 a class of matchings, also known as regular linearized chord diagrams. These matchings are linked to various combinatorial structures, all of which are associated with the Fishburn numbers. In this paper, we address a problem posed by Bevan et al.\ in 2025 concerning the identification of subsets of Stoimenow matchings that are counted by the Catalan numbers. We present five solutions in terms of pattern-avoiding matchings. We also consider four infinite families of patterns that generalize four of the five forbidden patterns appearing in the solution to the problem we solved and prove that the matchings avoiding them are equinumerous. Finally, we establish numerous results on distributions and joint equidistribution of statistics over these Catalan-counted subsets of Fishburn structures, namely Stoimenow matchings, $(2+2)$-free posets, ascent sequences, and Fishburn permutations, notably expressing some of them in terms of Narayana numbers and others in terms of ballot numbers. \\[-3mm]
	
	
	
	\section{Introduction}
	To give upper bounds on the dimension of the space of Vassiliev's knot
	invariants of a given degree, Stoimenow~\cite{Stoimenow} introduced what
	he called \textit{regular linearized chord diagrams}, which are now commonly referred to as \textit{Stoimenow matchings}. 
	
	\subsection{Stoimenow matchings} 
	A \textit{matching} on the set of integers \(\{1, 2, \ldots, 2n\}\) is a partition of the set into blocks of size 2, called \textit{arcs}. 
	For an arc \([a, b]\) where \(a < b\), we call  \(a\) (resp., \(b\))  the \textit{opener} (resp., \textit{closer}) of the arc. 
	Matchings can be represented by arranging openers and closers in increasing order from left to right and connecting each pair with a curve. 
	A matching is called a \textit{Stoimenow matching} if there
	are no occurrences of Type I or Type II arcs, which are defined as follows: \\[-3mm]
	
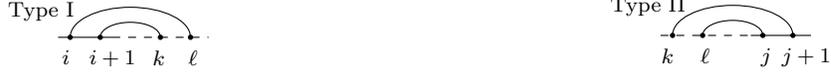
	\begin{minipage}[b]{0.45\textwidth}
		\centering
		\begin{tikzpicture}[scale = 0.4]
			\draw (-0.4,0) -- (1.4,0);
			\draw [dashed](1.4,0) -- (4.6,0);
			\foreach \x in {0,1,3,4}
			{
				\filldraw (\x,0) circle (2pt);
			}
			\draw[black] (0,0) arc (180:0:2 and 1);
			\draw[black] (1,0) arc (180:0:1 and 0.5);
			\node[anchor=north west] at (-0.6, -0.1) {\scriptsize{$i$}};
			\node[anchor=north west] at (0.3, -0.1) {\scriptsize{$i+1$}};
			\node[anchor=north west] at (2.4, -0.1) {\scriptsize{$k$}};
			\node[anchor=north west] at (3.6, -0.1) {\scriptsize{$\ell$}};
			\node[anchor=north west] at (-2.4, 1.5) {{\scriptsize Type I}};
		\end{tikzpicture}
		\vspace{0.2cm} 
		
	\end{minipage}
	\hfill 
	\begin{minipage}[b]{0.45\textwidth}
		\centering
		\begin{tikzpicture}[scale = 0.4]
			\draw[dashed] (-0.4,0) -- (2.6,0);
			\draw (2.6,0) -- (4.6,0);
			\foreach \x in {0,1,3,4}
			{
				\filldraw (\x,0) circle (2pt);
			}
			\draw[black] (0,0) arc (180:0:2 and 1);
			\draw[black] (1,0) arc (180:0:1 and 0.5);
			\node[anchor=north west] at (-0.7, -0.1) {\scriptsize{$k$}};
			\node[anchor=north west] at (0.6, -0.1) {\scriptsize{$\ell$}};
			\node[anchor=north west] at (2.6, -0.1) {\scriptsize{$j$}};
			\node[anchor=north west] at (3.3, -0.1) {\scriptsize{$j+1$}};
			\node[anchor=north west] at (-2.4, 1.6) {{\scriptsize Type II}};
		\end{tikzpicture}
		\vspace{0.2cm}
	\end{minipage}
	
	\vspace{-2mm}
	\noindent
	The \textit{length} of a matching is the number of arcs it contains. We denote by $\M_n$ the set of all Stoimenow matchings of length $n$, with $|\M_0|=1$ corresponding to the \textit{empty matching}. As an example, there are five matchings in $\M_3$: 
	
	
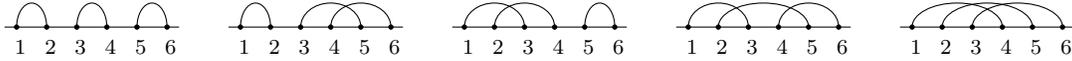
\vspace{-0.3cm}
	
	\begin{center}
		\begin{tikzpicture}[scale = 0.4]
			\draw (-0.4,0) -- (5.4,0);
			\foreach \x in {0,1,2,3,4,5}
			{
				\filldraw (\x,0) circle (2pt);
			}
			\draw[black] (0,0) arc (180:0:0.5 and 0.8);
			\draw[black] (2,0) arc (180:0:0.5 and 0.8);
			\draw[black] (4,0) arc (180:0:0.5 and 0.8);
			\node[anchor=north west] at (-0.4, -0.1) {\scriptsize{$1$}};
			\node[anchor=north west] at (0.6, -0.1) {\scriptsize{$2$}};
			\node[anchor=north west] at (1.6, -0.1) {\scriptsize{$3$}};
			\node[anchor=north west] at (2.6, -0.1) {\scriptsize{$4$}};
			\node[anchor=north west] at (3.6, -0.1) {\scriptsize{$5$}};
			\node[anchor=north west] at (4.6, -0.1) {\scriptsize{$6$}};
		\end{tikzpicture}
		\hspace{0.3cm}
		\begin{tikzpicture}[scale = 0.4]
			\draw (-0.4,0) -- (5.4,0);
			\foreach \x in {0,1,2,3,4,5}
			{
				\filldraw (\x,0) circle (2pt);
			}
			\draw[black] (0,0) arc (180:0:0.5 and 0.8);
			\draw[black] (2,0) arc (180:0:1 and 0.8);
			\draw[black] (3,0) arc (180:0:1 and 0.8);
			\node[anchor=north west] at (-0.4, -0.1) {\scriptsize{$1$}};
			\node[anchor=north west] at (0.6, -0.1) {\scriptsize{$2$}};
			\node[anchor=north west] at (1.6, -0.1) {\scriptsize{$3$}};
			\node[anchor=north west] at (2.6, -0.1) {\scriptsize{$4$}};
			\node[anchor=north west] at (3.6, -0.1) {\scriptsize{$5$}};
			\node[anchor=north west] at (4.6, -0.1) {\scriptsize{$6$}};
		\end{tikzpicture}
		\hspace{0.3cm}
		\begin{tikzpicture}[scale = 0.4]
			\draw (-0.4,0) -- (5.4,0);
			\foreach \x in {0,1,2,3,4,5}
			{
				\filldraw (\x,0) circle (2pt);
			}
			\draw[black] (0,0) arc (180:0:1 and 0.8);
			\draw[black] (1,0) arc (180:0:1 and 0.8);
			\draw[black] (4,0) arc (180:0:0.5 and 0.8);
			\node[anchor=north west] at (-0.4, -0.1) {\scriptsize{$1$}};
			\node[anchor=north west] at (0.6, -0.1) {\scriptsize{$2$}};
			\node[anchor=north west] at (1.6, -0.1) {\scriptsize{$3$}};
			\node[anchor=north west] at (2.6, -0.1) {\scriptsize{$4$}};
			\node[anchor=north west] at (3.6, -0.1) {\scriptsize{$5$}};
			\node[anchor=north west] at (4.6, -0.1) {\scriptsize{$6$}};
		\end{tikzpicture}
		\hspace{0.3cm}
		\begin{tikzpicture}[scale = 0.4]
			\draw (-0.4,0) -- (5.4,0);
			\foreach \x in {0,1,2,3,4,5}
			{
				\filldraw (\x,0) circle (2pt);
			}
			\draw[black] (0,0) arc (180:0:1 and 0.8);
			\draw[black] (1,0) arc (180:0:1.5 and 0.8);
			\draw[black] (3,0) arc (180:0:1 and 0.8);
			\node[anchor=north west] at (-0.4, -0.1) {\scriptsize{$1$}};
			\node[anchor=north west] at (0.6, -0.1) {\scriptsize{$2$}};
			\node[anchor=north west] at (1.6, -0.1) {\scriptsize{$3$}};
			\node[anchor=north west] at (2.6, -0.1) {\scriptsize{$4$}};
			\node[anchor=north west] at (3.6, -0.1) {\scriptsize{$5$}};
			\node[anchor=north west] at (4.6, -0.1) {\scriptsize{$6$}};
		\end{tikzpicture}
		\hspace{0.3cm}
		\begin{tikzpicture}[scale = 0.4]
			\draw (-0.4,0) -- (5.4,0);
			\foreach \x in {0,1,2,3,4,5}
			{
				\filldraw (\x,0) circle (2pt);
			}
			\draw[black] (0,0) arc (180:0:1.5 and 0.8);
			\draw[black] (1,0) arc (180:0:1.5 and 0.8);
			\draw[black] (2,0) arc (180:0:1.5 and 0.8);
			\node[anchor=north west] at (-0.4, -0.1) {\scriptsize{$1$}};
			\node[anchor=north west] at (0.6, -0.1) {\scriptsize{$2$}};
			\node[anchor=north west] at (1.6, -0.1) {\scriptsize{$3$}};
			\node[anchor=north west] at (2.6, -0.1) {\scriptsize{$4$}};
			\node[anchor=north west] at (3.6, -0.1) {\scriptsize{$5$}};
			\node[anchor=north west] at (4.6, -0.1) {\scriptsize{$6$}};
		\end{tikzpicture}
		\vspace{0.2cm} 
	\end{center}
	\vspace{-0.6cm}
	
	It is known that Stoimenow matchings are in one-to-one correspondence with several combinatorial objects that are enumerated  by the \textit{Fishburn numbers} (sequence A022493 in OEIS~\cite{OEIS}). These include \textit{interval orders}, \textbf{(2+2)}-\textit{free posets}, \textit{ascent sequences}, \textit{Fishburn permutations}, and \textit{Fishburn matrices}~\cite{Bousquet-Claesson-Dukes-Kitaev, Dukes-Parviainen}. This counting sequence has the following generating function (g.f.):
	\[
	\sum_{n\geq 0} \prod_{k=1}^{n} (1-(1-t)^k)=1+t+2t^2+5t^3+15t^4+53t^5+217t^6+\cdots.
	\]
	Claesson and Linusson \cite{Claesson-Linusson} studied matchings of length $n$ that avoid Type I arcs, referred to as matchings without \textit{left} (\textit{neighbor}) \textit{nestings}, which are enumerated by $n!$. They also constructed bijections between such matchings and \textit{permutations}, \textit{inversion tables}, and \textit{factorial posets}.
	
	Throughout this paper, we assume that a matching $M\in\M_n$ consists of the arcs $$[a_1,b_1],  [a_2,b_2],\ldots,[a_n,b_n],$$ where $1=a_1<a_2<\cdots<a_n$ and $b_n=2n$. 
	A subset of arcs  $[a_{i_1},b_{i_1}]$, $[a_{i_2},b_{i_2}],\ldots,[a_{i_k},b_{i_k}]$ in~$M$ where $k\geq 1$ is called a
	\begin{itemize}
		\item $k$-\textit{chain} if $a_{i_1}<a_{i_2}<b_{i_1}<a_{i_3}<b_{i_2}<a_{i_4}<b_{i_3}<\cdots<a_{i_k}<b_{i_{k-1}}<b_{i_k}$,
		\item $k$-\textit{crossing} if $a_{i_1}<a_{i_2}<\cdots<a_{i_k}<b_{i_1}<b_{i_2}<\cdots <b_{i_k}$, and
		\item $k$-\textit{noncrossing} if $a_{i_1}<b_{i_1}<a_{i_2}<b_{i_2}<\cdots<a_{i_k}<b_{i_k}$.
	\end{itemize}
	For example, in the matching $M$ shown in Figure~\ref{fig:matching_example}, there are two 3-chains, namely $\{[1,4],[3,7],[6,8]\}$ and $\{[2,5],[3,7],[6,8]\}$; one 3-crossing, namely $\{[1,4],[2,5],[3,7]\}$; and two 2-noncrossings, namely $\{[1,4],[6,8]\}$ and $\{[2,5],[6,8]\}$.
	
	\begin{figure}[ht]
		\centering
		\begin{tikzpicture}[scale = 0.35] 
			
			\draw (-0.5,0) -- (7.5,0);
			
			\foreach \x [count=\xi] in {0,1,2,3,4,5,6,7} {
				\filldraw (\x,0) circle (2pt) node[below=3pt] {\scriptsize \xi};
			}
			\draw[black] (0,0) arc (180:0:1.5 and 1); 
			\draw[black] (1,0) arc (180:0:1.5 and 1); 
			\draw[black] (2,0) arc (180:0:2 and 1);   
			\draw[black] (5,0) arc (180:0:1 and 1); 
			\node at (-0.8, 1) {\scriptsize $M$};
		\end{tikzpicture}
		\caption{An example of a matching $M$ with chains, crossings, and noncrossings.}
		\label{fig:matching_example}
	\end{figure}
	
	\subsection{Results in this paper}
	Bevan et al.~\cite{Bevan-Cheon-Kitaev} presented a hierarchy related to the Fishburn numbers (see \cite[Fig.~12]{Bevan-Cheon-Kitaev}), and in relation to this hierarchy, they stated the problem of describing a subset of Stoimenow matchings that are counted by the Catalan numbers. 
	Recall that the {\it $n$-th Catalan number} $C_n=\frac{1}{n+1}{2n\choose n}$ satisfies the recurrence relation
	\begin{equation}\label{Catalan-rec}
		C_n= \displaystyle\sum_{k=1}^{n} C_{k-1}\,C_{n-k}
	\end{equation}
	with $C_0=1$. This sequence is recorded as A000108 in~\cite{OEIS}, and its g.f.\ is given by 
	\(
	C(t)=\frac{1-\sqrt{1-4t}}{2t}=1+t+2t^2+5t^3+14t^4+42t^5+132t^6+\cdots.
	\)
	The Catalan numbers have already appeared in the literature in the context of matchings, specifically as so-called \emph{nonnesting matchings}~\cite{Stanley}, defined as matchings in which no two arcs ``nest'' over each other.
	Since Type~I and Type II arcs are specific cases of nesting arcs, we automatically obtain a solution to the problem posed by Bevan et al. \cite{Bevan-Cheon-Kitaev}, as the set of nonnesting matchings of a given length is a subset of Stoimenow matchings of the same length. However, in Theorem~\ref{thm-Catalan}, we obtain additional solutions to the problem by introducing forbidden patterns of length 4 (i.e., configurations of four arcs), with the ``nonnesting solution'' corresponding to the avoidance of a particular pattern.
	
	We say that a matching in \(\mathcal{M}_n\) contains a pattern if the resulting submatching, obtained by removing some arcs, is isomorphic to that pattern; otherwise the matching is said to \textit{avoid} the pattern.
	In this paper we introduce the following five forbidden configurations (patterns):\\[-3mm]
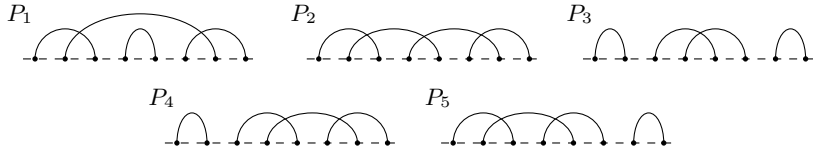
	\[
	\begin{tikzpicture}[scale = 0.4] 
		\node at (-0.5,1.5) {\upshape {\footnotesize{$P_1$}}};
		\draw[dashed] (-0.4,0) -- (7.4,0);
		\foreach \x in {0,1,2,3,4,5,6,7}
		{
			\filldraw (\x,0) circle (2pt);
		}
		\draw[black] (0,0) arc (180:0:1 and 1);
		\draw[black] (1,0) arc (180:0:2.5 and 1.5);
		\draw[black] (3,0) arc (180:0:0.5 and 1);
		\draw[black] (5,0) arc (180:0:1 and 1);
	\end{tikzpicture}
	\hspace{0.3cm}
	\begin{tikzpicture}[scale = 0.4] 
		\node at (-0.5,1.5) {\upshape {\footnotesize{$P_2$}}};
		\draw[dashed] (-0.4,0) -- (7.4,0);
		\foreach \x in {0,1,2,3,4,5,6,7}
		{
			\filldraw (\x,0) circle (2pt);
		}
		\draw[black] (0,0) arc (180:0:1 and 1);
		\draw[black] (1,0) arc (180:0:1.5 and 1);
		\draw[black] (3,0) arc (180:0:1.5 and 1);
		\draw[black] (5,0) arc (180:0:1 and 1);
	\end{tikzpicture}
	\hspace{0.2cm}
	\begin{tikzpicture}[scale = 0.4] 
		\node at (-0.5,1.5) {\upshape {\footnotesize{$P_3$}}};
		\draw[dashed] (-0.4,0) -- (7.4,0);
		\foreach \x in {0,1,2,3,4,5,6,7}
		{
			\filldraw (\x,0) circle (2pt);
		}
		\draw[black] (0,0) arc (180:0:0.5 and 1);
		\draw[black] (2,0) arc (180:0:1 and 1);
		\draw[black] (3,0) arc (180:0:1 and 1);
		\draw[black] (6,0) arc (180:0:0.5 and 1);
	\end{tikzpicture}
	\]
	\[
	\begin{tikzpicture}[scale = 0.4] 
		\node at (-0.5,1.5) {\upshape {\footnotesize{$P_4$}}};
		\draw[dashed] (-0.4,0) -- (7.4,0);
		\foreach \x in {0,1,2,3,4,5,6,7}
		{
			\filldraw (\x,0) circle (2pt);
		}
		\draw[black] (0,0) arc (180:0:0.5 and 1);
		\draw[black] (2,0) arc (180:0:1 and 1);
		\draw[black] (3,0) arc (180:0:1.5 and 1);
		\draw[black] (5,0) arc (180:0:1 and 1);
	\end{tikzpicture}
	\hspace{0.2cm}
	\begin{tikzpicture}[scale = 0.4] 
		\node at (-0.5,1.5) {\upshape {\footnotesize{$P_5$}}};
		\draw[dashed] (-0.4,0) -- (7.4,0);
		\foreach \x in {0,1,2,3,4,5,6,7}
		{
			\filldraw (\x,0) circle (2pt);
		}
		\draw[black] (0,0) arc (180:0:1 and 1);
		\draw[black] (1,0) arc (180:0:1.5 and 1);
		\draw[black] (3,0) arc (180:0:1 and 1);
		\draw[black] (6,0) arc (180:0:0.5 and 1);
	\end{tikzpicture}
	\]
	For \(i \in \{1, 2, 3, 4, 5\}\) we denote by \(\mathcal{M}_n(P_i)\) the set of Stoimenow matchings with \(n\) arcs avoiding the pattern \(P_i\), and $\M(P_i)=\cup_{n\geq 0}\M_n(P_i)$. Two patterns $P$ and $Q$ are called \textit{Wilf-equivalent} if $|\mathcal{M}_n(P)|=|\mathcal{M}_n(Q)|$ for all $n\geq 0$.
	
	We also introduce four infinite families of forbidden patterns that generalize the patterns $P_2$--$P_5$ (which correspond to the case $k=4$ as follows): \\[-3mm]
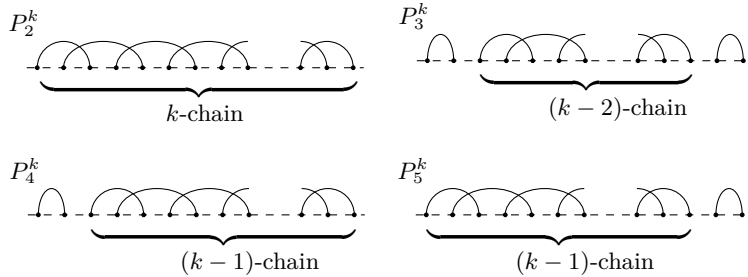
	\[
	\begin{tikzpicture}[scale = 0.35] 
		\node at (-0.5,1.7) {\upshape {\footnotesize{$P^k_2$}}};
		\draw[dashed] (-0.4,0) -- (12.4,0);
		\foreach \x in {0,1,2,3,4,5,6,7,8,10,11,12}
		{
			\filldraw (\x,0) circle (2pt);
		}
		
		\draw[domain=0:90, smooth, variable=\t, shift={(8,0)}] 
		plot ({-cos(\t)}, {sin(\t)});
		
		\draw[domain=0:90, smooth, variable=\t, shift={(10,0)}] 
		plot ({cos(\t)}, {sin(\t)});
		
		\draw[black] (0,0) arc (180:0:1 and 1);
		\draw[black] (1,0) arc (180:0:1.5 and 1);
		\draw[black] (3,0) arc (180:0:1.5 and 1);
		\draw[black] (5,0) arc (180:0:1.5 and 1);
		\draw[black] (10,0) arc (180:0:1 and 1);
		\node[anchor=north west] at (-0.25, -0.05) {$\underbrace{ \hspace{4.2cm}}$};
		\node[anchor=north west] at (4.5, -1) {\footnotesize{$k$-chain}};
	\end{tikzpicture}
	\hspace{0.3cm}
	\begin{tikzpicture}[scale = 0.35] 
		\node at (-0.5,1.7) {\upshape {\footnotesize{$P^k_3$}}};
		\draw[dashed] (-0.4,0) -- (12.4,0);
		\foreach \x in {0,1,2,3,4,5,6,8,9,10,11,12}
		{
			\filldraw (\x,0) circle (2pt);
		}
		\draw[domain=0:90, smooth, variable=\t, shift={(6,0)}] 
		plot ({-cos(\t)}, {sin(\t)});
		
		\draw[domain=0:90, smooth, variable=\t, shift={(8,0)}] 
		plot ({cos(\t)}, {sin(\t)});
		
		\draw[black] (0,0) arc (180:0:0.5 and 1);
		\draw[black] (2,0) arc (180:0:1 and 1);
		\draw[black] (3,0) arc (180:0:1.5 and 1);
		\draw[black] (8,0) arc (180:0:1 and 1);
		\draw[black] (11,0) arc (180:0:0.5 and 1);
		\node[anchor=north west] at (1.65, -0.05) {$\underbrace{ \hspace{2.8cm}}$};
		\node[anchor=north west] at (4.2, -1) {\footnotesize{$(k-2)$-chain}};
	\end{tikzpicture}
	\]
	\[
	\begin{tikzpicture}[scale = 0.35] 
		\node at (-0.5,1.7) {\upshape {\footnotesize{$P^k_4$}}};
		\draw[dashed] (-0.4,0) -- (12.4,0);
		\foreach \x in {0,1,2,3,4,5,6,7,8,10,11,12}
		{
			\filldraw (\x,0) circle (2pt);
		}
		\draw[domain=0:90, smooth, variable=\t, shift={(8,0)}] 
		plot ({-cos(\t)}, {sin(\t)});
		
		\draw[domain=0:90, smooth, variable=\t, shift={(10,0)}] 
		plot ({cos(\t)}, {sin(\t)});
		
		\draw[black] (0,0) arc (180:0:0.5 and 1);
		\draw[black] (2,0) arc (180:0:1 and 1);
		\draw[black] (3,0) arc (180:0:1.5 and 1);
		\draw[black] (5,0) arc (180:0:1.5 and 1);
		\draw[black] (10,0) arc (180:0:1 and 1);
		\node[anchor=north west] at (1.65, -0.05) {$\underbrace{ \hspace{3.5cm}}$};
		\node[anchor=north west] at (5, -1) {\footnotesize{$(k-1)$-chain}};
	\end{tikzpicture}
	\hspace{0.3cm}
	\begin{tikzpicture}[scale = 0.35] 
		\node at (-0.5,1.7) {\upshape {\footnotesize{$P^k_5$}}};
		\draw[dashed] (-0.4,0) -- (12.4,0);
		\foreach \x in {0,1,2,3,4,5,6,8,9,10,11,12}
		{
			\filldraw (\x,0) circle (2pt);
		}
		
		\draw[domain=0:90, smooth, variable=\t, shift={(6,0)}] 
		plot ({-cos(\t)}, {sin(\t)});
		
		\draw[domain=0:90, smooth, variable=\t, shift={(8,0)}] 
		plot ({cos(\t)}, {sin(\t)});
		
		\draw[black] (0,0) arc (180:0:1 and 1);
		\draw[black] (1,0) arc (180:0:1.5 and 1);
		\draw[black] (3,0) arc (180:0:1.5 and 1);
		\draw[black] (8,0) arc (180:0:1 and 1);
		\draw[black] (11,0) arc (180:0:0.5 and 1);
		\node[anchor=north west] at (-0.35, -0.05) {$\underbrace{ \hspace{3.5cm}}$};
		\node[anchor=north west] at (3, -1) {\footnotesize{$(k-1)$-chain}};
	\end{tikzpicture}
	\]
	
	In Section~\ref{sec-matching-Catalan}, we prove that Stoimenow matchings avoiding any of the patterns $P_1$--$P_5$ are counted by the Catalan numbers and establish a generalized Wilf-equivalence result for $P_2$--$P_5$ in terms of $P^k_2$--$P^k_5$, as stated in Theorems~\ref{thm-Catalan} and~\ref{thm1-general}.
	
	\begin{thm}\label{thm-Catalan} For $1\leq i\leq 5$ and $n\geq 0$, we have $|\M_n(P_i)|=C_n$.\end{thm}
	
	\begin{thm}\label{thm1-general}   We have $|\M_n(P^k_2)|=|\M_n(P^k_3)|=|\M_n(P^k_4)|=|\M_n(P^k_5)|$ for all $k\geq 1$. \end{thm}
	
	In Section~\ref{sec-bijection}, we recall the known bijections from~\cite{Bousquet-Claesson-Dukes-Kitaev,Claesson-Dukes-Kitaev} between Stoimenow matchings, $\textbf{(2+2)}$-free posets, ascent sequences, and Fishburn permutations. Furthermore, we investigate their restrictions to pattern-avoiding subsets, as outlined in Figure~\ref{fig-bijection-P1-P2}, where we do not include  $\mathcal{A}_n(\Psi(P_1))$ and $\mathcal{F}_n(\Upsilon(P_1))$ as these sets are not counted by Catalan numbers. Note that in Figure~\ref{fig-bijection-P1-P2}, the subsets labeled A (resp. B) correspond to each other under the bijections.
	
	\begin{figure}
		\centering
		\begin{tikzpicture}[scale=0.3, rounded corners = 3pt]
			\begin{scope}
				\draw (3,5.5) node {\scriptsize{$\M_n$: Stoimenow matchings}};
				\draw (0.4,2.5) node {\tiny{$\M_n(P_1)$}};
				\draw (7.8,3.5) node {\tiny{$\M_n(P_2)$}};
				\draw (-5,1) rectangle (12,6.5);
				\draw (-4.3,1.5) rectangle (5,4);
				\draw (2.8,2.3) rectangle (11.5,4.8);
				\draw (-3.8,3.4) node {\scriptsize{$A$}};
				\draw (11,4.2) node {\scriptsize{$B$}};
			\end{scope}
			\begin{scope}[xshift = 130ex]
				\draw (3,5.5) node {\scriptsize{$\mathcal{P}_n$}: \textbf{(2+2)}-free posets};
				\draw (0.3,2.3) node[text width=3cm, align=center, font=\tiny] {$\mathcal{P}_n\textbf{(3+1)}$:\\ \textbf{(2+2, 3+1)}-free posets};
				\draw (7.8,3.5) node[text width=3cm, align=center, font=\tiny] {$\mathcal{P}_n\textbf{(N)}$:\\ \textbf{(2+2, N)}-\\ free posets};
				\draw (-5,1) rectangle (12,6.5);
				\draw (-4.3,1.5) rectangle (5,4);
				\draw (2.8,2.3) rectangle (11.5,4.8);
				\draw (-3.8,3.4) node {\scriptsize{$A$}};
				\draw (11,4.2) node {\scriptsize{$B$}};
			\end{scope}
			\begin{scope}[yshift = -55ex]
				\draw (3.3,5.5) node {\scriptsize{$\F_n$: Fishburn permutations}};
				\draw (4,3) node[text width=3cm, align=center, font=\tiny] {$\F_n(3142)$:\\ 3142-avoiding\\ Fishburn permutations};
				\draw (-5,1) rectangle (12,6.5);
				\draw (-0.5,1.5) rectangle (8.2,4.5);
				\draw (7.6,4) node {\scriptsize{$B$}};
			\end{scope}
			\begin{scope}[xshift = 130ex, yshift = -55ex]
				\draw (3.7,5.5) node {\scriptsize{$\A_n$: Ascent sequences}};
				\draw (3.8,3) node[text width=3cm, align=center, font=\tiny] {$\A_n(101)$:\\ 101-avoiding\\ ascent sequences};
				\draw (-5,1) rectangle (12,6.5);
				\draw (-0.5,1.5) rectangle (8.2,4.5);
				\draw (7.6,4) node {\scriptsize{$B$}};
			\end{scope}
			\draw[->] (3,0.8) -- (3,-2.2);
			\draw (5.6,-0.8) node {\footnotesize{$\Upsilon=\Lambda\circ \Psi$}};
			\draw[->] (12.2,3.5) -- (16.2,3.5);
			\draw (14.2,4.2) node {\footnotesize{$\Omega$}};
			\draw[->] (12.2,1) -- (16.2,-2.5);
			\draw (16.7, -0.4) node {\footnotesize{$\Psi=\mathcal{O}\circ \Omega$}};
			\draw[<-] (12.2,-5.5) -- (16.2,-5.5);
			\draw (14.2,-4.8) node {\footnotesize{$\Lambda$}};
			\draw[->] (25.2,0.8) -- (25.2,-2.3);
			\draw (25.8, -0.9) node {\footnotesize{$\mathcal{O}$}};
		\end{tikzpicture}
		\caption{A summary of the sets and bijections in Section~\ref{sec-bijection}.}
		\label{fig-bijection-P1-P2}
	\end{figure}
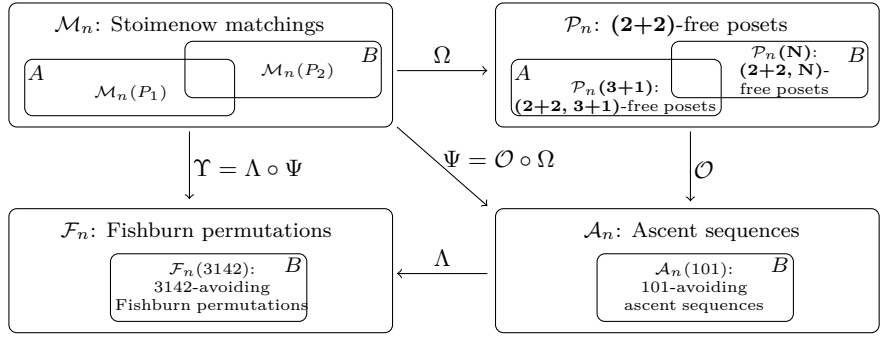
	
	In Section~\ref{sec-stats}, we use the bijections presented in Section~\ref{sec-bijection} to study the distribution of statistics over pattern-avoiding Stoimenow matchings and other Fishburn structures. We focus on $\M_n(P_1)$ and $\M_n(P_2)$ because these sets have natural counterparts among $\mathbf{(2+2)}$-free posets counted by the Catalan numbers. However, one can use the bijections in Theorem~\ref{thm1-general} for $k=4$, which link $\M_n(P_2)$--$\M_n(P_5)$, to obtain induced distribution results for $\M_n(P_3)$--$\M_n(P_5)$ that we do not investigate in this paper.
	
	Given $M \in \mathcal{M}_n$, let $\crr(M)$ denote the largest $k$ such that $M$ contains a $k$-crossing, 
	which was studied by Chen et~al.~\cite{Chen-Deng-Du-Stanley-Yan}. 
	For a poset $P$, $\wid(P)$ denotes the width of $P$, that is, the maximum size of an antichain.  Let $\mathcal{D}_n$ denote the set of Dyck paths of semilength $n$, and let $\mbox{\sf{h}}$ be the height statistic on Dyck paths, 
	as defined in Section~\ref{proof-thm-1-3}. 
	We refer to Section~\ref{sec-bijection} for the definitions of the sets of posets $\PP_n\mathbf{(3+1)}$ and $\PP_n\mathbf{(N)}$. 
	The following theorem will be proved in Section~\ref{proof-thm-1-3}.
	
	\begin{thm}\label{thm-equi-cr-nr-P1}
		We have
		\begin{align}\label{sta-cr-nr-P1}
			\sum_{M \in \M_n(P_1)} x^{\mbox{\sf{\footnotesize cr}}(M)}= \sum_{P\in \mathcal{P}_n{\bf (3+1)}} x^{\mbox{\sf{\footnotesize w}}(P)}= \sum_{P\in \mathcal{P}_n{\bf (N)}} x^{\mbox{\sf{\footnotesize w}}(P)} =\sum_{\mu \in \D_n} x^{\mbox{\sf{\footnotesize h}}(\mu)}.
		\end{align}
	\end{thm}
	
	The Narayana polynomials $N_n(x)$ are  defined by  
	\(
	N_n(x) = \sum_{k=1}^{n} \frac{1}{n} \binom{n}{k} \binom{n}{k-1} x^k,
	\)
	see OEIS~\cite[A001263]{OEIS}.
	The ordinary generating function for these polynomials  is
	\begin{equation}\label{Narayana-formula}
		N(x,t) =1+ \sum_{n=1}^{\infty} N_n(x) t^n= \frac{1+2t-t(1+x)-\sqrt{\bigl(1-t(1+x)\bigr)^{2}-4xt^{2}}}{2t}.
	\end{equation}
	
	Let $\nr(M)$ denote the largest integer $k$ such that $M$ contains a \( k \)-noncrossing.
	We also let $\mcr(M)$ denote the number of maximal crossings in $M\in\M_n$. 
	Theorems~\ref{thm-P1-nara}--\ref{thm-P1-P2-joint} below will be proved in Section~\ref{proofs-three-theorems}. 
	Theorem~\ref{thm-P1-P2-joint} generalizes Theorems~\ref{thm-P1-nara} and~\ref{thm-ballot}. 
	We refer to Sections~\ref{sec-bijection} and~\ref{stats-St-matchings-subsec}, respectively, for the definitions of the sets and statistics in the theorems.
	
	\begin{thm}\label{thm-P1-nara}
		We have
		\begin{align}\label{Narayana-sta-P1-P2}
			& \sum_{M \in \M_n(P_1)} x^{\mbox{\sf{\footnotesize mcr}}(M)}= \sum_{M \in \M_n(P_2)} x^{\mbox{\sf{\footnotesize mcr}}(M)} = \sum_{M \in \M_n(P_2)} x^{\mbox{\sf{\footnotesize nr}}(M)} 
			=\sum_{P \in \mathcal{P}_n\mathbf{(3+1)}} x^{\mbox{\sf{\footnotesize mag}}(P)} \notag \\
			= & \sum_{P \in \mathcal{P}_n\mathbf{(N)}} x^{\mbox{\sf{\footnotesize mag}}(P)}  = \sum_{P \in \mathcal{P}_n\mathbf{(N)}} x^{\mbox{\sf{\footnotesize h}}(P)}
			= \sum_{\alpha \in \A_n(101)} x^{\mbox{\sf{\footnotesize lmax}}(\alpha)} = \sum_{\pi \in \F_n(3142)} x^{\mbox{\sf{\footnotesize rmin}}(\pi)} =N_n(x).
		\end{align}
	\end{thm}
	
	A matching $M \in \M_n$ is said to be \textit{reducible} if there exists an integer $i$ with $1\le i<n$ such that the set $\{1,2,\dots,2i\}$ contains as many openers as closers; otherwise we say $M$ is \textit{irreducible}.
	Each matching can be decomposed into \textit{irreducible blocks}.
	Let $\bl(M)$ be the number of irreducible blocks of $M$. We also let $\fcr(M)$ denote the size of the \textit{first maximal crossing}; that is, $\fcr(M)$ is the largest possible $k$ such that there exists a $k$-crossing starting with $1$. The $\bl$ and $\fcr$ statistics have a symmetric distribution over $\M_n(P_1)$ and $\M_n(P_2)$. 
	
	\begin{thm}\label{thm-ballot}
		We have the following symmetric distribution in $y$ and $z$:
		\begin{align}\label{formula-fcr-block-C(x,t)}
			&\sum_{M \in \M(P_1)}\, y^{\mbox{\sf{\footnotesize fcr}}(M)} z^{\mbox{\sf{\footnotesize bl}}(M)} t^{|M|}
			= \sum_{M \in \M(P_2)}\,y^{\mbox{\sf{\footnotesize fcr}}(M)} z^{\mbox{\sf{\footnotesize bl}}(M)}t^{|M|}
			=\sum_{P \in \mathcal{P}{\bf (3+1)}}  y^{\mbox{\sf{\footnotesize min}}(P)} z^{\mbox{\sf{\footnotesize ssd}}(P)} t^{|P|}  \notag\\
			=&\sum_{P \in \mathcal{P}{\bf (N)}}  y^{\mbox{\sf{\footnotesize min}}(P)} z^{\mbox{\sf{\footnotesize smc}}(P)} t^{|P|}  
			=\sum_{\alpha \in \A(101)} y^{\mbox{\sf{\footnotesize zero}}(\alpha)} z^{\mbox{\sf{\footnotesize rmin}}(\alpha)} t^{|\alpha|} 
			=\sum_{\pi \in \F(3142)} y^{\mbox{\sf{\footnotesize idr}}(\pi)} z^{\mbox{\sf{\footnotesize lmax}}(\pi)} t^{|\pi|} \notag \\
			=&  1 + \frac{4  y z t}{\left(2  - y \left(1 - \sqrt{1 - 4 t}\right) \right) \left(2  - z \left(1 - \sqrt{1 - 4 t}\right) \right)}  
			=  1+yzt\,C(y,t) C(z,t),  
		\end{align}
		where $C(x,t) = \frac{1}{1 - x t C(t)}$ is the g.f.\ of the {\it ballot numbers}~{\rm \cite{Barcucci-Verri}}, recorded as entry {\rm A009766} in~{\rm \cite{OEIS}}.
	\end{thm}
	
	\begin{thm}\label{thm-P1-P2-joint}
		We have 
		\begin{small}
			\begin{align}
				&  \sum_{M \in \M(P_1)} x^{\mbox{\sf{\footnotesize mcr}}(M)} y^{\mbox{\sf{\footnotesize fcr}}(M)} z^{\mbox{\sf{\footnotesize bl}}(M)} t^{|M|}= \sum_{M \in \M(P_2)} x^{\mbox{\sf{\footnotesize mcr}}(M)} y^{\mbox{\sf{\footnotesize fcr}}(M)} z^{\mbox{\sf{\footnotesize bl}}(M)} t^{|M|}\notag\\
				=    & \sum_{P \in \mathcal{P}{\bf (3+1)}}  x^{\mbox{\sf{\footnotesize mag}}(P)} y^{\mbox{\sf{\footnotesize min}}(P)} z^{\mbox{\sf{\footnotesize ssd}}(P)} t^{|P|}  =  \sum_{P \in \mathcal{P}{\bf (N)}}  x^{\mbox{\sf{\footnotesize mag}}(P)} y^{\mbox{\sf{\footnotesize min}}(P)} z^{\mbox{\sf{\footnotesize smc}}(P)} t^{|P|}   \notag \\
				=   & \sum_{\alpha \in \A(101)} x^{\mbox{\sf{\footnotesize lmax}}(\alpha)}  y^{\mbox{\sf{\footnotesize zero}}(\alpha)} z^{\mbox{\sf{\footnotesize rmin}}(\alpha)} t^{|\alpha|} 
				=\sum_{\pi \in \F(3142)} x^{\mbox{\sf{\footnotesize rmin}}(\pi)} y^{\mbox{\sf{\footnotesize idr}}(\pi)} z^{\mbox{\sf{\footnotesize lmax}}(\pi)}t^{|\pi|} \notag \\
				=&  1 + \frac{x yzt \left( 1-y +t (x + y + x y-1) + (1 - y) \sqrt{1 + t^2 (1 - x)^2 - 2 t (1 + x)} \right) }{\left(1 + y (xt+yt -t -1)\right) \left(1 + t (x - 2 x z-1 ) + \sqrt{1 + t^2 (1 - x)^2 - 2 t (1 + x)}\right)} \label{main-formulas-dist} \\
				=  & 1+ \frac{xyzt}{(1- ytN(x,t)) \left(1- zt(N(x,t)+x-1) \right)}.\label{fomular-sta-P1-P2-joint}
			\end{align}
		\end{small}
	\end{thm}

	\section{Enumerative results and Wilf-equivalences}\label{sec-matching-Catalan}
	
	The enumerative results for patterns $P_1$ and $P_2$ are presented in Sections~\ref{P1-subsec} and~\ref{P2-sec}, respectively, and the proofs of Theorems~\ref{thm-Catalan} and~\ref{thm1-general} are given in Section~\ref{sub-genete}.
	
	\subsection{Pattern $P_1$}\label{P1-subsec}
	
	As explained above, nonnesting matchings on $\{1, 2, \ldots, 2n\}$, which are counted by the $n$-th Catalan number (see~\cite[Exercise~6.19]{Stanley}), form a subset of Stoimenow matchings. This observation resolves the problem raised by Bevan et al.~\cite{Bevan-Cheon-Kitaev}. The following lemma provides an alternative description of the set of nonnesting matchings in terms of pattern avoidance.
	
	\begin{lem}\label{lem-Nonnesting}
		The set $\M_n(P_1)$ coincides with the set of nonnesting matchings on $\{1, 2, \ldots, 2n\}$ for any $n \geq 0$.
		Thus, $|\mathcal{M}_n(P_1)| = C_n$. 
	\end{lem}
	\begin{proof}
		Suppose a matching $M \in \mathcal{M}_n$ contains a pair of nesting arcs $[a,b]$ and $[c,d]$ such that $a < c < d < b$. 
		We choose such an arc $[a,b]$ with the maximal possible $a$ among all such arcs $[a,b]$ and  the minimal possible $c$ among all such arcs $[c,d]$. 
		Since $M$ avoids Type I arcs, $M$ must have an arc $[a^-,a^+]$ such that $a^-<a<a^+<c$. 
		
		We shall prove that $M$ contains an occurrence of  pattern $P_1$. 
		Assume that all elements between $d$ and $b$ are closers.
		Indeed, starting with $b-1$ and $b$, if no Type II configuration occurs there, then we consider $b-2$ and $b-1$; if none occurs there either, then we consider $b-3$ and $b-2$; and so on, eventually reaching $d$ and $d+1$.
		Thus, there must be at least one opener between $d$ and $b$, and we let $[b^-, x]$ be the arc satisfying $d < b^- < b$ with $x$ being the rightmost possible closer. Then $x > b$ must hold; otherwise, the arcs $[a,b]$ and $[b^-,x]$ would form a Type II configuration. 
		But then the matching induced by the arcs $[a^-,a^+]$, $[a,b]$, $[c,d]$, and $[b^-,x]$ forms an occurrence of the pattern $P_1$.
		
		Conversely, a matching $M \in \M_n$ containing an occurrence of pattern $P_1$ must have nesting arcs. This completes the proof.\end{proof}
	
	\subsection{Pattern $P_2$}\label{P2-sec}
	Following~\cite{Claesson-Dukes-Kitaev}, the \textit{reduction arc} of a matching $M \in \mathcal{M}_n$, denoted $\red(M)$, is the arc whose closer is located immediately to the right of the final opener in $M$. For instance, let $M_1 = \{[1,3], [2,4], [5,6], [7,8]\}$ and $M_2 = \{[1,2], [3,5], [4,7], [6,8]\}$, as depicted in the left and right parts of the diagram below, respectively. Then, $\red(M_1) = [7,8]$ and $\red(M_2) = [4,7]$.\\[-0.7cm]

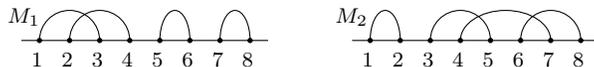
	\begin{center}
		\begin{tikzpicture}[scale = 0.4]
			\draw (-0.6,0) -- (7.6,0);
			\foreach \x in {0,1,2,3,4,5,6,7}
			{
				\filldraw (\x,0) circle (2pt);
			}
			\draw[black] (0,0) arc (180:0:1 and 1);
			\draw[black] (1,0) arc (180:0:1 and 1);
			\draw[black] (4,0) arc (180:0:0.5 and 1);
			\draw[black] (6,0) arc (180:0:0.5 and 1);
			\node[anchor=north west] at (-1.4, 1.4) {\scriptsize{$M_1$}};
			\node[anchor=north west] at (-0.6, -0.05) {\scriptsize{1}};
			\node[anchor=north west] at (0.4, -0.05) {\scriptsize{2}};
			\node[anchor=north west] at (1.4, -0.05) {\scriptsize{3}};
			\node[anchor=north west] at (2.4, -0.05) {\scriptsize{4}};
			\node[anchor=north west] at (3.4, -0.05) {\scriptsize{5}};
			\node[anchor=north west] at (4.4, -0.05) {\scriptsize{6}};
			\node[anchor=north west] at (5.4, -0.05) {\scriptsize{7}};
			\node[anchor=north west] at (6.4, -0.05) {\scriptsize{8}};
		\end{tikzpicture}
		\hspace{0.5cm}
		\begin{tikzpicture}[scale = 0.4]
			\draw (-0.6,0) -- (7.6,0);
			\foreach \x in {0,1,2,3,4,5,6,7}
			{
				\filldraw (\x,0) circle (2pt);
			}
			\draw[black] (0,0) arc (180:0:0.5 and 1);
			\draw[black] (2,0) arc (180:0:1 and 1);
			\draw[black] (3,0) arc (180:0:1.5 and 1);
			\draw[black] (5,0) arc (180:0:1 and 1);
			\node[anchor=north west] at (-1.5, 1.4) {\scriptsize{$M_2$}};
			\node[anchor=north west] at (-0.6, -0.05) {\scriptsize{1}};
			\node[anchor=north west] at (0.4, -0.05) {\scriptsize{2}};
			\node[anchor=north west] at (1.4, -0.05) {\scriptsize{3}};
			\node[anchor=north west] at (2.4, -0.05) {\scriptsize{4}};
			\node[anchor=north west] at (3.4, -0.05) {\scriptsize{5}};
			\node[anchor=north west] at (4.4, -0.05) {\scriptsize{6}};
			\node[anchor=north west] at (5.4, -0.05) {\scriptsize{7}};
			\node[anchor=north west] at (6.4, -0.05) {\scriptsize{8}};
		\end{tikzpicture}
	\end{center}
	\vspace{-0.5cm}
	
	\begin{lem}\label{lem-P2-catalan}
		We have $|\M_n(P_2)|=C_n$.
	\end{lem}
	\begin{proof}
		To prove that $|\M_n(P_2)|=C_n$, since $|\M_0(P_2)|=C_0=1$, it suffices to show that $|\M_n(P_2)|$ satisfies the same recurrence relation as $C_n$, namely,
		\begin{align}\label{P2-recurr}
			|\M_{n}(P_2)|=\sum_{k=1}^{n}|\M_{k-1}(P_2)| \cdot  |\M_{n-k}(P_2)|.
		\end{align}
		We next show how to generate all \( P_2 \)-avoiding Stoimenow matchings of length \( n \). Given any pair \( M' =\{[1, b_1], [a_2, b_2],\ldots,[a_{k-1},2k-2]\}\in \mathcal{M}_{k-1}(P_2) \) and \( M'' \in \mathcal{M}_{n-k}(P_2) \), where \( 1 \leq k \leq n \), we introduce a \textit{gluing procedure} that combines \( M' \) and \( M'' \) into a matching \( M \in \mathcal{M}_{n}(P_2) \), and a \textit{splitting procedure} that recovers \( M' \) and \( M'' \) from \( M \), thereby proving~\eqref{P2-recurr}.
		
		{\bf Gluing procedure.} We first obtain a matching $\widetilde{M}$ from $M'$ as follows.  
		If $M' = \emptyset$, then let $\widetilde{M}$ be the single-arc matching $[1,2]$.  
		Otherwise, we construct $\widetilde{M}$ by adding a new reduction arc $[a,b]$ to $M'$, where $b$ is inserted immediately after $a_{k-1}$.  
		We divide the proof into two cases, according to the position of the opener~$a$:\\[-4mm]
		
		\noindent
		\textbf{Case 1. }
		If the arcs $[1,b_1]$ and $[a_{k-1},2k-2]$ in $M'$ overlap (i.e., $a_{k-1} < b_1$), then to avoid Type~I and Type II arcs, $M'$ must be a $(k-1)$-crossing. We  place the element \( a \) immediately to the left of the first opener in \( M' \) to obtain \( \widetilde{M} \) (making $a=1$ in this case), as shown in Figure~\ref{fig-glue-1}.
		
		\begin{figure}[h]
			\centering
			\begin{tikzpicture}[scale = 0.4]
				\draw (-0.6,0) -- (1,0);
				\draw[dashed] (1,0) -- (3,0);
				\draw (3,0) -- (5,0);
				\draw[dashed] (5,0) -- (7.6,0);
				\foreach \x in {0,1,3,4,5,7} {
					\filldraw[black] (\x,0) circle (2pt);
				}
				
				\draw[black] (0,0) arc (180:0:2 and 1);
				\draw[black] (1,0) arc (180:0:2 and 1 );
				\draw[black] (3,0) arc (180:0:2 and 1 );
				\node[anchor=north west] at (-1, 1.8) {\scriptsize{$M'$}};
				\node[anchor=north west] at (-0.6, -0.05) {\scriptsize{1}};
				\node[anchor=north west] at (2.2, -0.05) {\scriptsize{$a_{k-1}$}};
				\node[anchor=north west] at (6, -0.05) {\scriptsize{$2k-2$}};
				\node[anchor=north west] at (1.5, 0.8) {\scriptsize{$\cdots$}};
			\end{tikzpicture}
			\hspace{0.1cm}
			\raisebox{3ex}{$\xrightarrow{\hspace*{0.2cm}}$}  
			\hspace{0.1cm}
			\begin{tikzpicture}[scale = 0.4]
				\draw (-0.6,0) -- (2,0);
				\draw[dashed] (2,0) -- (4,0);
				\draw (4,0) -- (7,0);
				\draw[dashed] (7,0) -- (9.6,0);
				\foreach \x in {1,2,4,6,7,9} {
					\filldraw[black] (\x,0) circle (2pt);
				}
				\draw[thick] (0,0) circle (5pt);
				\draw[thick] (0,0) circle (1pt);
				\draw[thick] (5,0) circle (5pt);
				\draw[thick] (5,0) circle (1pt);
				\draw[black, very thick] (0,0) arc (180:0:2.5 and 1.3);
				\draw[black] (4,0) arc (180:0:2.5 and 1);
				\draw[black] (1,0) arc (180:0:2.5 and 1 );
				\draw[black] (2,0) arc (180:0:2.5 and 1 );
				\node[anchor=north west] at (-1, 1.8) {\scriptsize $\widetilde{M}$};
				\node[anchor=north west] at (-0.6, -0.05) {\scriptsize{$a$}};
				\node[anchor=north west] at (0.5, -0.05) {\scriptsize{$2$}};
				\node[anchor=north west] at (2.8, -0.05) {\scriptsize{$a'_{k-1}$}};
				\node[anchor=north west] at (4.6, -0.05) {\scriptsize{$b$}};
				\node[anchor=north west] at (8.1, -0.05) {\scriptsize{$2k$}};
				\node[anchor=north west] at (2.5, 0.8) {\scriptsize{$\cdots$}};
			\end{tikzpicture}
			\caption{Method for obtaining $\widetilde{M}$ from $M'$ in Case~1.}
			\label{fig-glue-1}
		\end{figure}
		
		\noindent
		\textbf{Case 2. }
		If the arcs $[1,b_1]$ and $[a_{k-1},2k-2]$ do not overlap in $M'$, we let \( a_i, a_{i+1}, \ldots, b_1 - 1 \) be all openers less than \( b_1 \) whose corresponding closers are greater than \( a_{k-1} \). These openers, if any, form a contiguous run of openers \( X \) in order to avoid Type~I and Type~II arcs. We place the element \( a \) immediately before \( X \) (resp., \( b_1 \)) if \( X \) is nonempty (resp., empty) to obtain $\widetilde{M}$; see Figure~\ref{fig-glue-2} for an illustration. 
		
		\begin{figure}[h]
			\centering
			\begin{tikzpicture}[scale = 0.3] 
				\draw (-0.4,0) -- (1,0);
				\draw [dashed](1,0) -- (3,0);
				\draw (3,0) -- (4,0);
				\draw[dashed] (4,0) -- (6,0);
				\draw (6,0) -- (8,0);
				\draw[dashed] (8,0) -- (11,0);
				\draw(11,0) -- (12,0);
				\draw[dashed] (12,0) -- (14,0);
				\draw (14,0) -- (15,0);
				\draw[dashed] (15,0) -- (17,0);
				\draw (17,0) -- (18,0);
				\draw[dashed] (18,0) -- (20.4,0);
				\foreach \x in {0,1,3,4,6,7,8,11,12,14,15,17,18,20}
				{
					\filldraw (\x,0) circle (2pt);
				}
				\draw[black] (0,0) arc (180:0:3.5 and 1.5);
				\draw[black] (1,0) arc (180:0:3.5 and 1.5);
				\draw[black] (3,0) arc (180:0:4 and 1.5);
				\draw[black] (4,0) arc (180:0:5.5 and 1.5);
				\draw[black] (6,0) arc (180:0:5.5 and 1.5);
				\draw[black] (12,0) arc (180:0:3 and 1.5);
				\draw[black] (14,0) arc (180:0:3 and 1.5);
				\node[anchor=north west] at (-0.3, 2.8) {\scriptsize $M'$};
				\node[anchor=north west][rotate=90] at (3.3, -1.4) {$\Big{\{}$};
				\node[anchor=north west] at (4.2, -0.5) {\scriptsize{$X$}};
				\node[anchor=north west] at (-0.7, -0.05) {\scriptsize{1}};
				\node[anchor=north west] at (6.4, -0.05) {\scriptsize{$b_1$}};
				\node[anchor=north west] at (11.3, -0.05) {\scriptsize{$a_j$}};
				\node[anchor=north west] at (13, -0.05) {\scriptsize{$a_{k-1}$}};
				\node[anchor=north west] at (17.3, -0.05) {\scriptsize{$b_j$}};
				\node[anchor=north west] at (18.7, -0.05) {\scriptsize{$2k-2$}};
				\node[anchor=north west] at (8, 1.5) {{$\cdots$}};
			\end{tikzpicture}
			\hspace{0.1cm}
			\raisebox{4ex}{$\xrightarrow{\hspace*{0.2cm}}$}  
			\hspace{0.1cm}
			\begin{tikzpicture}[scale = 0.3] 
				\draw (-0.4,0) -- (1,0);
				\draw [dashed](1,0) -- (3,0);
				\draw (3,0) -- (5,0);
				\draw[dashed] (5,0) -- (7,0);
				\draw (7,0) -- (9,0);
				\draw[dashed] (9,0) -- (12,0);
				\draw(12,0) -- (13,0);
				\draw[dashed] (13,0) -- (15,0);
				\draw (15,0) -- (17,0);
				\draw[dashed] (17,0) -- (19,0);
				\draw (19,0) -- (20,0);
				\draw[dashed] (20,0) -- (22.4,0);
				\foreach \x in {0,1,3,5,7,8,9,12,13,15,17,19,20,22}
				{
					\filldraw (\x,0) circle (2pt);
				}
				\draw[thick] (4,0) circle (6pt);
				\draw[thick] (4,0) circle (1.5pt);
				\draw[thick] (16,0) circle (6pt);
				\draw[thick] (16,0) circle (1.5pt);
				\draw[black] (0,0) arc (180:0:4 and 1.5);
				\draw[black] (1,0) arc (180:0:4 and 1.5);
				\draw[black] (3,0) arc (180:0:4.5 and 1.5);
				\draw[black] (5,0) arc (180:0:6 and 1.5);
				\draw[black] (7,0) arc (180:0:6 and 1.5);
				\draw[black] (13,0) arc (180:0:3.5 and 1.5);
				\draw[black] (15,0) arc (180:0:3.5 and 1.5);
				\draw[black, very thick] (4,0) arc (180:0:6 and 2.3);
				\node[anchor=north west] at (-0.3, 2.8) {\scriptsize $\widetilde{M}$};
				\node[anchor=north west][rotate=90] at (4.3, -1.4) {$\Big{\{}$};
				\node[anchor=north west] at (5.2, -0.5) {\scriptsize{$X$}};
				\node[anchor=north west] at (-0.7, -0.05) {\scriptsize{1}};
				\node[anchor=north west] at (7.4, -0.05) {\scriptsize{$b'_1$}};
				\node[anchor=north west] at (12.1, -0.05) {\scriptsize{$a'_j$}};
				\node[anchor=north west] at (13.4, -0.05) {\scriptsize{$a'_{k-1}$}};
				\node[anchor=north west] at (19.3, -0.05) {\scriptsize{$b'_j$}};
				\node[anchor=north west] at (20.9, -0.05) {\scriptsize{$2k$}};
				\node[anchor=north west] at (3.3, -0.05) {\scriptsize{$a$}};
				\node[anchor=north west] at (15.7, -0.05) {\scriptsize{$b$}};
				\node[anchor=north west] at (9, 1.5) {{$\cdots$}};
			\end{tikzpicture}
			\caption{Method for obtaining $\widetilde{M}$ from $M'$ in Case~2.}
			\label{fig-glue-2}
		\end{figure}
		
		It is clear that $\widetilde{M} \in \M_{k}(P_2)$. To glue $M'$ and $M''$, we give the following definition. The \textit{merge} of two matchings $M_1 \in \M_k$ and $M_2 \in \M_\ell$, denoted \( M_1 \oplus M_2 \), is a matching obtained by merging \( M_1 \) and \( M_2 \) after increasing each element in \( M_2 \) by \( 2k \). For example, if \( M_1 = \{[1,3], [2,4]\} \) and \( M_2 = \{[1,2], [3,4]\} \), then \( M_1 \oplus M_2 = \{[1,3], [2,4], [5,6], [7,8]\} \).
		
		Finally, we let \( M = \widetilde{M} \oplus M''\). It is clear that $M \in \M_{n}(P_2)$ and that it is unique for the fixed $M'$ and $M''$. 
		We illustrate this procedure with two examples in Figure~\ref{fig-exa-glue}.
		
		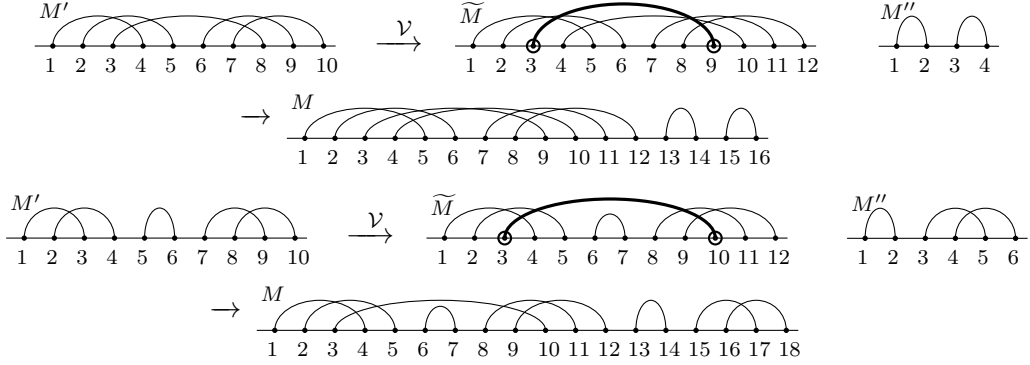
\begin{figure}[h]
			\centering
			\begin{tikzpicture}[scale = 0.4] 
				\draw (-0.6,0) -- (9.4,0);
				\foreach \x in {0,1,2,3,4,5,6,7,8,9}
				{
					\filldraw (\x,0) circle (2pt);
				}
				\draw[black] (0,0) arc (180:0:1.5 and 1);
				\draw[black] (1,0) arc (180:0:1.5 and 1);
				\draw[black] (2,0) arc (180:0:2.5 and 1);
				\draw[black] (5,0) arc (180:0:1.5 and 1);
				\draw[black] (6,0) arc (180:0:1.5 and 1);
				
				\node[anchor=north west] at (-0.8, 1.8) {\scriptsize $M'$};
				\node[anchor=north west] at (-0.6, -0.05) {\scriptsize{1}};
				\node[anchor=north west] at (0.4, -0.05) {\scriptsize{2}};
				\node[anchor=north west] at (1.4, -0.05) {\scriptsize{3}};
				\node[anchor=north west] at (2.4, -0.05) {\scriptsize{4}};
				\node[anchor=north west] at (3.4, -0.05) {\scriptsize{5}};
				\node[anchor=north west] at (4.4, -0.05) {\scriptsize{6}};
				\node[anchor=north west] at (5.4, -0.05) {\scriptsize{7}};
				\node[anchor=north west] at (6.4, -0.05) {\scriptsize{8}};
				\node[anchor=north west] at (7.4, -0.05) {\scriptsize{9}};
				\node[anchor=north west] at (8.4, -0.05) {\scriptsize{10}};
			\end{tikzpicture}
			\hspace{0.1cm}
			\raisebox{2.5ex}{$\xrightarrow{\hspace*{0.2cm} \V}$}
			\hspace{0.1cm}
			\begin{tikzpicture}[scale = 0.4] 
				\draw (-0.6,0) -- (11.4,0);
				\foreach \x in {0,1,3,4,5,6,7,9,10,11}
				{
					\filldraw (\x,0) circle (2pt);
				}
				\draw[thick] (2,0) circle (6pt);
				\draw[thick] (2,0) circle (1.5pt);
				\draw[thick] (8,0) circle (6pt);
				\draw[thick] (8,0) circle (1.5pt);
				\draw[black] (0,0) arc (180:0:2 and 1);
				\draw[black] (1,0) arc (180:0:2 and 1);
				\draw[black, very thick] (2,0) arc (180:0:3 and 1.4);
				\draw[black] (3,0) arc (180:0:3 and 1);
				\draw[black] (6,0) arc (180:0:2 and 1);
				\draw[black] (7,0) arc (180:0:2 and 1);  
				
				\node[anchor=north west] at (-0.8,1.8) {\scriptsize $\widetilde{M}$};        
				\node[anchor=north west] at (-0.6, -0.05) {\scriptsize{1}};
				\node[anchor=north west] at (0.4, -0.05) {\scriptsize{2}};
				\node[anchor=north west] at (1.4, -0.05) {\scriptsize{3}};
				\node[anchor=north west] at (2.4, -0.05) {\scriptsize{4}};
				\node[anchor=north west] at (3.4, -0.05) {\scriptsize{5}};
				\node[anchor=north west] at (4.4, -0.05) {\scriptsize{6}};
				\node[anchor=north west] at (5.4, -0.05) {\scriptsize{7}};
				\node[anchor=north west] at (6.4, -0.05) {\scriptsize{8}};
				\node[anchor=north west] at (7.4, -0.05) {\scriptsize{9}};
				\node[anchor=north west] at (8.4, -0.05) {\scriptsize{10}};
				\node[anchor=north west] at (9.4, -0.05) {\scriptsize{11}};
				\node[anchor=north west] at (10.4, -0.05) {\scriptsize{12}};
			\end{tikzpicture}
			\hspace{0.3cm}
			\begin{tikzpicture}[scale = 0.4] 
				\draw (-0.6,0) -- (3.4,0);
				\foreach \x in {0,1,2,3}
				{
					\filldraw (\x,0) circle (2pt);
				}
				\draw[black] (0,0) arc (180:0:0.5 and 1);
				\draw[black] (2,0) arc (180:0:0.5 and 1);
				
				\node[anchor=north west] at (-0.9, 1.8) {\scriptsize $M''$};
				\node[anchor=north west] at (-0.6, -0.05) {\scriptsize{1}};
				\node[anchor=north west] at (0.4, -0.05) {\scriptsize{2}};
				\node[anchor=north west] at (1.4, -0.05) {\scriptsize{3}};
				\node[anchor=north west] at (2.4, -0.05) {\scriptsize{4}};
			\end{tikzpicture}
			\\
			\raisebox{4ex}{$\xrightarrow{\hspace*{0.2cm}}$}  
			\vspace{0.4cm}
			\begin{tikzpicture}[scale = 0.4, shift={(0,1)}] 
				\draw (-0.6,0) -- (15.4,0);
				\foreach \x in {0,1,2,3,4,5,6,7,8,9,10,11,12,13,14,15}
				{
					\filldraw (\x,0) circle (2pt);
				}
				\draw[black] (0,0) arc (180:0:2 and 1);
				\draw[black] (1,0) arc (180:0:2 and 1);
				\draw[black] (2,0) arc (180:0:3 and 1);
				\draw[black] (3,0) arc (180:0:3 and 1);
				\draw[black] (6,0) arc (180:0:2 and 1);
				\draw[black] (7,0) arc (180:0:2 and 1);
				\draw[black] (12,0) arc (180:0:0.5 and 1);
				\draw[black] (14,0) arc (180:0:0.5 and 1);
				
				\node[anchor=north west] at (-0.8,1.8) {\scriptsize $M$};
				\node[anchor=north west] at (-0.6, -0.05) {\scriptsize{1}};
				\node[anchor=north west] at (0.4, -0.05) {\scriptsize{2}};
				\node[anchor=north west] at (1.4, -0.05) {\scriptsize{3}};
				\node[anchor=north west] at (2.4, -0.05) {\scriptsize{4}};
				\node[anchor=north west] at (3.4, -0.05) {\scriptsize{5}};
				\node[anchor=north west] at (4.4, -0.05) {\scriptsize{6}};
				\node[anchor=north west] at (5.4, -0.05) {\scriptsize{7}};
				\node[anchor=north west] at (6.4, -0.05) {\scriptsize{8}};
				\node[anchor=north west] at (7.4, -0.05) {\scriptsize{9}};
				\node[anchor=north west] at (8.4, -0.05) {\scriptsize{10}};
				\node[anchor=north west] at (9.4, -0.05) {\scriptsize{11}};
				\node[anchor=north west] at (10.4, -0.05) {\scriptsize{12}};
				\node[anchor=north west] at (11.4, -0.05) {\scriptsize{13}};
				\node[anchor=north west] at (12.4, -0.05) {\scriptsize{14}};
				\node[anchor=north west] at (13.4, -0.05) {\scriptsize{15}};
				\node[anchor=north west] at (14.4, -0.05) {\scriptsize{16}};
			\end{tikzpicture}
			\vspace{-0.3cm}
			
			\centering
			\begin{tikzpicture}[scale = 0.4] 
				\draw (-0.6,0) -- (9.4,0);
				\foreach \x in {0,1,2,3,4,5,6,7,8,9}
				{
					\filldraw (\x,0) circle (2pt);
				}
				\draw[black] (0,0) arc (180:0:1 and 1);
				\draw[black] (1,0) arc (180:0:1 and 1);
				\draw[black] (4,0) arc (180:0:0.5 and 1);
				\draw[black] (6,0) arc (180:0:1 and 1);
				\draw[black] (7,0) arc (180:0:1 and 1);
				
				\node[anchor=north west] at (-0.8, 1.8) {\scriptsize $M'$};
				\node[anchor=north west] at (-0.6, -0.05) {\scriptsize{1}};
				\node[anchor=north west] at (0.4, -0.05) {\scriptsize{2}};
				\node[anchor=north west] at (1.4, -0.05) {\scriptsize{3}};
				\node[anchor=north west] at (2.4, -0.05) {\scriptsize{4}};
				\node[anchor=north west] at (3.4, -0.05) {\scriptsize{5}};
				\node[anchor=north west] at (4.4, -0.05) {\scriptsize{6}};
				\node[anchor=north west] at (5.4, -0.05) {\scriptsize{7}};
				\node[anchor=north west] at (6.4, -0.05) {\scriptsize{8}};
				\node[anchor=north west] at (7.4, -0.05) {\scriptsize{9}};
				\node[anchor=north west] at (8.4, -0.05) {\scriptsize{10}};
			\end{tikzpicture}
			\hspace{0.1cm}
			\raisebox{2.5ex}{$\xrightarrow{\hspace*{0.2cm} \V}$}
			\hspace{0.1cm}
			\begin{tikzpicture}[scale = 0.4] 
				\draw (-0.6,0) -- (11.4,0);
				\foreach \x in {0,1,3,4,5,6,7,8,10,11}
				{
					\filldraw (\x,0) circle (2pt);
				}
				\draw[thick] (2,0) circle (6pt);
				\draw[thick] (2,0) circle (1.5pt);
				\draw[thick] (9,0) circle (6pt);
				\draw[thick] (9,0) circle (1.5pt);
				\draw[black] (0,0) arc (180:0:1.5 and 1);
				\draw[black] (1,0) arc (180:0:1.5 and 1);
				\draw[black,  very thick] (2,0) arc (180:0:3.5 and 1.3);
				\draw[black] (5,0) arc (180:0:0.5 and 0.8);
				\draw[black] (7,0) arc (180:0:1.5 and 1);
				\draw[black] (8,0) arc (180:0:1.5 and 1);  
				
				\node[anchor=north west] at (-0.8,1.8) {\scriptsize $\widetilde{M}$};        
				\node[anchor=north west] at (-0.6, -0.05) {\scriptsize{1}};
				\node[anchor=north west] at (0.4, -0.05) {\scriptsize{2}};
				\node[anchor=north west] at (1.4, -0.05) {\scriptsize{3}};
				\node[anchor=north west] at (2.4, -0.05) {\scriptsize{4}};
				\node[anchor=north west] at (3.4, -0.05) {\scriptsize{5}};
				\node[anchor=north west] at (4.4, -0.05) {\scriptsize{6}};
				\node[anchor=north west] at (5.4, -0.05) {\scriptsize{7}};
				\node[anchor=north west] at (6.4, -0.05) {\scriptsize{8}};
				\node[anchor=north west] at (7.4, -0.05) {\scriptsize{9}};
				\node[anchor=north west] at (8.4, -0.05) {\scriptsize{10}};
				\node[anchor=north west] at (9.4, -0.05) {\scriptsize{11}};
				\node[anchor=north west] at (10.4, -0.05) {\scriptsize{12}};
			\end{tikzpicture}
			\hspace{0.3cm}
			\begin{tikzpicture}[scale = 0.4] 
				\draw (-0.6,0) -- (5.4,0);
				\foreach \x in {0,1,2,3,4,5}
				{
					\filldraw (\x,0) circle (2pt);
				}
				\draw[black] (0,0) arc (180:0:0.5 and 1);
				\draw[black] (2,0) arc (180:0:1 and 1);
				\draw[black] (3,0) arc (180:0:1 and 1);
				
				\node[anchor=north west] at (-0.8, 1.8) {\scriptsize $M''$};
				\node[anchor=north west] at (-0.6, -0.05) {\scriptsize{1}};
				\node[anchor=north west] at (0.4, -0.05) {\scriptsize{2}};
				\node[anchor=north west] at (1.4, -0.05) {\scriptsize{3}};
				\node[anchor=north west] at (2.4, -0.05) {\scriptsize{4}};
				\node[anchor=north west] at (3.4, -0.05) {\scriptsize{5}};
				\node[anchor=north west] at (4.4, -0.05) {\scriptsize{6}};
			\end{tikzpicture}
			\\
			\raisebox{4ex}{$\xrightarrow{\hspace*{0.2cm}}$}  
			\vspace{0.4cm}
			\begin{tikzpicture}[scale = 0.4, shift={(0,1)}] 
				\draw (-0.6,0) -- (17.4,0);
				\foreach \x in {0,1,2,3,4,5,6,7,8,9,10,11,12,13,14,15,16,17}
				{
					\filldraw (\x,0) circle (2pt);
				}
				\draw[black] (0,0) arc  (180:0:1.5 and 1);
				\draw[black] (1,0) arc (180:0:1.5 and 1);
				\draw[black] (2,0) arc (180:0:3.5 and 1);
				\draw[black] (5,0) arc (180:0:0.5 and 0.8);
				\draw[black] (7,0) arc (180:0:1.5 and 1);
				\draw[black] (8,0) arc (180:0:1.5 and 1);
				\draw[black] (12,0) arc (180:0:0.5 and 1);
				\draw[black] (14,0) arc (180:0:1 and 1);
				\draw[black] (15,0) arc (180:0:1 and 1);
				
				\node[anchor=north west] at (-0.8,1.8) {\scriptsize $M$};
				\node[anchor=north west] at (-0.6, -0.05) {\scriptsize{1}};
				\node[anchor=north west] at (0.4, -0.05) {\scriptsize{2}};
				\node[anchor=north west] at (1.4, -0.05) {\scriptsize{3}};
				\node[anchor=north west] at (2.4, -0.05) {\scriptsize{4}};
				\node[anchor=north west] at (3.4, -0.05) {\scriptsize{5}};
				\node[anchor=north west] at (4.4, -0.05) {\scriptsize{6}};
				\node[anchor=north west] at (5.4, -0.05) {\scriptsize{7}};
				\node[anchor=north west] at (6.4, -0.05) {\scriptsize{8}};
				\node[anchor=north west] at (7.4, -0.05) {\scriptsize{9}};
				\node[anchor=north west] at (8.4, -0.05) {\scriptsize{10}};
				\node[anchor=north west] at (9.4, -0.05) {\scriptsize{11}};
				\node[anchor=north west] at (10.4, -0.05) {\scriptsize{12}};
				\node[anchor=north west] at (11.4, -0.05) {\scriptsize{13}};
				\node[anchor=north west] at (12.4, -0.05) {\scriptsize{14}};
				\node[anchor=north west] at (13.4, -0.05) {\scriptsize{15}};
				\node[anchor=north west] at (14.4, -0.05) {\scriptsize{16}};
				\node[anchor=north west] at (15.4, -0.05) {\scriptsize{17}};
				\node[anchor=north west] at (16.4, -0.05) {\scriptsize{18}};
			\end{tikzpicture}
			\vspace{-0.5cm}
			\caption{Two examples of gluing $P_2$-avoiding Stoimenow matchings $M'$ and $M''$ to form a $P_2$-avoiding Stoimenow matching $M$.}
			\label{fig-exa-glue}
		\end{figure}
		
		\textbf{Splitting procedure.} 
		For the matching \( M \), let \( \widetilde{M} \) denote its first irreducible block, and \( M'' \) the collection of the relabeled remaining arcs. We then obtain \( M' \) from \( \widetilde{M} \) by removing its reduction arc. The identification of this reduction arc is strictly unique due to the $P_2$-avoiding condition. If $\widetilde{M}$ contained multiple candidate reduction arcs that could be removed to yield a valid matching, their relative positions would inevitably form an occurrence of the pattern $P_2$, contradicting $M \in \M_n(P_2)$. This gives a uniquely determined decomposition of \( M \) into \( M' \) and \( M'' \).
		Clearly, the splitting procedure is the inverse of the gluing procedure (for an illustration just reverse the steps in Figure~\ref{fig-exa-glue}). This completes our proof of (\ref{P2-recurr}).\end{proof} 
	
	Let $\M^{Ir}_n$ denote the set of all irreducible matchings in $\mathcal{M}_n$.  
	Define a mapping $\mathcal{V}: \mathcal{M}_{n-1}(P_2) \to \M^{Ir}_n(P_2)$ by adding a reduction arc to each $M \in \mathcal{M}_{n-1}(P_2)$, as in the gluing procedure used to obtain $\widetilde{M}$ from $M'$ in the proof of Lemma~\ref{lem-P2-catalan}.  
	This mapping is a bijection, and, together with Lemma~\ref{lem-P2-catalan}, it yields the following corollary. Therefore, any $M\in\M_n(P_2)$ can be decomposed as $\V(M')\oplus M''$ for uniquely determined $M'\in\M_{k-1}(P_2)$ and $M'' \in \M_{n-k}(P_2)$.
	
	\begin{coro}
		\label{coro-P2-irre}
		We have $|\M^{Ir}_n(P_2)|=C_{n-1}$ for any $n\geq 1$.
	\end{coro}
	
	\subsection{Proofs of Theorems~\ref{thm-Catalan} and~\ref{thm1-general}}\label{sub-genete}
	
	We first prove Theorem~\ref{thm1-general}  (Theorem~\ref{thm-Catalan} is proved at the end of this subsection). 
	For \(k=1\), we define \(P_3^1\) to be the single arc pattern. Then all four patterns \(P_2^1,P_3^1,P_4^1\), and \(P_5^1\) coincide.
	For \(k=2\), we have
	\(\left|\M_n\left(
	\begin{tikzpicture}[scale=0.2, baseline=-0.5ex] 
		\draw (-0.6,0) -- (3.4,0);
		\foreach \x in {0,1,2,3}
		{
			\filldraw (\x,0) circle (2pt);
		}
		\draw[black] (0,0) arc (180:0:1 and 1);
		\draw[black] (1,0) arc (180:0:1 and 1);
	\end{tikzpicture}
	\right)\right|
	=
	\left|\M_n\left(
	\begin{tikzpicture}[scale=0.2, baseline=-0.5ex] 
		\draw (-0.6,0) -- (3.4,0);
		\foreach \x in {0,1,2,3}
		{
			\filldraw (\x,0) circle (2pt);
		}
		\draw[black] (0,0) arc (180:0:0.5 and 1);
		\draw[black] (2,0) arc (180:0:0.5 and 1);
	\end{tikzpicture}
	\right)\right|=1\), since the two avoidance classes consist only of the \(n\)-noncrossing and the \(n\)-crossing, respectively.
	Hence the following construction is for \(k\ge 3\).
	To prove Theorem~\ref{thm1-general} for \(k\ge 3\), we construct a bijection \(\Phi\).
	
	\textbf{The construction of $\Phi$:} Given $M \in \M_n(P^k_2)$, if $M \in \M_n(P^k_4)$, then we let $\Phi(M)=M$. Otherwise, let the arcs $[a_{i_1}, b_{i_1}], [a_{i_2}, b_{i_2}],$ $\ldots, [a_{i_k}, b_{i_k}]$ represent the first occurrence of the pattern \(P^k_4\) in \(M\), obtained by choosing the leftmost possible \(a_{i_1}\), then the leftmost possible \(a_{i_2}\), and so on. 
	Consider the maximal crossing $[a_{i_2}, b_{i_2}], [a_{i_2}+1, b_{i_2}+1], \ldots, [a_{i_2}+j-1, b_{i_2}+j-1]$ with consecutive openers and closers. Let $X$ be the (possibly empty) set of openers inside the arc $[a_{i_1}, b_{i_1}]$ whose corresponding closers lie to the right of $b_{i_2}$. The map $\Phi(M)$ is defined as follows.
	We move the openers \(a_{i_2}, a_{i_2}+1, \ldots, a_{i_2}+j-1\) to the position immediately preceding the first opener in \(X\) (or \(b_{i_1}\) if \(X = \emptyset\)) as shown in Figure~\ref{fig-M_nP^k_2-M_nP^k_4}.
	\begin{figure}[h]
		\centering
		\begin{tikzpicture}[scale = 0.34] 
			\draw [dashed](-0.5,0) -- (4,0);
			\draw (4,0) -- (5,0);
			\draw[dashed] (5,0) -- (10,0);
			\draw (10,0) -- (11,0);
			\draw[dashed] (11,0) -- (20.8,0);
			\foreach \x in {0.3,2,4,5,11,12,14,15,16,18,19,20}
			{
				\filldraw (\x,0) circle (2pt);
			}
			\draw[thick] (8,0) circle (6pt);
			\draw[thick] (8,0) circle (1.5pt);
			\draw[thick] (10,0) circle (6pt);
			\draw[thick] (10,0) circle (1.5pt);
			\draw[domain=0:90, smooth, variable=\t, shift={(3,0)}] 
			plot ({-cos(\t)}, {sin(\t)});
			
			\draw[domain=0:90, smooth, variable=\t, shift={(5,0)}] 
			plot ({-cos(\t)}, {sin(\t)});

			\draw[domain=0:90, smooth, variable=\t, shift={(16,0)}] 
			plot ({-cos(\t)}, {sin(\t)});
			
			\draw[domain=0:90, smooth, variable=\t, shift={(18,0)}] 
			plot ({cos(\t)}, {sin(\t)});
			\draw[black] (0.3,0) arc (180:0:2.35 and 1.5);
			\draw[black] (8,0) arc (180:0:2 and 1.5);
			\draw[black] (10,0) arc (180:0:2 and 1.5);
			\draw[black] (11,0) arc (180:0:2.5 and 1.5);
			\draw[black] (18,0) arc (180:0:1 and 1.5);
			
			\node[anchor=north west] at (-0.3, 2.3) {\scriptsize{$M$}};
			\node[anchor=north west][rotate=90] at (1.6, -1.2) {$\Big{\{}$};
			\node[anchor=north west] at (2.2, -0.5) {\scriptsize{$X$}};
			\node[anchor=north west] at (-0.7, -0.05) {\tiny{$a_{i_1}$}};
			\node[anchor=north west] at (4.3, -0.05) {\tiny{$b_{i_1}$}};
			\node[anchor=north west] at (7.3, -0.05) {\tiny{$a_{i_2}$}};
			\node[anchor=north west] at (10.3, -0.05) {\tiny{$a_{i_3}$}};
			\node[anchor=north west] at (11.5, -0.05) {\tiny{$b_{i_2}$}};
			\node[anchor=north west] at (15.3, -0.05) {\tiny{$b_{i_3}$}};
			\node[anchor=north west] at (17.3, -0.05) {\tiny{$a_{i_k}$}};
			\node[anchor=north west] at (19.3, -0.05) {\tiny{$b_{i_k}$}};
			\node[anchor=north west] at (5.5, 1.5) {\footnotesize{$\cdots$}};
			\node[anchor=north west] at (16, 1.5) {\footnotesize{$\cdots$}};
		\end{tikzpicture}
		\raisebox{4ex}{$\xrightarrow{}$}
		\begin{tikzpicture}[scale = 0.34] 
			\draw [dashed](-0.5,0) -- (4,0);
			\draw (4,0) -- (5,0);
			\draw[dashed] (5,0) -- (7,0);
			\draw (7,0) -- (8,0);
			\draw[dashed] (8,0) -- (20.8,0);
			\foreach \x in {0.3,5,7,8,11,12,14,15,16,18,19,20}
			{
				\filldraw (\x,0) circle (2pt);
			}
			\draw[thick] (2,0) circle (6pt);
			\draw[thick] (2,0) circle (1.5pt);
			\draw[thick] (4,0) circle (6pt);
			\draw[thick] (4,0) circle (1.5pt);
			\draw[domain=0:90, smooth, variable=\t, shift={(6,0)}] 
			plot ({-cos(\t)}, {sin(\t)});
			
			\draw[domain=0:90, smooth, variable=\t, shift={(8,0)}] 
			plot ({-cos(\t)}, {sin(\t)});

			\draw[domain=0:90, smooth, variable=\t, shift={(16,0)}] 
			plot ({-cos(\t)}, {sin(\t)});
			
			\draw[domain=0:90, smooth, variable=\t, shift={(18,0)}] 
			plot ({cos(\t)}, {sin(\t)});
			\draw[black] (0.3,0) arc (180:0:3.85 and 1.5);
			\draw[black] (2,0) arc (180:0:5 and 1.5);
			\draw[black] (4,0) arc (180:0:5 and 1.5);
			\draw[black] (11,0) arc (180:0:2.5 and 1.5);
			\draw[black] (18,0) arc (180:0:1 and 1.5);
			
			\node[anchor=north west] at (-0.3, 2.3) {\scriptsize{$M^{(1)}$}};
			\node[anchor=north west][rotate=90] at (4.6, -1.1) {$\Big{\{}$};
			\node[anchor=north west] at (5.2, -0.5) {\scriptsize{$X$}};
			\node[anchor=north west] at (-0.7, -0.05) {\tiny{$a_{i_1}$}};
			\node[anchor=north west] at (7.3, -0.05) {\tiny{$b_{i_1}$}};
			\node[anchor=north west] at (1.3, -0.05) {\tiny{$a_{i_2}$}};
			\node[anchor=north west] at (10.1, -0.05) {\tiny{$a_{i_3}$}};
			\node[anchor=north west] at (11.4, -0.05) {\tiny{$b_{i_2}$}};
			\node[anchor=north west] at (15.3, -0.05) {\tiny{$b_{i_3}$}};
			\node[anchor=north west] at (17.3, -0.05) {\tiny{$a_{i_k}$}};
			\node[anchor=north west] at (19.3, -0.05) {\tiny{$b_{i_k}$}};
			\node[anchor=north west] at (8.5, 1.3) {\footnotesize{$\cdots$}};
			\node[anchor=north west] at (16, 1.3) {\footnotesize{$\cdots$}};
		\end{tikzpicture}
		\caption{The first moving procedure in the map $\Phi$.}
		\label{fig-M_nP^k_2-M_nP^k_4}
	\end{figure}
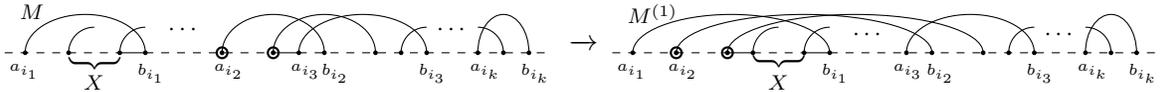
	Iterate the moving operation until no occurrences of \(P^k_4\) remain. This process yields a sequence of matchings \(M = M^{(0)}, M^{(1)}, \dots, M^{(\ell)}\).  We define \(\Phi(M) = M^{(\ell)}\).
	Figure~\ref{fig-exa-Phi} illustrates an example of $\Phi$ mapping $\M_n(P^4_2)$ to $\M_n(P^4_4)$.

	\begin{figure}[ht]
		\centering
		\begin{tikzpicture}[scale = 0.34] 
			\draw (-0.5,0) -- (21.5,0);
			\foreach \x in {0,1,2,5,6,7,8,9,10,11,12,13,14,15,16,17,18,19,20,21}
			{
				\filldraw (\x,0) circle (2pt);
			}
			\draw[thick] (4,0) circle (6pt);
			\draw[thick] (4,0) circle (1.5pt);
			\draw[thick] (3,0) circle (6pt);
			\draw[thick] (3,0) circle (1.5pt);
			\draw[very thick, line width=0.8pt] (0,0) arc (180:0:1 and 1);
			\draw[black] (1,0) arc (180:0:4 and 1.5);
			\draw[very thick, line width=0.8pt] (3,0) arc (180:0:1.5 and 1);
			\draw[black] (4,0) arc (180:0:1.5 and 1);
			\draw[very thick, line width=0.8pt] (5,0) arc (180:0:2.5 and 1);
			\draw[very thick, line width=0.8pt] (8,0) arc (180:0:1.5 and 1);
			\draw[black] (12,0) arc (180:0:1.5 and 1);
			\draw[black] (13,0) arc (180:0:2 and 1);
			\draw[black] (14,0) arc (180:0:3 and 1.5);
			\draw[black] (16,0) arc (180:0:1 and 1);
			\draw[black] (19,0) arc (180:0:1 and 1);
			
			\node[anchor=north west] at (-0.3, 2.3) {\scriptsize{$M$}};
			\node[anchor=north west] at (-0.6, -0.05) {\tiny{1}};
			\node[anchor=north west] at (0.4, -0.05) {\tiny{2}};
			\node[anchor=north west] at (1.4, -0.05) {\tiny{3}};
			\node[anchor=north west] at (2.4, -0.05) {\tiny{4}};
			\node[anchor=north west] at (3.4, -0.05) {\tiny{5}};
			\node[anchor=north west] at (4.4, -0.05) {\tiny{6}};
			\node[anchor=north west] at (5.4, -0.05) {\tiny{7}};
			\node[anchor=north west] at (6.4, -0.05) {\tiny{8}};
			\node[anchor=north west] at (7.4, -0.05) {\tiny{9}};
			\node[anchor=north west] at (8.3, -0.05) {\tiny{10}};
			\node[anchor=north west] at (9.3, -0.05) {\tiny{11}};
			\node[anchor=north west] at (10.3, -0.05) {\tiny{12}};
			\node[anchor=north west] at (11.3, -0.05) {\tiny{13}};
			\node[anchor=north west] at (12.3, -0.05) {\tiny{14}};
			\node[anchor=north west] at (13.3, -0.05) {\tiny{15}};
			\node[anchor=north west] at (14.3, -0.05) {\tiny{16}};
			\node[anchor=north west] at (15.3, -0.05) {\tiny{17}};
			\node[anchor=north west] at (16.3, -0.05) {\tiny{18}};
			\node[anchor=north west] at (17.3, -0.05) {\tiny{19}};
			\node[anchor=north west] at (18.3, -0.05) {\tiny{20}};
			\node[anchor=north west] at (19.3, -0.05) {\tiny{21}};
			\node[anchor=north west] at (20.3, -0.05) {\tiny{22}};
		\end{tikzpicture}
		\hspace{-0.4cm}
		\raisebox{3ex}{$\xrightarrow{}$}
		\hspace{-0.2cm}
		\begin{tikzpicture}[scale = 0.34] 
			\draw (-0.5,0) -- (21.5,0);
			\foreach \x in {0,1,2,3,4,5,6,7,8,9,10,11,13,14,15,16,17,18,19,20,21}
			{
				\filldraw (\x,0) circle (2pt);
			}
			\draw[thick] (12,0) circle (6pt);
			\draw[thick] (12,0) circle (1.5pt);
			\draw[very thick, line width=0.8pt] (0,0) arc (180:0:2 and 1);
			\draw[black] (1,0) arc (180:0:2.5 and 1);
			\draw[black] (2,0) arc (180:0:2.5 and 1);
			\draw[black] (3,0) arc (180:0:3 and 1.5);
			\draw[black] (5,0) arc (180:0:2.5 and 1);
			\draw[black] (8,0) arc (180:0:1.5 and 1);
			\draw[very thick, line width=0.8pt] (12,0) arc (180:0:1.5 and 1);
			\draw[very thick, line width=0.8pt] (13,0) arc (180:0:2 and 1);
			\draw[black] (14,0) arc (180:0:3 and 1.5);
			\draw[very thick, line width=0.8pt] (16,0) arc (180:0:1 and 1);
			\draw[black] (19,0) arc (180:0:1 and 1);
			
			\node[anchor=north west] at (-0.3, 2.3) {\scriptsize{$M^{(1)}$}};
			\node[anchor=north west] at (-0.6, -0.05) {\tiny{1}};
			\node[anchor=north west] at (0.4, -0.05) {\tiny{2}};
			\node[anchor=north west] at (1.4, -0.05) {\tiny{3}};
			\node[anchor=north west] at (2.4, -0.05) {\tiny{4}};
			\node[anchor=north west] at (3.4, -0.05) {\tiny{5}};
			\node[anchor=north west] at (4.4, -0.05) {\tiny{6}};
			\node[anchor=north west] at (5.4, -0.05) {\tiny{7}};
			\node[anchor=north west] at (6.4, -0.05) {\tiny{8}};
			\node[anchor=north west] at (7.4, -0.05) {\tiny{9}};
			\node[anchor=north west] at (8.3, -0.05) {\tiny{10}};
			\node[anchor=north west] at (9.3, -0.05) {\tiny{11}};
			\node[anchor=north west] at (10.3, -0.05) {\tiny{12}};
			\node[anchor=north west] at (11.3, -0.05) {\tiny{13}};
			\node[anchor=north west] at (12.3, -0.05) {\tiny{14}};
			\node[anchor=north west] at (13.3, -0.05) {\tiny{15}};
			\node[anchor=north west] at (14.3, -0.05) {\tiny{16}};
			\node[anchor=north west] at (15.3, -0.05) {\tiny{17}};
			\node[anchor=north west] at (16.3, -0.05) {\tiny{18}};
			\node[anchor=north west] at (17.3, -0.05) {\tiny{19}};
			\node[anchor=north west] at (18.3, -0.05) {\tiny{20}};
			\node[anchor=north west] at (19.3, -0.05) {\tiny{21}};
			\node[anchor=north west] at (20.3, -0.05) {\tiny{22}};
		\end{tikzpicture}
		\raisebox{3ex}{$\xrightarrow{}$}
		\hspace{-0.2cm}
		\begin{tikzpicture}[scale = 0.34] 
			\draw (-0.5,0) -- (21.5,0);
			\foreach \x in {0,1,2,3,4,5,6,7,8,9,10,11,15,14,12,16,17,18,19,20,21}
			{
				\filldraw (\x,0) circle (2pt);
			}
			\draw[thick] (13,0) circle (6pt);
			\draw[thick] (13,0) circle (1.5pt);
			\draw[very thick, line width=0.8pt] (0,0) arc (180:0:2.5 and 1);
			\draw[black] (1,0) arc (180:0:3 and 1);
			\draw[black] (2,0) arc (180:0:3 and 1);
			\draw[black] (3,0) arc (180:0:3.5 and 1.5); \draw[black] (4,0) arc (180:0:5.5 and 1.5);
			\draw[black] (6,0) arc (180:0:2.5 and 1);
			\draw[black] (9,0) arc (180:0:1.5 and 1);
			\draw[very thick, line width=0.8pt] (13,0) arc (180:0:2 and 1);
			\draw[very thick, line width=0.8pt] (14,0) arc (180:0:3 and 1.5);
			\draw[black] (16,0) arc (180:0:1 and 1);
			\draw[very thick, line width=0.8pt] (19,0) arc (180:0:1 and 1);
			
			\node[anchor=north west] at (-0.3, 2.3) {\scriptsize{$M^{(2)}$}};
			\node[anchor=north west] at (-0.6, -0.05) {\tiny{1}};
			\node[anchor=north west] at (0.4, -0.05) {\tiny{2}};
			\node[anchor=north west] at (1.4, -0.05) {\tiny{3}};
			\node[anchor=north west] at (2.4, -0.05) {\tiny{4}};
			\node[anchor=north west] at (3.4, -0.05) {\tiny{5}};
			\node[anchor=north west] at (4.4, -0.05) {\tiny{6}};
			\node[anchor=north west] at (5.4, -0.05) {\tiny{7}};
			\node[anchor=north west] at (6.4, -0.05) {\tiny{8}};
			\node[anchor=north west] at (7.4, -0.05) {\tiny{9}};
			\node[anchor=north west] at (8.3, -0.05) {\tiny{10}};
			\node[anchor=north west] at (9.3, -0.05) {\tiny{11}};
			\node[anchor=north west] at (10.3, -0.05) {\tiny{12}};
			\node[anchor=north west] at (11.3, -0.05) {\tiny{13}};
			\node[anchor=north west] at (12.3, -0.05) {\tiny{14}};
			\node[anchor=north west] at (13.3, -0.05) {\tiny{15}};
			\node[anchor=north west] at (14.3, -0.05) {\tiny{16}};
			\node[anchor=north west] at (15.3, -0.05) {\tiny{17}};
			\node[anchor=north west] at (16.3, -0.05) {\tiny{18}};
			\node[anchor=north west] at (17.3, -0.05) {\tiny{19}};
			\node[anchor=north west] at (18.3, -0.05) {\tiny{20}};
			\node[anchor=north west] at (19.3, -0.05) {\tiny{21}};
			\node[anchor=north west] at (20.3, -0.05) {\tiny{22}};
		\end{tikzpicture}
		\hspace{-0.4cm}
		\raisebox{3ex}{$\xrightarrow{}$}
		\hspace{-0.2cm}
		\begin{tikzpicture}[scale = 0.34] 
			\draw (-0.5,0) -- (21.5,0);
			\foreach \x in {0,1,2,3,4,5,6,7,8,9,10,11,15,14,12,13,16,17,18,19,20,21}
			{
				\filldraw (\x,0) circle (2pt);
			}
			\draw[black] (0,0) arc (180:0:3 and 1);
			\draw[black] (1,0) arc (180:0:3.5 and 1);
			\draw[black] (2,0) arc (180:0:3.5 and 1);
			\draw[black] (3,0) arc (180:0:4 and 1.5); \draw[black] (4,0) arc (180:0:5.5 and 1.5);
			\draw[black] (5,0) arc (180:0:6 and 1.5);
			\draw[black] (7,0) arc (180:0:2.5 and 1);
			\draw[black] (10,0) arc (180:0:1.5 and 1);
			\draw[black] (14,0) arc (180:0:3 and 1.5);
			\draw[black] (16,0) arc (180:0:1 and 1);
			\draw[black] (19,0) arc (180:0:1 and 1);
			\node[anchor=north west] at (-0.3, 2.3) {\scriptsize{$M^{(3)}$}};
			\node[anchor=north west] at (-0.6, -0.05) {\tiny{1}};
			\node[anchor=north west] at (0.4, -0.05) {\tiny{2}};
			\node[anchor=north west] at (1.4, -0.05) {\tiny{3}};
			\node[anchor=north west] at (2.4, -0.05) {\tiny{4}};
			\node[anchor=north west] at (3.4, -0.05) {\tiny{5}};
			\node[anchor=north west] at (4.4, -0.05) {\tiny{6}};
			\node[anchor=north west] at (5.4, -0.05) {\tiny{7}};
			\node[anchor=north west] at (6.4, -0.05) {\tiny{8}};
			\node[anchor=north west] at (7.4, -0.05) {\tiny{9}};
			\node[anchor=north west] at (8.3, -0.05) {\tiny{10}};
			\node[anchor=north west] at (9.3, -0.05) {\tiny{11}};
			\node[anchor=north west] at (10.3, -0.05) {\tiny{12}};
			\node[anchor=north west] at (11.3, -0.05) {\tiny{13}};
			\node[anchor=north west] at (12.3, -0.05) {\tiny{14}};
			\node[anchor=north west] at (13.3, -0.05) {\tiny{15}};
			\node[anchor=north west] at (14.3, -0.05) {\tiny{16}};
			\node[anchor=north west] at (15.3, -0.05) {\tiny{17}};
			\node[anchor=north west] at (16.3, -0.05) {\tiny{18}};
			\node[anchor=north west] at (17.3, -0.05) {\tiny{19}};
			\node[anchor=north west] at (18.3, -0.05) {\tiny{20}};
			\node[anchor=north west] at (19.3, -0.05) {\tiny{21}};
			\node[anchor=north west] at (20.3, -0.05) {\tiny{22}};
		\end{tikzpicture}
		\caption{An example of the map $\Phi$ with $\Phi(M) = M^{(3)}$. The thick arcs represent the leftmost occurrence of 
			$P^4_4$, and the circled dots correspond to 
			$a_{i_2}, a_{i_2}+1, \ldots, a_{i_2}+j-1$.}
		\label{fig-exa-Phi}
	\end{figure}
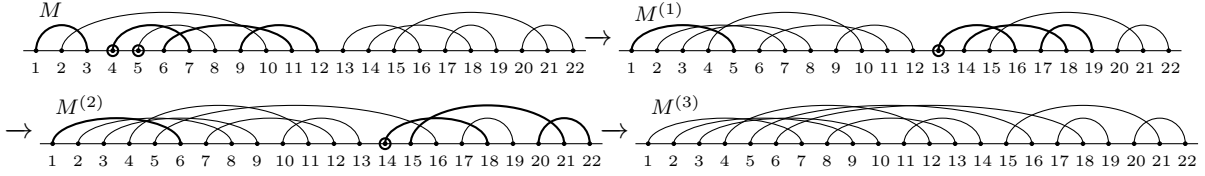
	
	\begin{lem}\label{thm-Phi-P2-P4}
		The map $\Phi$ is a bijection from $\mathcal{M}_n(P^k_2)$ to $\mathcal{M}_n(P^k_4)$.
	\end{lem}
	\begin{proof}
		It suffices to prove that $\Phi$ is well-defined and the inverse map $\Phi^{-1}$ exists.
		
		\textbf{$\Phi$ is well-defined.} We need to justify that the process terminates and results in $\Phi(M) \in \mathcal{M}_n(P^k_4)$.  
		Clearly, $\Phi(M) \in \mathcal{M}_n$.  
		Since each application of the moving procedure eliminates existing occurrences of $P^k_4$ and creates at least one occurrence of $P^k_2$, which does not disappear in subsequent steps, it remains to show that no new occurrences of $P^k_4$ are created.  Suppose, for contradiction, that a new occurrence of $P^k_4$ were introduced at some stage.  
		Such an occurrence would necessarily involve arcs with the recently moved openers $a_{i_2}, a_{i_2}+1, \ldots, a_{i_2}+j-1$, which contradicts our choice of $a_{i_1}$ as the first possible opener.
		
		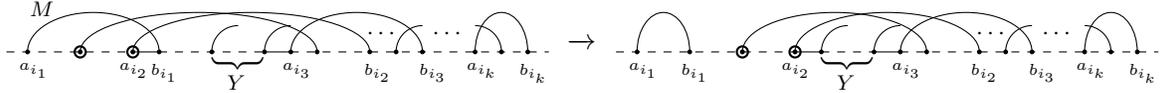
\begin{figure}[h]
			\centering
			\begin{tikzpicture}[scale = 0.35] 
				\draw [dashed](-0.8,0) -- (4,0);
				\draw (4,0) -- (5,0);
				\draw[dashed] (5,0) -- (9,0);
				\draw (9,0) -- (10,0);
				\draw[dashed] (10,0) -- (19.8,0);
				\foreach \x in {0,5,7,9,10,11,13,14,15,17,18,19}
				{
					\filldraw (\x,0) circle (2pt);
				}
				\draw[thick] (2,0) circle (6pt);
				\draw[thick] (2,0) circle (1.5pt);
				\draw[thick] (4,0) circle (6pt);
				\draw[thick] (4,0) circle (1.5pt);
				\draw[domain=0:90, smooth, variable=\t, shift={(8,0)}] 
				plot ({-cos(\t)}, {sin(\t)});
				\draw[domain=0:90, smooth, variable=\t, shift={(10,0)}] 
				plot ({-cos(\t)}, {sin(\t)});
				\draw[domain=0:90, smooth, variable=\t, shift={(15,0)}] 
				plot ({-cos(\t)}, {sin(\t)});
				\draw[domain=0:90, smooth, variable=\t, shift={(17,0)}] 
				plot ({cos(\t)}, {sin(\t)});
				\draw[black] (0,0) arc (180:0:2.5 and 1.5);
				\draw[black] (2,0) arc (180:0:4.5 and 1.5);
				\draw[black] (4,0) arc (180:0:4.5 and 1.5);
				\draw[black] (10,0) arc (180:0:2.5 and 1.5);
				\draw[black] (17,0) arc (180:0:1 and 1.5);
				
				\node[anchor=north west] at (-0.3, 2.3) {\scriptsize $M$};
				\node[anchor=north west][rotate=90] at (6.6, -1.2) {$\Big{\{}$};
				\node[anchor=north west] at (7.2, -0.5) {\scriptsize{$Y$}};
				\node[anchor=north west] at (-0.7, -0.05) {\tiny{$a_{i_1}$}};
				\node[anchor=north west] at (4.3, -0.05) {\tiny{$b_{i_1}$}};
				\node[anchor=north west] at (3.1, -0.05) {\tiny{$a_{i_2}$}};
				\node[anchor=north west] at (9.3, -0.05) {\tiny{$a_{i_3}$}};
				\node[anchor=north west] at (12.4, -0.05) {\tiny{$b_{i_2}$}};
				\node[anchor=north west] at (14.5, -0.05) {\tiny{$b_{i_3}$}};
				\node[anchor=north west] at (16.3, -0.05) {\tiny{$a_{i_k}$}};
				\node[anchor=north west] at (18.3, -0.05) {\tiny{$b_{i_k}$}};
				\node[anchor=north west] at (12.5, 1.3) {\footnotesize{$\cdots$}};
				\node[anchor=north west] at (15, 1.3) {\footnotesize{$\cdots$}};
			\end{tikzpicture}
			\raisebox{4ex}{$\xrightarrow{}$}
			\begin{tikzpicture}[scale = 0.35] 
				\draw [dashed](-0.8,0) -- (6,0);
				\draw (6,0) -- (7,0);
				\draw[dashed] (7,0) -- (9,0);
				\draw (9,0) -- (10,0);
				\draw[dashed] (10,0) -- (19.8,0);
				\foreach \x in {0,2,7,9,10,11,13,14,15,17,18,19}
				{
					\filldraw (\x,0) circle (2pt);
				}
				\draw[thick] (4,0) circle (6pt);
				\draw[thick] (4,0) circle (1.5pt);
				\draw[thick] (6,0) circle (6pt);
				\draw[thick] (6,0) circle (1.5pt);
				\draw[domain=0:90, smooth, variable=\t, shift={(15,0)}] 
				plot ({-cos(\t)}, {sin(\t)});
				\draw[domain=0:90, smooth, variable=\t, shift={(8,0)}] plot ({-cos(\t)}, {sin(\t)});
				\draw[domain=0:90, smooth, variable=\t, shift={(10,0)}] 
				plot ({-cos(\t)}, {sin(\t)});
				
				\draw[domain=0:90, smooth, variable=\t, shift={(17,0)}] 
				plot ({cos(\t)}, {sin(\t)});
				\draw[black] (0,0) arc (180:0:1 and 1.5);
				\draw[black] (4,0) arc (180:0:3.5 and 1.5);
				\draw[black] (6,0) arc (180:0:3.5 and 1.5);
				\draw[black] (10,0) arc (180:0:2.5 and 1.5);
				\draw[black] (17,0) arc (180:0:1 and 1.5);
				
				\node[anchor=north west] at (-1.2, 2.6) {};
				\node[anchor=north west][rotate=90] at (6.6, -1.2) {$\Big{\{}$};
				\node[anchor=north west] at (7.2, -0.5) {\scriptsize{$Y$}};
				\node[anchor=north west] at (-0.7, -0.05) {\tiny{$a_{i_1}$}};
				\node[anchor=north west] at (1.3, -0.05) {\tiny{$b_{i_1}$}};
				\node[anchor=north west] at (5.1, -0.05) {\tiny{$a_{i_2}$}};
				\node[anchor=north west] at (9.3, -0.05) {\tiny{$a_{i_3}$}};
				\node[anchor=north west] at (12.3, -0.05) {\tiny{$b_{i_2}$}};
				\node[anchor=north west] at (14.5, -0.05) {\tiny{$b_{i_3}$}};
				\node[anchor=north west] at (16.3, -0.05) {\tiny{$a_{i_k}$}};
				\node[anchor=north west] at (18.3, -0.05) {\tiny{$b_{i_k}$}};
				\node[anchor=north west] at (12.5, 1.3) {\footnotesize{$\cdots$}};
				\node[anchor=north west] at (15, 1.3) {\footnotesize{$\cdots$}};
			\end{tikzpicture}
			\caption{The first step in the inverse map $\Phi^{-1}$.}
			\label{fig-M_nP^k_4-M_nP^k_2}
		\end{figure}
		
		{\bf Reversibility of $\Phi$.} For any $M \in \mathcal{M}_n(P^k_4)$, if $M \in \mathcal{M}_n(P^k_2)$, then we let $\Phi^{-1}(M) = M$.  
		Otherwise, consider an occurrence $[a_{i_1}, b_{i_1}], [a_{i_2}, b_{i_2}], \ldots, [a_{i_k}, b_{i_k}]$ of the pattern $P^k_2$ in $M$, chosen by first selecting the rightmost possible $b_{i_k}$, then the rightmost possible $b_{i_{k-1}}$, and so on.  Let $[a_{i_2}-j+1, b_{i_2}-j+1], \ldots, [a_{i_2}-1, b_{i_2}-1], [a_{i_2}, b_{i_2}]$ form a maximal crossing with consecutive openers and closers.  
		Let $Y$ be the (possibly empty) set of openers between $b_{i_1}$ and $a_{i_3}$ whose corresponding closers lie to the right of $b_{i_2}$.  
		We obtain $M^{(1)}$ by moving the openers $a_{i_2}-j+1, \ldots, a_{i_2}$ immediately to the left of the first opener in $Y$ (or $a_{i_3}$ if $Y = \emptyset$), as shown in Figure~\ref{fig-M_nP^k_4-M_nP^k_2}.  The occurrence(s) of the pattern $P^k_2$ associated with the arcs $[a_{i_2}-j+1, b_{i_2}-j+1], \ldots, [a_{i_2}-1, b_{i_2}-1], [a_{i_2}, b_{i_2}]$ will be eliminated, and at least one occurrence of $P^k_4$ will be introduced, which will remain until the end of the procedure.  
		Repeat this process until no $P^k_2$ remains, and define the resulting matching as $\Phi^{-1}(M)$.
	\end{proof}
	
	Following a similar line of reasoning as in Lemma~\ref{thm-Phi-P2-P4}, we get the following result.
	\begin{coro}\label{coro-Phi-P5-P3}
		There exists a bijection from $\M_n(P^k_5)$ to $\M_n(P^k_3)$.
	\end{coro}
	
	\begin{proof}[Proof of Theorem~\ref{thm1-general}]
		By reversing all matchings, it follows that 
		\(|\mathcal{M}_n(P^k_4)| = |\mathcal{M}_n(P^k_5)|\).  Together with Lemma~\ref{thm-Phi-P2-P4} and Corollary~\ref{coro-Phi-P5-P3}, the proof is complete.
	\end{proof}
	
	\begin{proof}[Proof of Theorem~\ref{thm-Catalan}]
		By setting $k=4$ in Theorem~\ref{thm1-general} and using Lemma~\ref{lem-P2-catalan}, we  obtain 
		$|\M_n(P_i)|= C_n$ for $2\le i\le 5$.
		Together with Lemma~\ref{lem-Nonnesting}, we complete the proof.
	\end{proof}
	
	\section{Fishburn structures and known bijections}\label{sec-bijection}
	
	We first recall the basic definitions of combinatorial objects counted by the Fishburn numbers.
	
	An unlabelled poset \( P = (X, \le) \) is \textbf{(2+2)}-\emph{free} if it does not contain an induced subposet isomorphic to \textbf{2+2}, the union of two disjoint 2-element chains. We let \( \mathcal{P}_n \) denote the set of \textbf{(2+2)}-free posets with \( n \) elements.

	Let \( \alpha = (\alpha_1, \alpha_2, \ldots, \alpha_n) \) be an integer sequence. We say that \( i \in \{1, \ldots, n-1\} \) is an \emph{ascent} in \( \alpha \) if \( \alpha_i < \alpha_{i+1} \).
	Denote by $\asc(\alpha)$ the number of ascents of $\alpha$. 
	The set of \textit{ascent sequences} of length $n$ is defined as
	\(\Asc{n} = \big\{\,(\alpha_1,\ldots , \alpha_n): \alpha_1=0\text{ and } 
	0 \leq \alpha_i \leq 1+\asc(\alpha_1,\ldots , \alpha_{i-1})
	\mbox{ for }2 \leq i\leq n \,\big\}.
	\)
	
	A permutation $\pi=\pi_1\pi_2\cdots\pi_n$ contains an occurrence of the \textit{Fishburn pattern}, if there exist $i,j$ with $1 \le i, j \le n$ such that $j>i+1$ and $\pi_j+1=\pi_i<\pi_{i+1}$. Permutations avoiding the Fishburn pattern are called \emph{Fishburn permutations}, and we denote by $\mathcal{F}_n$ the set of such permutations of length $n$.
	
	Denote by $\mathcal{P}_n(X)$ (resp., $\mathcal{A}_n(X)$, $\mathcal{F}_n(X)$) the subset of $\mathcal{P}_n$ (resp., $\mathcal{A}_n$, $\mathcal{F}_n$) consisting of objects that avoid a forbidden subobject $X$, and let $\PP(X)=\cup_{n\geq 0}\PP_n(X)$, $\A(X)=\cup_{n\geq 0}\A_n(X)$, and $\F(X)=\cup_{n\geq 0}\F_n(X)$.
	In the following subsections, we describe some known bijections $\Omega$, $\Psi$, $\Upsilon$, and $\Lambda$, along with their restrictions, between these sets, which will be used in Section~\ref{sec-stats}.
	
	\subsection{Restrictions of $\Omega$ on Stoimenow matchings yielding $\mathcal{P}_n\text{{\bf (3+1)}}$ and $\mathcal{P}_n\text{{\bf (N)}}$}\label{subsec-matching-poset}
	
	The bijection $\Omega$ was introduced by Bousquet-M\'elou \etal ~\cite{Bousquet-Claesson-Dukes-Kitaev}.
	Let $M=\{[a_1, b_1], \dots, [a_n, b_n]\} \in \M_n$. For each pair of arcs $[a_i, b_i]$ and $[a_j, b_j]$ in $M$, let $p_i$ and $p_j$ be their respective images in $P=\Omega(M)$. Then, we define an order $p_i < p_j$ on $P$ if
	\(
	b_i < a_j.
	\)
	Figure~\ref{fig-exa-matching-perm} illustrates an example of the bijection $\Omega$.
	
	The poset \textbf{3+1} is formed by the disjoint union of a 3-element chain and a 1-element chain, while \textbf{N} is the poset on $\{a,b,c,d\}$ with relations $a < c$, $a < d$, and $b < d$.  It is known that both $\mathcal{P}_n\textbf{(3+1)}$ and $\mathcal{P}_n\textbf{(N)}$ are enumerated by the $n$-th Catalan number, see~\cite{Disanto-Ferrari-Pinzani-Rinaldi, Stanley}. In the following two theorems, we impose restrictions on the map $\Omega$ to establish bijections between specific subsets of Stoimenow matchings and \textbf{(2+2)}-free posets.
	
	\begin{thm} \label{thm-P1-Poset3+1}
		The restriction of the map $\Omega$ to $\M_n(P_1)$ is also a bijection onto $\mathcal{P}_n\text{{\bf (3+1)}}$.   
	\end{thm}
	\begin{proof}
		Under the map $\Omega$, if $M\in\M_n$ contains an occurrence of the pattern $P_1$, then $\Omega(M)$ contains an occurrence of the pattern \textbf{3+1}. 
		Conversely, suppose $P \in \mathcal{P}_n$ contains an  occurrence of \textbf{3+1} with elements $p_1 < p_2 < p_3$ and $p_4$ incomparable to the others. Let $m_1$, $m_2$, $m_3$, and $m_4$ be the respective pre-images of $p_1$, $p_2$, $p_3$, and $p_4$ under $\Omega^{-1}$. If $m_1$, $m_2$, $m_3$, and $m_4$ form an occurrence of $P_1$, then we are done. Otherwise, since $m_1$, $m_2$, and $m_3$ must form a 3-noncrossing, only three cases are possible, which are presented schematically in Figure~\ref{fig-M_n(T1)-(2+2,3+1)-case2}. In each case, to avoid Type~I and Type~II arcs, we can find arcs $m'$ and $m''$ such that in the first (resp., second; third) situation, we obtain an occurrence of $P_1$ in $\Omega^{-1}(P)$ formed by the arcs $m'$, $m_1$, $m_3$, $m_4$ (resp., $m_1$, $m_2$, $m'$, $m_4$; $m'$, $m_1$, $m''$, $m_4$).  \end{proof}
	
	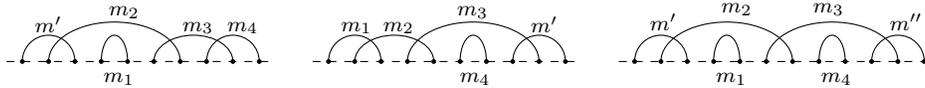
\begin{figure}
		\centering
		\begin{tikzpicture}[scale = 0.35]
			\draw [dashed](-0.6,0) -- (9.4,0);
			\foreach \x in {0,1,2,3,4,5,6,7,8,9}
			{
				\filldraw (\x,0) circle (2pt);
			}
			\draw[black] (0,0) arc (180:0:1 and 1);
			\draw[black] (1,0) arc (180:0:3.5 and 1.5);
			\draw[black] (3,0) arc (180:0:0.5 and 1);
			\draw[black] (5,0) arc (180:0:0.5 and 1);
			\draw[black] (7,0) arc (180:0:1 and 1);
			\node[anchor=north west] at (0.3, 2.1) {{\scriptsize $m'$}};
			\node[anchor=north west] at (2.7, -0.05) {{\scriptsize $m_{1}$}};
			\node[anchor=north west] at (3.8, 2.5) {{\scriptsize $m_{4}$}};
			\node[anchor=north west] at (4.6, -0.05) {{\scriptsize $m_{2}$}};
			\node[anchor=north west] at (7.2, 1.9) {{\scriptsize $m_{3}$}};
		\end{tikzpicture}
		\hspace{0.3cm}
		\begin{tikzpicture}[scale = 0.35]
			\draw [dashed](-0.6,0) -- (9.4,0);
			\foreach \x in {0,1,2,3,4,5,6,7,8,9}
			{
				\filldraw (\x,0) circle (2pt);
			}
			\draw[black] (0,0) arc (180:0:1 and 1);
			\draw[black] (1,0) arc (180:0:3.5 and 1.5);
			\draw[black] (3,0) arc (180:0:0.5 and 1);
			\draw[black] (5,0) arc (180:0:0.5 and 1);
			\draw[black] (7,0) arc (180:0:1 and 1);
			\node[anchor=north west] at (0.1, 1.9) {{\scriptsize $m_{1}$}};
			\node[anchor=north west] at (2.7, -0.05) {{\scriptsize $m_{2}$}};
			\node[anchor=north west] at (3.8, 2.5) {{\scriptsize $m_{4}$}};
			\node[anchor=north west] at (4.6, -0.05) {{\scriptsize $m_{3}$}};
			\node[anchor=north west] at (7.3, 2.1) {{\scriptsize $m'$}};
		\end{tikzpicture}
		\hspace{0.3cm}
		\begin{tikzpicture}[scale = 0.35]
			\draw [dashed](-0.6,0) -- (11.4,0);
			\foreach \x in {0,1,2,3,4,5,6,7,8,9,10,11}
			{
				\filldraw (\x,0) circle (2pt);
			}
			\draw[black] (0,0) arc (180:0:1 and 1);
			\draw[black] (1,0) arc (180:0:4.5 and 1.5);
			\draw[black] (3,0) arc (180:0:0.5 and 1);
			\draw[black] (5,0) arc (180:0:0.5 and 1);
			\draw[black] (7,0) arc (180:0:0.5 and 1);
			\draw[black] (9,0) arc (180:0:1 and 1);
			\node[anchor=north west] at (0.2, 2.2) {{\scriptsize $m'$}};
			\node[anchor=north west] at (2.7, -0.05) {{\scriptsize $m_{1}$}};
			\node[anchor=north west] at (4, 2.6) {{\scriptsize $m_{4}$}};
			\node[anchor=north west] at (4.6, -0.05) {{\scriptsize $m_{2}$}};
			\node[anchor=north west] at (6.7, -0.05) {{\scriptsize $m_{3}$}};
			\node[anchor=north west] at (9.3, 2.2) {{\scriptsize $m''$}};
		\end{tikzpicture}
		\caption{Related to the proof of Theorem~\ref{thm-P1-Poset3+1}.} 
		\label{fig-M_n(T1)-(2+2,3+1)-case2}
	\end{figure}
	
	\begin{thm} \label{thm-P2-N}
		The map $\Omega$, restricted to $\M_n(P_2)$, is also a bijection onto $\mathcal{P}_n{\bf {(N)}}$.
	\end{thm}
	\begin{proof}
		If $M\in\M_n$ contains an occurrence of the pattern $P_2$, then it is clear that $\Omega(M)$ contains an occurrence of the pattern \textbf{N}.
		To prove the converse, suppose $P \in \mathcal{P}_n$ contains an occurrence of the pattern \textbf{N} formed by elements $p_1$, $p_2$, $p_3$, and $p_4$, where $p_1 < p_3$, $p_1 < p_4$, $p_2 < p_4$ and all the other pairs are incomparable. Let $m_1$, $m_2$, $m_3$, and $m_4$ be the respective pre-images of $p_1$, $p_2$, $p_3$, and $p_4$ under $\Omega^{-1}$. Similar to the proof of Theorem~\ref{thm-P1-Poset3+1}, if $m_1$, $m_2$, $m_3$, and $m_4$ form an occurrence of $P_2$, then we are done. Otherwise, there are three possible cases, as shown in Figure~\ref{fig-M_n(P2)-N-case2}. In each of these cases, we can find two arcs, $m'$ and $m''$, that are involved  in an occurrence of $P_2$ in $\Omega^{-1}(P)$. 
	\end{proof}
	
	\begin{figure}[h]
		\centering
		\begin{tikzpicture}[scale = 0.35]
			\draw [dashed](-0.6,0) -- (9.4,0);
			\foreach \x in {0,1,2,3,4,5,6,7,8,9}
			{
				\filldraw (\x,0) circle (2pt);
			}
			\draw[black] (0,0) arc (180:0:1 and 1);
			\draw[black] (1,0) arc (180:0:2.5 and 1.5);
			\draw[black] (3,0) arc (180:0:0.5 and 1);
			\draw[black] (5,0) arc (180:0:1.5 and 1);
			\draw[black] (7,0) arc (180:0:1 and 1);
			\node[anchor=north west] at (0.1, 2.1) {{\scriptsize $m'$}};
			\node[anchor=north west] at (2.7, -0.05) {{\scriptsize $m_{1}$}};
			\node[anchor=north west] at (2.9, 2.5) {{\scriptsize $m_{2}$}};
			\node[anchor=north west] at (5.7, 1.9) {{\scriptsize $m_{3}$}};
			\node[anchor=north west] at (7.4, 1.9) {{\scriptsize $m_{4}$}};
		\end{tikzpicture}
		\hspace{0.3cm}
		\begin{tikzpicture}[scale = 0.35]
			\draw[dashed] (-0.6,0) -- (9.4,0);
			\foreach \x in {0,1,2,3,4,5,6,7,8,9}
			{
				\filldraw (\x,0) circle (2pt);
			}
			\draw[black] (0,0) arc (180:0:1 and 1);
			\draw[black] (1,0) arc (180:0:1.5 and 1);
			\draw[black] (3,0) arc (180:0:2.5 and 1.5);
			\draw[black] (5,0) arc (180:0:0.5 and 1);
			\draw[black] (7,0) arc (180:0:1 and 1);
			\node[anchor=north west] at (0.1, 1.9) {{\scriptsize $m_{1}$}};
			\node[anchor=north west] at (1.7, 1.9) {{\scriptsize $m_{2}$}};
			\node[anchor=north west] at (4.5, 2.5) {{\scriptsize $m_{3}$}};
			\node[anchor=north west] at (4.6, -0.05) {{\scriptsize $m_{4}$}};
			\node[anchor=north west] at (7.3, 2.1) {{\scriptsize $m'$}};
		\end{tikzpicture}
		\hspace{0.3cm}
		\begin{tikzpicture}[scale = 0.35]
			\draw [dashed](-0.6,0) -- (11.4,0);
			\foreach \x in {0,1,2,3,4,5,6,7,8,9,10,11}
			{
				\filldraw (\x,0) circle (2pt);
			}
			\draw[black] (0,0) arc (180:0:1 and 1);
			\draw[black] (1,0) arc (180:0:2.5 and 1.5);
			\draw[black] (3,0) arc (180:0:0.5 and 1);
			\draw[black] (5,0) arc (180:0:2.5 and 1.5);
			\draw[black] (7,0) arc (180:0:0.5 and 1);
			\draw[black] (9,0) arc (180:0:1 and 1);
			\node[anchor=north west] at (0.2, 2.2) {{\scriptsize $m'$}};
			\node[anchor=north west] at (2.7, -0.05) {{\scriptsize $m_{1}$}};
			\node[anchor=north west] at (2.9, 2.6) {{\scriptsize $m_{2}$}};
			\node[anchor=north west] at (6.4, 2.6) {{\scriptsize $m_{3}$}};
			\node[anchor=north west] at (6.7, -0.05) {{\scriptsize $m_{4}$}};
			\node[anchor=north west] at (9.3, 2.2) {{\scriptsize $m''$}};
		\end{tikzpicture}
		\caption{Related to the proof of Theorem~\ref{thm-P2-N}.} 
		\label{fig-M_n(P2)-N-case2}
	\end{figure}
	
	\subsection{Restriction of $\Psi$ on Stoimenow matchings yielding  $\A_n(101)$}\label{sec-matching-101}
	
	Duncan and Steingr\'{\i}msson~\cite{Duncan-Steingrimsson} proved that $|\A_n(101)|=C_n$, and  Bousquet-M\'elou \etal ~\cite{Bousquet-Claesson-Dukes-Kitaev}
	established a bijection $\mathcal{O}\colon\PP_n\to\A_n$. Dukes and McNamara~\cite{Dukes-McNamara} further studied the restricted bijection between $\PP_n(\mathbf{N})$ and $\A_n(101)$.
	We consider the bijection $\Psi=\mathcal{O} \circ \Omega: \mathcal{M}_n \to \mathcal{A}_n$, 
	and note that a direct bijection from $\mathcal{M}_n$ to $\mathcal{A}_n$ was given by Claesson et al.~\cite{Claesson-Dukes-Kitaev}. 
	Since the description of $\mathcal{O}$ is rather lengthy, we omit it here and refer the reader to~\cite{Bousquet-Claesson-Dukes-Kitaev} for details.
	
	\begin{lem}[\cite{Dukes-McNamara},  Proposition 5.3]\label{lem-cite-N-101}
		The restriction of the map $\mathcal{O}$ to $\PP_n{\bf{(N)}}$ is also a bijection onto $\A_n(101)$.
	\end{lem}
	
	\begin{thm}\label{thm-P2-101}
		The map $\Psi$, restricted to $\M_n(P_2)$, is also a bijection onto $\A_n(101)$.
	\end{thm}
	\begin{proof}
		The proof follows from Theorem~\ref{thm-P2-N} and Lemma~\ref{lem-cite-N-101}.
	\end{proof}
	
	\subsection{Restriction of $\Upsilon$ on Stoimenow matchings yielding  $\F_n(3142)$}\label{subsec-matching-perm}
	
	Gil and Weiner~\cite{Gil-Weiner} showed that $\F_n(3142)$, the set of Fishburn permutations of length $n$ avoiding the pattern $3142$, is counted by $C_n$. In order to show that the pre-image of $\F_n(3142)$ under the bijection $\Upsilon$ is given by $\M(P_2)$ (see Theorem~\ref{thm-MnP2-Fn3142}), we need to consider ascent sequences as an intermediate object.
	
	The set of \emph{modified ascent sequences}
	of length $n$, denoted \( \mathcal{\hat{A}}_n \), forms the bijective image of \( \mathcal{A}_n \) under the mapping \( \Delta : \alpha \mapsto \hat{\alpha} \) (we follow the description in~\cite{Cerbai-Claesson}).
	Given an ascent sequence \( \alpha = (\alpha_1, \alpha_2, \ldots, \alpha_n) \in \mathcal{A}_n \) and $1\leq j\leq n-1$, let
	\[
	M(\alpha, j) = \alpha^+, \quad \text{where } \alpha^+(i) = 
	\begin{cases} 
		\alpha_i + 1 & \text{if } i < j \text{ and } \alpha_i \geq \alpha_{j+1}, \\
		\alpha_i & \text{otherwise}.
	\end{cases}
	\]
	We extend the definition of \( M \) to multiple indices \( j_1, j_2, \ldots, j_k \) by
	\[
	M(\alpha, j_1, j_2, \ldots, j_k) = M\left(M(\alpha, j_1, \ldots, j_{k-1}), j_k\right).
	\]
	Denote by \( \operatorname{Asc}(\alpha) \)  the ascent list of \( \alpha \), i.e., the ascent set of $\alpha$ written in increasing order. Let \( \Delta(\alpha) = M(\alpha, \operatorname{Asc}(\alpha)) \). As an example, consider the ascent sequence \( \alpha = (0, 1, 0, 1, 3, 1, 1, 2) \). Then \( \operatorname{Asc}(\alpha) = (1, 3, 4, 7) \), and the map \( \Delta \) computes the modified ascent sequence \( \hat{\alpha} \) as follows:
	\[
	\begin{array}{r@{\;}l}
		\alpha =&(0,\underline{1},0,1,3,1,1,2)  \\
		&(0,1,0,\underline{1},3,1,1,2)\\
		&(0,2,0,1,\underline{3},1,1,2) \\
		&(0,2,0,1,3,1,1,\underline{2}) \\
		&(0,3,0,1,4,1,1,2) = \hat{\alpha} \\
	\end{array}
	\]
	Let $\mathrm{id}_n = 1\,2\,\cdots\,n$ be the identity permutation of length $n$.  
	The transposition operation \( \mathcal{T} : \hat{\mathcal{A}}_n \to \mathcal{F}_n \) is defined on biwords as follows~\cite{Cerbai-Claesson}:
	\[
	{\binom{\mathrm{id}}{\hat{\alpha}}}^{\mathcal{T}} = \binom{\mathrm{sort}(\hat{\alpha})}{\mathcal{T}(\hat{\alpha})},
	\]
	where $\mathrm{sort}(\hat{\alpha})$ is obtained by sorting the entries of $\hat{\alpha}$ in non-decreasing order.
	For example, if $\hat{\alpha}= (0,3,0,1,4,1,1,2)$, then
	\[
	{\binom{\text{id}}{\hat{\alpha}}}^{\T} = {\binom{1\,2\,3\,4\,5\,6\,7\,8}{0\,3\,0\,1\,4\,1\,1\,2}}^{\T} = \binom{0\,0\,1\,1\,1\,2\,3\,4}{3\,1\,7\,6\,4\,8\,2\,5} = \binom{\text{sort}(\hat{\alpha})}{\T(\hat{\alpha})}.
	\]
	Thus, $\pi = \T(\hat{\alpha})= 31764825$.
	The mapping $\Lambda: \A_n \to \F_n$ is the composition $\Lambda = \T \circ \Delta$. Finally, the mapping $\Upsilon: \M_n \to \F_n$ is given by the composition $\Upsilon := \Lambda \circ \Psi = \T \circ \Delta \circ \Psi$. 
	An example is shown in Figure~\ref{fig-exa-matching-perm}: for the leftmost matching $M$, we have $\Upsilon(M)= 31764825$.
	
	\begin{figure}[h]
		\centering
		\begin{tikzpicture}[scale = 0.4, baseline={(current bounding box.center)}]
			\draw (-0.3,0) -- (15.3,0);
			\foreach \x in {0,1,2,3,4,5,6,7,8,9, 10,11, 12,13,14,15}
			{
				\filldraw (\x,0) circle (2pt);
			}
			\draw[black] (0,0) arc (180:0:1 and 1);
			\draw[black] (1,0) arc (180:0:2.5 and 1.2);
			\draw[black] (3,0) arc (180:0:2.5 and 1.2);
			\draw[black] (4,0) arc (180:0:2.5 and 1.2);
			\draw[black] (5,0) arc (180:0:3 and 1.2);
			\draw[black] (7,0) arc (180:0:3.5 and 1.2);
			\draw[black] (10,0) arc (180:0:1 and 1);
			\draw[black] (13,0) arc (180:0:1 and 1);
			\node[anchor=north west] at (6.6, -0.8) {\scriptsize{$M$}};
			\node[anchor=north west] at (-0.6, -0.05) {\scriptsize{1}};
			\node[anchor=north west] at (0.4, -0.05) {\scriptsize{2}};
			\node[anchor=north west] at (1.4, -0.05) {\scriptsize{3}};
			\node[anchor=north west] at (2.4, -0.05) {\scriptsize{4}};
			\node[anchor=north west] at (3.4, -0.05) {\scriptsize{5}};
			\node[anchor=north west] at (4.4, -0.05) {\scriptsize{6}};
			\node[anchor=north west] at (5.4, -0.05) {\scriptsize{7}};
			\node[anchor=north west] at (6.4, -0.05) {\scriptsize{8}};
			\node[anchor=north west] at (7.4, -0.05) {\scriptsize{9}};
			\node[anchor=north west] at (8.2, -0.05) {\scriptsize{10}};
			\node[anchor=north west] at (9.2, -0.05) {\scriptsize{11}};
			\node[anchor=north west] at (10.2, -0.05) {\scriptsize{12}};
			\node[anchor=north west] at (11.2, -0.05) {\scriptsize{13}};
			\node[anchor=north west] at (12.2, -0.05) {\scriptsize{14}};
			\node[anchor=north west] at (13.2, -0.05) {\scriptsize{15}};
			\node[anchor=north west] at (14.2, -0.05) {\scriptsize{16}};
		\end{tikzpicture}
		\hspace{-0.1cm}
		\raisebox{1ex}{$\xlongrightarrow{\Omega}$}
		\hspace{-0.3cm}
		\begin{tikzpicture}[scale=0.4,
			node/.style={circle, fill=black, inner sep=0pt, minimum size=2.5pt},
			node distance = 1.5cm,
			baseline={(current bounding box.center)}
			]
			\node[node, label=left:\scriptsize{$p_1$}] (p1) at (-0.5, -2) {};
			\node[node, label={[xshift=0.1cm]left:\scriptsize{$p_2$}}] (p2) at (2, -2) {};
			\node[node, label={[xshift=-0.1cm]below:{\scriptsize $p_3$}}] (p3) at (-2.5,-0.5) {};
			\node[node, label={[xshift=0.1cm]left:{\scriptsize $p_4$}}] (p4) at (-1,-0.5) {};
			\node[node, label={[xshift=0.1cm]left:{\scriptsize $p_5$}}] (p5) at (0.5, -0.5) {};
			\node[node, label={[yshift=-0.1cm]above:\scriptsize{$p_6$}}]  (p6) at (2, 0.5) {};
			\node[node, label=left:\scriptsize{$p_7$}] (p7) at (-0.5, 1.5) {};
			\node[node, label=right:\scriptsize{$p_8$}] (p8) at (1, 2.5) {};
			\draw (p1) -- (p3);
			\draw (p1) -- (p4);
			\draw (p1) -- (p5);
			\draw (p1) -- (p6);
			\draw (p2) -- (p6);
			\draw (p2) -- (p7);
			\draw (p3) -- (p7);
			\draw (p4) -- (p7);
			\draw (p5) -- (p8);
			\draw (p7) -- (p8);
			\node[anchor=north west] at (-1.2, -2.4) {\scriptsize{$P$}};
			\node[anchor=north west] at (-0.5,-2.6) {\rotatebox{-90}{$\xlongrightarrow{}$}};
			\node at (0.7,-3.7) {\scriptsize$\mathcal{O}$};
		\end{tikzpicture}
		\\
		\vspace{-0.2cm}
		\hspace{3.5cm}
		\begin{tikzpicture}
			\node[anchor=north west] at (4.5,-0.9) {$\alpha=(0,1,0,1,3,1,1,2)$};
			\node[anchor=north west] at (3.6,-0.8) {$\xlongleftarrow{\scriptstyle \Lambda}$};
			\node[anchor=north west] at (1,-1) {$\pi=31764825$};
		\end{tikzpicture}
		\vspace{-0.1cm}
		\caption{An example of the applications of the bijections between the four Fishburn structures related to Section~\ref{sec-bijection}.}
		\label{fig-exa-matching-perm}
	\end{figure}
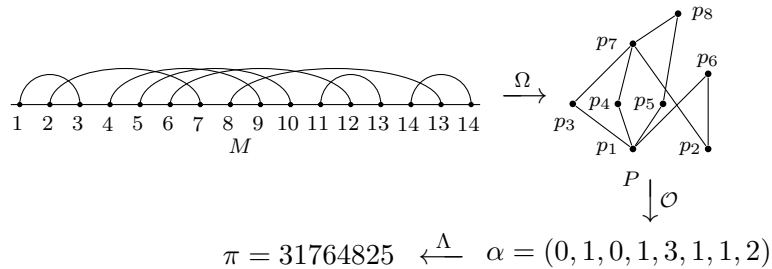    
	
	\begin{lem}[\cite{Duncan-Steingrimsson}, Theorem 2.6]\label{lem-cite-101}
		Any $\alpha \in \A_n(101)$ is a restricted growth function (RGF) and encodes precisely all noncrossing partitions.
	\end{lem}
	Lemma~\ref{lem-cite-101} states that the entries of $\alpha$ decompose into maximal weakly increasing runs, 
	whose minima exceed the preceding maxima by $1$, 
	separated by weakly decreasing segments whose elements are strictly smaller than the preceding maximum. 
	For example, for $\alpha = 00123314566744110889$, the weakly increasing runs are $001233$, $45667$, and $889$.
	
	\begin{thm}\label{thm-MnP2-Fn3142}
		When applied to $\M_n(P_2)$, the map $\Upsilon$ establishes a bijection with $\F_n(3142)$.
	\end{thm}
	\begin{proof}
		By Theorem~\ref{thm-P2-101}, the fact that $\A_n(0101)=\A_n(101)$ shown in \cite[Theorem 2.5]{Duncan-Steingrimsson}, 
		and the bijection $\T\colon \hat{\A}_n(0101) \to \F_n(3142)$ from~\cite[Table 2]{Cerbai-Claesson}, 
		it suffices to prove that $\Delta$ is a bijection from $\A_n(0101)$ onto $\hat{\A}_n(0101)$. 
		However, from the definition of $\Delta$ and Lemma~\ref{lem-cite-101}, 
		one can see that $\Delta(\alpha)=\alpha$ for any $\alpha \in \A_n(0101)$, 
		so that $\Delta$ restricted to $\A_n(0101)$ is the identity map, which means  a stronger statement holds, namely, $\A_n(0101) = \hat{\A}_n(0101)$.
		The proof is complete.
	\end{proof}
	
	\section{Statistics and their distributions over pattern-avoiding Fishburn structures}\label{sec-stats}
	In this section, we investigate statistics and distributions over Fishburn subsets enumerated by the Catalan numbers, which were considered in Section~\ref{sec-bijection}.
	In Section~\ref{subsec-decomp}, we present structural decompositions of these sets under their bijections.
	In Section~\ref{stats-St-matchings-subsec}, we establish  several equidistributed statistics and prove Theorems~\ref{thm-equi-cr-nr-P1}--\ref{thm-P1-P2-joint}.
	
	\subsection{The decompositions of Fishburn subsets enumerated by  Catalan numbers}\label{subsec-decomp}
	In this subsection, we introduce decompositions of the Fishburn subsets, providing combinatorial interpretations of the convolution recurrence (\ref{Catalan-rec}) for the Catalan numbers.  Conversely, we also illustrate how two smaller objects can be combined to form a larger one. Note that the decomposition for $M \in \mathcal{M}_n(P_2)$ has already been discussed in Section~\ref{P2-sec}.
	
	\textbf{The decomposition of $ M \in \mathcal{M}_n(P_1)$.} 
	Given a matching $M \in \mathcal{M}_n(P_1)$, let $M^*$ be the first irreducible block of $M$, and let $M_2$ be the matching consisting of the remaining arcs. Clearly, both $M^*$ and $M_2$ are $P_1$-avoiding.
	We construct a matching $M_1$ in the following way. Remove the first opener and the last closer from $M^*$. Among the remaining vertices, pair the openers and closers in their natural left-to-right order; that is, pair the $i$th remaining opener with the $i$th remaining closer for $1\le i\le |M^*|-1$.
	Finally, relabel the remaining vertices increasingly by \(1,2,\ldots,2(k-1)\), or equivalently, take the reduced form of the resulting matching.
	The resulting matching $M_1$ is nonnesting and therefore, by Lemma~\ref{lem-Nonnesting}, is $P_1$-avoiding.
	See Figure~\ref{fig-exa-P1-decomposition} for an example.
	
	\begin{figure}[ht]
		\begin{center}
			\begin{tikzpicture}[scale = 0.3]
				\draw (-0.6,0) -- (11.6,0);
				\foreach \x in {0,1,2,3,4,5,6,7,8,9,10,11}
				{
					\filldraw (\x,0) circle (2pt);
				}
				\draw[black] (0,0) arc (180:0:1 and 1);
				\draw[black] (1,0) arc (180:0:2 and 1);
				\draw[black] (3,0) arc (180:0:1.5 and 1);
				\draw[black] (4,0) arc (180:0:1.5 and 1);
				\draw[black] (8,0) arc (180:0:0.5 and 1);
				\draw[black] (10,0) arc (180:0:0.5 and 1);
				\node[below left] at (6,2.6) {\scriptsize{$M$}};
				\node[anchor=north west] at (-0.8, -0.05) {\scriptsize{1}};
				\node[anchor=north west] at (0.2, -0.05) {\scriptsize{2}};
				\node[anchor=north west] at (1.2, -0.05) {\scriptsize{3}};
				\node[anchor=north west] at (2.2, -0.05) {\scriptsize{4}};
				\node[anchor=north west] at (3.3, -0.05) {\scriptsize{5}};
				\node[anchor=north west] at (4.3, -0.05) {\scriptsize{6}};
				\node[anchor=north west] at (5.2, -0.05) {\scriptsize{7}};
				\node[anchor=north west] at (6.2, -0.05) {\scriptsize{8}};
				\node[anchor=north west] at (7.2, -0.05) {\scriptsize{9}};
				\node[anchor=north west] at (8, -0.05) {\scriptsize{10}};
				\node[anchor=north west] at (9.2, -0.05) {\scriptsize{11}};
				\node[anchor=north west] at (10.4, -0.05) {\scriptsize{12}};
			\end{tikzpicture}
			\hspace{-0.35cm}
			\raisebox{3ex}{$\quad$}
			\begin{tikzpicture}[scale = 0.3]
				\draw (-0.4,0) -- (7.6,0);
				\foreach \x in {1,2,3,4,5,6}
				{
					\filldraw (\x,0) circle (2pt);
				}
				\draw[thick] (0,0) circle (6pt); 
				\draw[thick] (0,0) circle (2pt);
				
				\draw[thick] (7,0) circle (6pt); 
				\draw[thick] (7,0) circle (2pt);
				
				\draw[black] (0,0) arc (180:0:1 and 1);
				\draw[black] (1,0) arc (180:0:2 and 1);
				\draw[black] (3,0) arc (180:0:1.5 and 1);
				\draw[black] (4,0) arc (180:0:1.5 and 1);
				\node[below left] at (4.3,2.7) {\scriptsize{$M^*$}};
				\node[anchor=north west] at (-0.8, -0.05) {\scriptsize{1}};
				\node[anchor=north west] at (0.2, -0.05) {\scriptsize{2}};
				\node[anchor=north west] at (1.2, -0.05) {\scriptsize{3}};
				\node[anchor=north west] at (2.2, -0.05) {\scriptsize{4}};
				\node[anchor=north west] at (3.3, -0.05) {\scriptsize{5}};
				\node[anchor=north west] at (4.3, -0.05) {\scriptsize{6}};
				\node[anchor=north west] at (5.2, -0.05) {\scriptsize{7}};
				\node[anchor=north west] at (6.2, -0.05) {\scriptsize{8}};
			\end{tikzpicture}
			\hspace{-0.35cm}
			\raisebox{3ex}{$\quad$}
			\begin{tikzpicture}[scale = 0.3]
				\draw (-0.6,0) -- (3.6,0);
				\foreach \x in {0,1,2,3}
				{
					\filldraw (\x,0) circle (2pt);
				}
				\draw[black] (0,0) arc (180:0:0.5 and 1);
				\draw[black] (2,0) arc (180:0:0.5 and 1);
				\node[below left] at (2.8,2.5) {\scriptsize{$M_2$}};
				\node[anchor=north west] at (-0.8, -0.05) {\scriptsize{1}};
				\node[anchor=north west] at (0.2, -0.05) {\scriptsize{2}};
				\node[anchor=north west] at (1.2, -0.05) {\scriptsize{3}};
				\node[anchor=north west] at (2.2, -0.05) {\scriptsize{4}};
			\end{tikzpicture}
			\raisebox{3ex}{$\quad$}
			\begin{tikzpicture}[scale = 0.3]
				\draw (-0.6,0) -- (5.6,0);
				\foreach \x in {0,1,2,3,4,5}
				{
					\filldraw (\x,0) circle (2pt);
				}
				\draw[black] (0,0) arc (180:0:0.5 and 1);
				\draw[black] (2,0) arc (180:0:1 and 1);
				\draw[black] (3,0) arc (180:0:1 and 1);
				\node[below] at (3,2.6) {\scriptsize{$M_1$}};
				\node[anchor=north west] at (-0.8, -0.05) {\scriptsize{1}};
				\node[anchor=north west] at (0.2, -0.05) {\scriptsize{2}};
				\node[anchor=north west] at (1.2, -0.05) {\scriptsize{3}};
				\node[anchor=north west] at (2.2, -0.05) {\scriptsize{4}};
				\node[anchor=north west] at (3.3, -0.05) {\scriptsize{5}};
				\node[anchor=north west] at (4.3, -0.05) {\scriptsize{6}};
			\end{tikzpicture}
			\vspace{-0.45cm}
		\end{center}
		\caption{An example of the decomposition of $M \in \mathcal{M}_{6}(P_1)$.}
		\label{fig-exa-P1-decomposition}
	\end{figure}
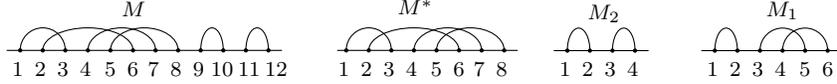
	Conversely, given $M_1 = \{[a_1, b_1], [a_2, b_2], \ldots, [a_{k-1}, b_{k-1}]\} \in \mathcal{M}_{k-1}(P_1)$ and $M_2 \in \mathcal{M}_{n-k}(P_1)$, we reconstruct $M \in \mathcal{M}_n(P_1)$ as follows.  
	Let
	\[
	M^* = \left\{ [ 1, b_1+1], [a_1+1, b_2+1], \ldots, [a_{k-1}+1, 2k] \right\}
	\]
	be the matching obtained by inserting a new opener immediately to the left of $1$ and a new closer immediately to the right of $b_{k-1}$, then pairing the $i$-th opener with the $i$-th closer in order from left to right. 
	We denote this operation by $\Theta$, so that $\Theta(M_1) = M^*$.  
	Clearly, $M^* \in \mathcal{M}_{k}(P_1)$. We claim that $M^*$ is irreducible.  Indeed, suppose there exists a leftmost irreducible block 
	\[
	M^j = \{[1, b_1+1], [2, b_2+1], \ldots, [a_j+1, b_{j+1}+1]\}.
	\] 
	This leads to a contradiction with the condition $a_j < a_{j+1} < b_{j+1}$.  Finally, we set $M = M^* \oplus M_2$.
	
	Observe that the map $\Theta : \mathcal{M}_{k-1}(P_1) \to \mathcal{M}^{Ir}_k(P_1)$ is a bijection.  
	Hence, any $M \in \mathcal{M}_n(P_1)$ can be decomposed as $\Theta(M_1) \oplus M_2$, where $M_1 \in \mathcal{M}_{k-1}(P_1)$ and $M_2 \in \mathcal{M}_{n-k}(P_1)$ are uniquely determined.  
	Moreover, the bijection $\Theta$, together with Theorem~\ref{thm-Catalan}, yields the following corollary.
	\begin{coro}\label{lem-P1-decomposition}
		We have $|\M_{n}^{Ir}(P_1)|=C_{n-1}$ for any $n\geq 1$. 
	\end{coro}
	
	\textbf{The decomposition of $P \in \PP_n\bf{(3+1)}$.} 
	Following~\cite{Bousquet-Claesson-Dukes-Kitaev}, given a poset \( P = (X, \le ) \in \mathcal{P}_n\), the \textit{strict down-set} of \( x \in P\) is  $D(x) = \{ y \, |\,   y < x \}$. Similarly, the \textit{strict up-set} of $x$ is $U(x)=\{y \, |\,  y > x\}$. 
	It is known that a poset is \textbf{(2+2)}-free if and only if its collection of strict down-sets can be linearly ordered by inclusion
	$\emptyset =D_0 \subset  D_1 \cdots \subset D_t$.
	An element \(x\) is said to be at \textit{level} \(i\) in \(P\), denoted \(\ell(x) = i\), if \(D(x) = D_i\). 
	
	The {\it ordinal sum} of two posets $R$ and $Q$, denoted $R \oplus Q$, is defined on their disjoint union and partially ordered by $x \le y$  if  (a) $x,y \in R$ and $x \le_R y$, or (b) $x,y \in Q$ and $x \le_Q y$, or (c) $x \in R$ and $y \in Q$.
	A poset $P$ is \emph{irreducible} if it cannot be expressed as the ordinal sum of two nonempty posets.
	Each poset has a unique {\it ordinal sum decomposition} $P=P^{(1)} \oplus P^{(2)} \oplus \cdots \oplus P^{(k)}$ with $P^{(i)}$ irreducible.
	An example of such a decomposition is shown in Figure~\ref{fig-ordinal decomposition}.
	
	\begin{figure}[h]
		\begin{center}
			\begin{tikzpicture}[scale=0.4,
				node/.style={circle, fill=black, inner sep=0pt, minimum size=2.5pt},
				node distance = 1.5cm,
				baseline={(current bounding box.center)}
				]
				\node[node, label=left:\scriptsize{}] (x0) at (0.5, -2) {};
				\node[node, label=left:\scriptsize{}] (x1) at (-1.5, -2) {};
				\node[node, label=right:\scriptsize{}] (x2) at (-0.5, -1.25) {};
				\node[node, label=left:\scriptsize{}] (x3) at (1.5, -1) {};
				\node[node, label=left:\scriptsize{}] (x4) at (-1.5, -0.5) {};
				\node[node, label=right:\scriptsize{}] (x5) at (1.5, 1) {};
				\node[node, label=left:\scriptsize{}] (x6) at (-1.5, 1) {};
				\node[node, label=left:\scriptsize{}] (x7) at (-1.5, 2) {};
				\node[node, label=left:\scriptsize{}] (x8) at (-2.5, 2) {};
				\node[node, label=left:\scriptsize{}] (x9) at (-1.5, 3) {};
				\draw (x0) -- (x2);
				\draw (x0) -- (x3);
				\draw (x1) -- (x4);
				\draw (x2) -- (x4);
				\draw (x3) -- (x5);
				\draw (x3) -- (x6);
				\draw (x4) -- (x5);
				\draw (x4) -- (x6);
				\draw (x5) -- (x9);
				\draw (x6) -- (x7);
				\draw (x6) -- (x8);
				\draw (x7) -- (x9);
				\draw (x8) -- (x9);
				\node[anchor=north west] at (-1, -1.8) {\scriptsize{$P$}};
			\end{tikzpicture}
			\hspace{0.3cm}
			\text{$=$}
			\begin{tikzpicture}[scale=0.4,
				node/.style={circle, fill=black, inner sep=0pt, minimum size=2.5pt},
				node distance = 1.5cm,
				baseline={(current bounding box.center)}
				]
				\node[node, label=right:\scriptsize{}] (x0) at (0, -2) {};
				\node[node, label=left:\scriptsize{}] (x1) at (-1.5, -2) {};
				\node[node, label=right:\scriptsize{}] (x2) at (0,-0.5) {};
				\node[node, label={[right] \scriptsize{}}] (x3) at (1.5, -1) {};
				\node[node, label=left:\scriptsize{}] (x4) at (0, 1) {};
				
				\draw (x0) -- (x2);
				\draw (x0) -- (x3);
				\draw (x1) -- (x4);
				\draw (x2) -- (x4);
				\node[anchor=north west] at (-1, -2.4) {\scriptsize{$P^{(1)}$}};
			\end{tikzpicture}
			\text{$\oplus$}
			\begin{tikzpicture}[scale=0.4,
				node/.style={circle, fill=black, inner sep=0pt, minimum size=2.5pt},
				node distance = 1.5cm,
				baseline={(current bounding box.center)}
				]
				\node[node, label=left:\scriptsize{}] (x0) at (1, -2) {};
				\node[node, label=left:\scriptsize{}] (x1) at (0, -2) {};
				\node[node, label=left:\scriptsize{}] (x2) at (-1, -1) {};
				\node[node, label=left:\scriptsize{}] (x3) at (1, -1) {};
				\draw (x1) -- (x2);
				\draw (x1) -- (x3);
				\node[anchor=north west] at (-0.8, -2.2) {\scriptsize{$P^{(2)}$}};
			\end{tikzpicture}
			\hspace{0.3cm}
			\text{$\oplus$}
			\begin{tikzpicture}[scale=0.4,
				node/.style={circle, fill=black, inner sep=0pt, minimum size=2.5pt},
				node distance = 1.5cm,
				baseline={(current bounding box.center)}
				]
				\node[node, label=left:\scriptsize{}] (x1) at (-1.5, -2) {};
				\node[anchor=north west] at (-2.2, -2) {\scriptsize{$P^{
							(3)}$}};
			\end{tikzpicture}
			\caption{An example of the ordinal sum decomposition of $P$.}
			\label{fig-ordinal decomposition}
		\end{center}
	\end{figure}
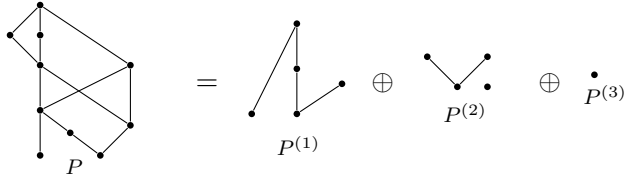
	
	For $P\in \PP_n\textbf{(3+1)}$,
	the decomposition of $P$ is based on the bijection $\Omega$  described in Section~\ref{subsec-matching-poset}.
	Let $M=\Omega^{-1}(P)$ and write its decomposition as $M=M^*\oplus M_2$ with $M^*=\Theta(M_1)$.
	It follows that $P^* = \Omega(M^*)$ is the first irreducible subposet of $P$, while the subposet formed by the remaining components is $P_{\text{II}} = \Omega(M_2)$.
	We next label the elements of $ P^* $ as $ p^*_0$, $p^*_1, \dots, p^*_{k-1}$ in order of weakly increasing level, and within each level, arrange them in weakly descending order of the size of their up-sets. 
	Define the order on the set $ \{p_1, \dots, p_{k-1}\}$ of elements of $ P_{\text{I}} =\Omega(M_1)$  by setting 
	$p_i < p_j $ in $P_{\text{I}}$ if $p^*_{i-1} < p^*_j$ in $ P^* $. 
	Then,  we have 
	\(
	\ell(p_1)  \leq \ell(p_2) \leq \cdots \leq \ell(p_{k-1})
	\) in $ P_{\text{I}}$.
	See Figure~\ref{fig-P^*-P1} for an example.
	
	\begin{figure}[h]
		\begin{center}
			\begin{tikzpicture}[scale=0.4,
				node/.style={circle, fill=black, inner sep=0pt, minimum size=2.5pt},
				node distance = 1.5cm,
				baseline={(current bounding box.center)}
				]
				\node[node, label=left:\scriptsize{}] (x0) at (0.5, -2) {};
				\node[node, label=left:\scriptsize{}] (x1) at (-1.5, -2) {};
				\node[node, label=right:\scriptsize{}] (x2) at (-0.5, -1.25) {};
				\node[node, label=left:\scriptsize{}] (x3) at (1.5, -1) {};
				\node[node, label=left:\scriptsize{}] (x4) at (-1.5, -0.5) {};
				\node[node, label=right:\scriptsize{}] (x5) at (1.5, 1) {};
				\node[node, label=left:\scriptsize{}] (x6) at (-1.5, 1) {};
				\node[node, label=left:\scriptsize{}] (x7) at (-1.5, 2) {};
				\node[node, label=left:\scriptsize{}] (x8) at (-2.5, 2) {};
				\node[node, label=left:\scriptsize{}] (x9) at (-1.5, 3) {};
				\draw (x0) -- (x2);
				\draw (x0) -- (x3);
				\draw (x1) -- (x4);
				\draw (x2) -- (x4);
				\draw (x3) -- (x5);
				\draw (x3) -- (x6);
				\draw (x4) -- (x5);
				\draw (x4) -- (x6);
				\draw (x5) -- (x9);
				\draw (x6) -- (x7);
				\draw (x6) -- (x8);
				\draw (x7) -- (x9);
				\draw (x8) -- (x9);
				\node[anchor=north west] at (-1, -1.8) {\scriptsize{$P$}};
			\end{tikzpicture}
			\hspace{0.3cm}
			\begin{tikzpicture}[scale=0.4,
				node/.style={circle, fill=black, inner sep=0pt, minimum size=2.5pt},
				node distance = 1.5cm,
				baseline={(current bounding box.center)}
				]
				\node[node, label=right:\scriptsize{$p^*_0$}] (x0) at (0, -2) {};
				\node[node, label=left:\scriptsize{$p^*_1$}] (x1) at (-1.5, -2) {};
				\node[node, label=right:\scriptsize{$p^*_2$}] (x2) at (0,-0.5) {};
				\node[node, label={[right] \scriptsize{$p^*_3$}}] (x3) at (1.5, -1) {};
				\node[node, label=left:\scriptsize{$p^*_4$}] (x4) at (0, 1) {};
				
				\draw (x0) -- (x2);
				\draw (x0) -- (x3);
				\draw (x1) -- (x4);
				\draw (x2) -- (x4);
				\node[anchor=north west] at (-0.3, -2.4) {\scriptsize{$P^*$}};
			\end{tikzpicture}
			\hspace{0.5cm}
			\begin{tikzpicture}[scale=0.4,
				node/.style={circle, fill=black, inner sep=0pt, minimum size=2.5pt},
				node distance = 1.5cm,
				baseline={(current bounding box.center)}
				]
				\node[node, label=left:\scriptsize{$p_1$}] (x1) at (0.5, -2) {};
				\node[node, label=left:\scriptsize{$p_2$}] (x4) at (-1, -0.5) {};
				\node[node, label=right:\scriptsize{$p_3$}] (x3) at (2, -0.5) {};
				\node[node, label=right:\scriptsize{$p_4$}] (x5) at (0.5, 1) {};
				\draw (x4) -- (x5);
				\draw (x1) -- (x3);
				\draw (x1) -- (x4);
				\draw (x3) -- (x5);
				\node[anchor=north west] at (0, -2.4) {\scriptsize{$P_{\text{I}}$}};
			\end{tikzpicture}
			\hspace{0.3cm}
			\begin{tikzpicture}[scale=0.4,
				node/.style={circle, fill=black, inner sep=0pt, minimum size=2.5pt},
				node distance = 1.5cm,
				baseline={(current bounding box.center)}
				]
				\node[node, label=left:\scriptsize{}] (x0) at (1.5, -2) {};
				\node[node, label=left:\scriptsize{}] (x1) at (0, -2) {};
				\node[node, label=left:\scriptsize{}] (x2) at (-1, -0.5) {};
				\node[node, label=right:\scriptsize{}] (x3) at (0, -0.5) {};
				\node[node, label=right:\scriptsize{}] (x4) at (0, 1) {};
				\draw (x0) -- (x4);
				\draw (x2) -- (x4);
				\draw (x1) -- (x3);
				\draw (x1) -- (x2);
				\draw (x3) -- (x4);
				\node[anchor=north west] at (-0.8, -2) {\scriptsize{$P_{\text{II}}$}};
			\end{tikzpicture}
			\caption{An example of the decomposition of $P \in \mathcal{P}_{10}\textbf{(3+1)}$.}
			\label{fig-P^*-P1}
		\end{center}
	\end{figure}
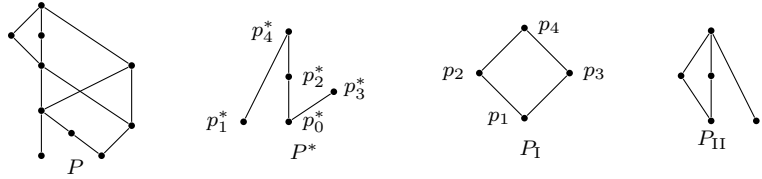

	Conversely, for any given $P_\text{I}\in \PP_{k-1}\bf{(3+1)}$ and $P_\text{II} \in \PP_{n-k}\bf{(3+1)}$, we first label the elements $p_1$, $p_2,\dots,p_{k-1}$ of $P_{\text{I}}$ in the same way as we labeled the elements of $P^*$ above. 
	Then the elements $p^*_0$, $p^*_1,\dots,p^*_{k-1}$ in $P^*$ are recovered by reversing the previous rule: $p^*_{i-1}<p^*_j$ if $p_i<p_j$ in $P_{\text{I}}$. Note that the elements $p^*_0$ and $p^*_1$, as well as the elements $p^*_{k-1}$ and $p^*_k$, are incomparable. 
	Finally, we let $P= P^* \oplus P_{\text{II}}$.

	\textbf{The decomposition of $P \in \PP_n\bf{(N)}$.} 
	For a poset $P\in\mathcal{P}_n\mathbf{(N)}$ we set $M=\Omega^{-1}(P)=\widetilde{M}\oplus M''$ with $\widetilde{M}=\mathcal{V}(M')$, where $\mathcal{V}$ is detailed in Section~\ref{P2-sec}.
	Thus, $P$ decomposes into its first irreducible subposet $\widetilde{P} = \Omega(\widetilde{M})$ and the subposet of remaining components $P'' = \Omega(M'')$. The poset $P'=\Omega(M')$ is obtained by removing an isolated element from $\widetilde{P}$, as illustrated in Figure~\ref{fig-P'-P''-N}. Observe that an isolated element in $\widetilde{P}$ always exists, since there is an isolated element corresponding to the reduction arc in $\widetilde{M}$.
	
	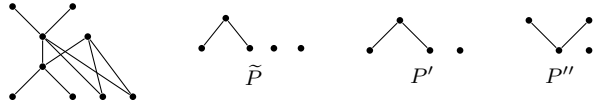
\begin{figure}[h]
		\begin{center}
			\begin{tikzpicture}[scale=0.4,
				node/.style={circle, fill=black, inner sep=0pt, minimum size=2.5pt},
				node distance = 1.5cm,
				baseline={(current bounding box.center)}
				]
				\node[node, label=left:\scriptsize{}] (x0) at (3, -1) {};
				\node[node, label=left:\scriptsize{}] (x1) at (-1, -1) {};
				\node[node, label=left:\scriptsize{}] (x2) at (1, -1) {};
				\node[node, label=left:\scriptsize{}] (x3) at (2, -1) {};
				\node[node, label=left:\scriptsize{}] (x4) at (0, 0) {};
				\node[node, label=left:\scriptsize{}] (x5) at (0, 1) {};
				\node[node, label=left:\scriptsize{}] (x6) at (-1, 2) {};
				\node[node, label=left:\scriptsize{}] (x7) at (1,2) {};
				\node[node, label=left:\scriptsize{}] (x8) at (1.5,1) {};
				\draw (x0) -- (x5);
				\draw (x1) -- (x4);
				\draw (x2) -- (x4);
				\draw (x5) -- (x4);
				\draw (x3) -- (x5);
				\draw (x5) -- (x6);
				\draw (x5) -- (x7);
				\draw (x0) -- (x8);
				\draw (x4) -- (x8);
				\draw (x3) -- (x8);
			\end{tikzpicture}
			\hspace{0.3cm}              
			\begin{tikzpicture}[scale=0.4,
				node/.style={circle, fill=black, inner sep=0pt, minimum size=2.5pt},
				node distance = 1.5cm,
				baseline={(current bounding box.center)}
				]
				\node[node, label=left:\scriptsize{}] (x0) at (2.5, -1) {};
				\node[node, label=left:\scriptsize{}] (x1) at (-0.8, -1) {};
				\node[node, label=left:\scriptsize{}] (x2) at (0.8, -1) {};
				\node[node, label=left:\scriptsize{}] (x3) at (0, 0) {};
				\node[node, label=below:\scriptsize{}] (x4) at (1.6, -1) {};
				\draw (x2) -- (x3);
				\draw (x1) -- (x3);
				\node[anchor=north west] at (0.3, -1.2) {\scriptsize{$\widetilde{P}$}};
			\end{tikzpicture}
			\hspace{0.3cm}
			\begin{tikzpicture}[scale=0.4,
				node/.style={circle, fill=black, inner sep=0pt, minimum size=2.5pt},
				node distance = 1.5cm,
				baseline={(current bounding box.center)}
				]
				\node[node, label=left:\scriptsize{}] (x0) at (2, -1) {};
				\node[node, label=left:\scriptsize{}] (x1) at (-1, -1) {};
				\node[node, label=left:\scriptsize{}] (x2) at (1, -1) {};
				\node[node, label=left:\scriptsize{}] (x3) at (0, 0) {};
				\draw (x2) -- (x3);
				\draw (x1) -- (x3);
				\node[anchor=north west] at (0, -1.2) {\scriptsize{$P'$}};
			\end{tikzpicture}
			\hspace{0.3cm}
			\begin{tikzpicture}[scale=0.4,
				node/.style={circle, fill=black, inner sep=0pt, minimum size=2.5pt},
				node distance = 1.5cm,
				baseline={(current bounding box.center)}
				]
				\node[node, label=left:\scriptsize{}] (x0) at (1, -1) {};
				\node[node, label=left:\scriptsize{}] (x1) at (0, -1) {};
				\node[node, label=left:\scriptsize{}] (x2) at (-1, 0) {};
				\node[node, label=left:\scriptsize{}] (x3) at (1, 0) {};
				\draw (x1) -- (x2);
				\draw (x1) -- (x3);
				\node[anchor=north west] at (-0.8, -1.2) {\scriptsize{$P''$}};
			\end{tikzpicture}
			\caption{An example of the decomposition of $P \in \mathcal{P}_9(\mathbf{N})$.}
			\label{fig-P'-P''-N}
		\end{center}
	\end{figure}

	Conversely, for any $P' \in \mathcal{P}_{k-1}(\mathbf{N})$ and $P'' \in \mathcal{P}_{n-k}(\mathbf{N})$, 
	we obtain $\widetilde{P}$ from $P'$ by adding an isolated element, and then set $P = \widetilde{P} \oplus P''$.
	
	\textbf{The decomposition of $\alpha \in \A_n(101)$.} 
	Let $\alpha=(\alpha_1, \alpha_2, \dots,\alpha_n)\in \A_n(101)$.
	Based on the bijection~$\Psi$ mentioned in Section~\ref{sec-matching-101}, set $M=\Psi^{-1}(\alpha)$, which decomposes into $\widetilde{M}\oplus M'' \in \M_{n}(P_{2})$ with $\widetilde{M}=\V(M')$, where $M'\in \M_{k-1}(P_2)$ and $M'' \in \M_{n-k}(P_2)$.
	We claim that $k$ is the largest index such that $\alpha_k=0$.
	This is because the first irreducible block $\widetilde{M}$ is mapped to the contiguous subsequence of $\alpha$ that starts with the first entry and ends with 0.
	Then
	\begin{align*}
		\Psi(\widetilde M)& =\tilde\alpha=(\alpha_1, \alpha_2,  \dots, \alpha_{k-1}, 0),\\
		\Psi(M')& =\alpha'=(\alpha_1, \alpha_2,  \dots, \alpha_{k-1}),
		\text{ and } \\
		\Psi(M'')& =\alpha''=(\alpha_{k+1}-m,\alpha_{k+2}-m,\dots,\alpha_{n}-m),
	\end{align*}
	where $m=\max(\alpha')+1$ and, by Lemma~\ref{lem-cite-101}, also $m=\alpha_{k+1}$.
    If $k=n$, then $M''$ is empty and $\alpha''$ is the empty sequence.
	For example, for $\alpha=01002232$, we have  $\widetilde{\alpha}=0100$, $\alpha'=010$, and $\alpha''=0010$.
	
	Conversely, given $\alpha'\in\A_{k-1}(101)$ and $\alpha''\in\A_{n-k}(101)$, form $\tilde\alpha$ by appending $0$ to $\alpha'$,  add $\max(\widetilde{\alpha})+1$ to every entry of $\alpha''$, and concatenate the result to the right of $\widetilde{\alpha}$ to obtain $\alpha$.
	
	\textbf{The decomposition of $\pi \in \F_n(3142)$.}
	Let $\pi = \pi_1\pi_2\cdots\pi_n \in \F_n$. The decomposition of $\pi$ relies on the bijection $\Lambda$ described in Section~\ref{subsec-matching-perm}. Suppose $\pi_1 = k$, and let $\alpha = \Lambda^{-1}(\pi)$. Then decompose $\alpha$ into $\alpha' \in \mathcal{A}_{k-1}(101)$ and $\alpha'' \in \mathcal{A}_{n-k}(101)$.
	Applying $\Lambda$ to each component, we obtain that
	\[
	\pi' = \Lambda(\alpha') = \pi_2\cdots\pi_k
	\quad\text{and}\quad
	\pi'' = \Lambda(\alpha'') = (\pi_{k+1} - k)\cdots(\pi_n - k).
	\]
	We define $\widetilde{\pi} = k\pi_2\cdots\pi_k = k\pi'$, so that $\pi = \widetilde{\pi} \oplus \pi''$, where ``$\oplus$'' denotes the direct sum of permutations.
	For example, if $\pi = 41328657$, then $k = 4$, and we have $\widetilde{\pi} = 4132$, $\pi' = 132$, and $\pi'' = 4213$. We note that the decomposition of ascent sequences induced by $\Psi$, together with the decomposition of permutations induced by $\Lambda$, defines a decomposition of permutations induced by $\Upsilon = \Lambda \circ \Psi$.
	
	Conversely, given $\pi' \in \F_{k-1}(3142)$ and $\pi'' \in \F_{n-k}(3142)$, define $\widetilde{\pi} = k\pi'$ and set $\pi = \widetilde{\pi} \oplus \pi''$. Then $\pi$ is a permutation in $\F_n(3142)$ with $\pi_1 = k$.
	
	\subsection{Statistics over pattern-avoiding Fishburn structures}\label{stats-St-matchings-subsec}
	This subsection is dedicated to the statistics on four pattern-avoiding Fishburn structures, leading to the proofs of Theorems~\ref{thm-equi-cr-nr-P1}--\ref{thm-P1-P2-joint}.
	
	We first introduce the statistics over \textbf{(2+2)}-free posets, ascent sequences, and Fishburn permutations.
	Given a poset \( P = (X, \le ) \in \mathcal{P}_n\), 
	the {\it height}, denoted $\h(P)$, is the maximum size
	of a chain, while the {\it width}, denoted $\wid(P)$, is the maximum size
	of an antichain. Let $\smc(P)$ be the  size of the shortest maximal chain.
	An element $y$ is said to be a minimal element if $\ell(y)=0$.
	Let $\minn(P)$ be the number of minimal elements in $P$. 
	Following Fishburn \cite{Fishburn-1983,Fishburn-1985}, the \emph{magnitude} of $P$, denoted $\magg(P)$, is defined as the number of distinct strict down-sets, or equivalently, the number of maximal antichains. Its distribution over $\mathcal{P}_n$ was studied by Jel\'{\i}nek~\cite{Jelinek}. 
	We also let $\ssd(P)$ denote the number of subposets in the ordinal sum decomposition of $P$.
	For example, for the leftmost poset in Figure~\ref{fig-P'-P''-N}, we have $\h(P) = \wid(P)=\magg(P)=\minn(P) = 4$ and $\smc(P)=\ssd(P) = 2$.
	
	Given a sequence $\alpha = (\alpha_1, \alpha_2, \ldots, \alpha_n) \in \A_n$, we consider the following statistics:
	\begin{itemize}
		\item $\lmax(\alpha)$ = the number of left-to-right maxima  in $\alpha$ = $\#\{i\ |\ \alpha_i> \alpha_j \text{ for all } j<i\}$, 
		\item $\rmin(\alpha)$ = the number of right-to-left minima  in $\alpha$ = $\#\{i\ |\ \alpha_i< \alpha_j \text{ for all } j>i\}$,
		\item $\zero(\alpha)$ = the number of zeros in $\alpha$.
	\end{itemize}
	
	Given a permutation $\pi=\pi_1\pi_2\cdots\pi_n \in \mathcal{F}_n$, a \emph{descending run} is a maximal consecutive subsequence $\pi_i\pi_{i+1}\cdots\pi_j$ satisfying $\pi_i>\pi_{i+1}>\cdots>\pi_j$. Let $\idr(\pi)$ be the length of the {\it initial descending run} in $\pi$. We also define $\lmax(\pi)$ and $\rmin(\pi)$ in the same way as for ascent sequences.
	
	Kitaev and Remmel~\cite{Kitaev-Remmel} proved equidistribution of the following statistics:
	\begin{itemize}
		\item $\minn$ over $\PP_n$, $\zero$ over $\A_n$, and $\idr$ over $\F_n$;
		\item $\magg(P)$ over $\PP_n$ and $\asc+1$ over $\A_n$.
	\end{itemize}
	Chen, Yan, and Zhou~\cite{Chen-Yan-Zhou} demonstrated that the left-to-right maxima and right-to-left maxima have a symmetric joint distribution over $\F_n$, 
	and that the right-to-left maxima over $\F_n$ are equidistributed with the right-to-left minima over $\A_n$.
	
	\subsubsection{The proof of Theorem~\ref{thm-equi-cr-nr-P1}}\label{proof-thm-1-3}
	
	Our proof relies on a well-known bijection between Dyck paths and nonnesting matchings (equivalently, $\M(P_1)$); see, for example, \cite{Stanley}. The set of \emph{Dyck paths} of semilength~$ n $, denoted $ \mathcal{D}_n $, consists of lattice paths from $ (0, 0) $ to $ (2n, 0) $ composed of up-steps $ U = (1, 1) $ and down-steps $ D = (1, -1) $ that never fall below the $ x $-axis. We recall the bijection $ \Gamma \colon \mathcal{D}_n \to \mathrm{NN}_n $, where $ \mathrm{NN}_n $ denotes the set of nonnesting matchings on $ [2n] $. Given a Dyck path $ \mu $, the matching $ \Gamma(\mu) $ is obtained by pairing the $i$-th up-step with the corresponding $i$-th down step.
	Figure~\ref{fig-exa-dyck-m} shows an example of the application of~$\Gamma$.
	
	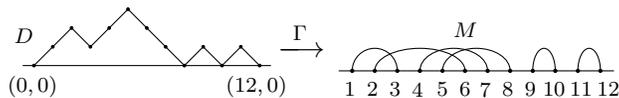
\begin{figure}[h]
		\begin{center}
			\begin{tikzpicture}[scale=0.25]
				\draw (-0.6,0) -- (12.6,0);
				\coordinate (start) at (0,0);
				\draw (start)--(2,2)--(3,1)--(4,2)--(5,3)--(6,2)--(7,1)--(8,0)--(9,1)--(10,0)--(11,1)--(12,0);
				\filldraw (0,0) circle (2pt);
				\filldraw (1,1) circle (2pt);
				\filldraw (2,2) circle (2pt);
				\filldraw (3,1) circle (2pt);
				\filldraw (4,2) circle (2pt);
				\filldraw (5,3) circle (2pt);
				\filldraw (6,2) circle (2pt);
				\filldraw (7,1) circle (2pt);
				\filldraw (8,0) circle (2pt);
				\filldraw (9,1) circle (2pt);
				\filldraw (10,0) circle (2pt);
				\filldraw (11,1) circle (2pt);
				\filldraw (12,0) circle (2pt);
				\node[below left] at (0.5,2.5) {\scriptsize{$D$}};
				\node[below left] at (1.8,0) {\scriptsize{$(0,0)$}};
				\node[below left] at (14,0) {\scriptsize{$(12,0)$}};
			\end{tikzpicture}
			\hspace{-0.5cm}
			\raisebox{4ex}{$\xlongrightarrow{\Gamma}$}
			\begin{tikzpicture}[scale = 0.3]
				\draw (-0.6,0) -- (11.6,0);
				\foreach \x in {0,1,2,3,4,5,6,7,8,9,10,11}
				{
					\filldraw (\x,0) circle (2pt);
				}
				\draw[black] (0,0) arc (180:0:1 and 1);
				\draw[black] (1,0) arc (180:0:2 and 1);
				\draw[black] (3,0) arc (180:0:1.5 and 1);
				\draw[black] (4,0) arc (180:0:1.5 and 1);
				\draw[black] (8,0) arc (180:0:0.5 and 1);
				\draw[black] (10,0) arc (180:0:0.5 and 1);
				\node[below left] at (6,2.6) {\scriptsize{$M$}};
				\node[anchor=north west] at (-0.8, -0.05) {\scriptsize{1}};
				\node[anchor=north west] at (0.2, -0.05) {\scriptsize{2}};
				\node[anchor=north west] at (1.2, -0.05) {\scriptsize{3}};
				\node[anchor=north west] at (2.2, -0.05) {\scriptsize{4}};
				\node[anchor=north west] at (3.3, -0.05) {\scriptsize{5}};
				\node[anchor=north west] at (4.3, -0.05) {\scriptsize{6}};
				\node[anchor=north west] at (5.2, -0.05) {\scriptsize{7}};
				\node[anchor=north west] at (6.2, -0.05) {\scriptsize{8}};
				\node[anchor=north west] at (7.2, -0.05) {\scriptsize{9}};
				\node[anchor=north west] at (8, -0.05) {\scriptsize{10}};
				\node[anchor=north west] at (9.2, -0.05) {\scriptsize{11}};
				\node[anchor=north west] at (10.4, -0.05) {\scriptsize{12}};
			\end{tikzpicture}
		\end{center}
		\caption{An example of the application of the bijection $\Gamma$.}
		\label{fig-exa-dyck-m}
	\end{figure}
	
	The \emph{height} of a Dyck path \( \mu \in \mathcal{D}_n \) is the maximum \( y \)-coordinate attained by the path, namely, 
	\(
	\h(\mu) = \max_{0 \le i \le 2n} y_i,
	\)
	where \( (x_i, y_i) \) is the position after the \(i\)-th step of the path starting from \( (x_0, y_0) = (0, 0) \). 
	It is clear that the bijection $\Gamma$ sends the statistic $\h$ on Dyck paths to the statistic $\crr$ 
	on nonnesting matchings; that is, $\h(\mu)=\crr(\Gamma(\mu))$ for any Dyck path $\mu$. The enumeration of the height statistic on Dyck paths appears in sequence \textup{A080936} in \cite{OEIS}.
	Suppose the first return of the Dyck path $\mu$ to the $x$-axis decomposes $\mu$ into an initial segment $\mu'=U\mu_1D$ and a remaining segment $\mu_2$, that is, $\mu=U\mu_1D\mu_2$. Then it is easy to see that
	\begin{align}\label{formular-h(D)}
		\h(\mu) = \max\{\h(\mu_1) + 1, \h(\mu_2)\}.
	\end{align}
	
	\begin{proof}[Proof of Theorem \ref{thm-equi-cr-nr-P1}]
		Using the decompositions described in Section~\ref{subsec-decomp},
		we set $M=\Theta(M_1)\oplus M_2\in\M_n(P_1)$,
		$P=P^*\oplus P_{\text{II}}\in\PP_n(\mathbf{3+1})$, and  
		$Q=\widetilde{P}\oplus P''\in\PP_n(\mathbf{N})$.
		Let $M^*=\Theta(M_1)$.  
		We also set $P_{\text{I}}$ to be the poset obtained from $P^*$ by reducing an element and rearranging its relations, and $P'$ to be the poset obtained by removing an isolated element from $\widetilde{P}$.
		Since
		\begin{align*}
			\crr(M) &= \max\{\crr(M^*), \crr(M_2)\}, & 
			\wid(P) &= \max\{\wid(P^*), \wid(P_{\text{II}})\}, & 
			\wid(Q) &= \max\{\wid(\widetilde{P}), \wid(P'')\},
		\end{align*}
		by \eqref{formular-h(D)}, it suffices to show that
		\begin{align}
			\crr(M^*) &= \crr(M_1) + 1, \label{cr-form} \\
			\wid(P^*) &= \wid(P_{\text{I}}) + 1, \label{formular-w(P_1)} \\
			\wid(\widetilde{P}) &= \wid(P') + 1.  \label{w-form-P2}
		\end{align}
		In (i)--(iii) below, we prove \eqref{cr-form}--\eqref{w-form-P2}, respectively.
		
		\noindent
		(i) Observe that any maximal crossing in $M$ is formed by consecutive openers and closers. Let 
		\[
		m_1 = [a_{i}, b_{i}],\ m_2 = [a_{i+1}, b_{i+1}],\ \ldots,\ m_{t} = [a_{i+t-1}, b_{i+t-1}]
		\]
		form a maximal \( t \)-crossing in $M_1$, where $t$ is the maximum possible size.
		In $M^*$, consider the arcs $m'_1=[a_{i-1}+1, b_{i}+1], m'_2=[a_{i}+1, b_{i+1}+1], \ldots, m'_{t+1}=[a_{i+t-1}+1,b_{i+t}+1]$, where $a_0=0$.
		Clearly, we have
		\[
		a_{i-1} < a_{i} < \dots < a_{i+t-1} < b_{i} < \dots < b_{i+t-1} < b_{i+t},
		\]
		and consequently $m'_1,m'_2,\dots,m'_{t+1}$ form a $(t+1)$-crossing in~$M^*$. Next, we shall prove that $t+1$ is the maximal size. 
		Suppose there exist arcs $[a_{j-1}+1, b_j+1], [a_j+1, b_{j+1}+1], \ldots, [a_{j+t}+1, b_{j+t+1}+1]$ forming a $(t+2)$-crossing in $M^*$. Then this leads to
		\[
		a_j < a_{j+1} < \cdots < a_{j+t} < b_j < b_{j+1} < \cdots < b_{j+t},
		\]
		implying that the arcs $[a_j, b_j], [a_{j+1}, b_{j+1}], \ldots, [a_{j+t}, b_{j+t}]$ form a $(t+1)$-crossing in $M_1$, which contradicts the maximality of $t$.
		Therefore, we get~\eqref{cr-form}.
		
		\noindent
		(ii) Consider a maximal-size antichain \( \{p_{i_1}, p_{i_2}, \ldots, p_{i_t}\} \) in \( P_{\text{I}} \), where \( i_1 < i_2 < \cdots < i_t \). We shall show that the elements in \( X = \{p^*_{i_1 - 1}, p^*_{i_2 - 1}, \ldots, p^*_{i_t-1}, p^*_{i_t}\} \) form a maximal-size antichain in \( P^* \) as follows.  
		
		(a) The elements in $\{p^*_{i_1 - 1}, p^*_{i_2 - 1},\ldots, p^*_{i_t - 1}\}$ must form an antichain in $P^*$. If not, there would exist $i_j, i_s \in \{i_1, \ldots, i_t\}$ with $p^*_{i_j-1} < p^*_{i_s-1}$ in $P^*$, implying $p_{i_j} < p_{i_s-1}$ in $P_{\rm I}$. Since $\ell(p_{i_s-1}) \leq \ell(p_{i_s})$, we have $p_{i_j}\in D(p_{i_s-1}) \subset D(p_{i_s})$, leading to $p_{i_j} < p_{i_s}$ in $P_{\text{I}}$, a contradiction. 
		Moreover, $p^*_{i_t}$ is incomparable with the other elements of $X$ by the construction of $P^*$.
		Thus $X$ is pairwise incomparable.
		
		(b) Suppose there exists \( Y = \{p^*_{j_1}, p^*_{j_2},\ldots, p^*_{j_{t+2}}\} \) forming an antichain in \( P^* \) with \( j_1 < j_2< \cdots < j_{t+2} \). Then \( \{p_{j_1+1}, p_{j_2+1}, \ldots, p_{j_{t+1}+1}\} \) forms an antichain in \( P_{\rm I}\). 
		If not, there would exist \( j_r < j_s \) in \( \{j_1, \ldots, j_{t+1}\} \) with \( p_{j_r+1} < p_{j_s+1} \) in \( P_{\text{I}} \), implying \( p^*_{j_r} < p^*_{j_s+1} \) in \( P^* \). Since \( \ell(p^*_{j_s+1}) \leq \ell(p^*_{j_{t+2}}) \), it follows that \( p^*_{j_r}\in D(p^*_{j_s+1}) \subset D(p^*_{j_{t+2}}) \), which implies \( p^*_{j_r} < p^*_{j_{t+2}} \), contradicting the choice of \( Y \).
		Thus, no antichain in \( P^* \) exceeds \( X \) in size.
		
		\noindent
		Combining (a) and (b), we get~\eqref{formular-w(P_1)}.
		
		\noindent
		(iii) Let $\{p_{i_1},p_{i_2},\dots,p_{i_{t}}\}$ be a maximal antichain in $P'$.
		By the construction of $\widetilde{P}$, these elements together with the newly added isolated element form an antichain of size~$t+1$, which is maximum. Hence we obtain~\eqref{w-form-P2}.
	\end{proof}
	
	\begin{remark} There are other statistics in the literature on various combinatorial structures that have the same distribution as the statistics in Theorem~\ref{thm-equi-cr-nr-P1}, such as the statistic of the maximum number of nesting arcs in the set of noncrossing matchings.  \end{remark}
	
	\subsubsection{The proofs of Theorems~\ref{thm-P1-nara}--\ref{thm-P1-P2-joint}}\label{proofs-three-theorems}
	
	To establish the equidistribution of the statistics, we use induction on the size of the objects.

	\begin{lem}\label{lem-P1-3+1}
		The statistics $(\mcr, \fcr, \bl)$ over $\M_n(P_1)$ and $(\magg(P), \minn, \ssd)$ over $\PP_n\bf{(3+1)}$ have the same distribution.
	\end{lem}
	\begin{proof}
		Let $M=\Theta(M_1) \oplus M_2 \in \mathcal{M}_n(P_1)$, where $M_1 \in \M_{k-1}(P_1)$, $M_2\in \M_{n-k}(P_1)$, and $M^*=\Theta(M_1)$. Set
		$
		P=P^{*}\oplus P_{\text{II}}\in\PP_{n}\textbf{(3+1)}$ (resp., $P_{\text{I}}$)
		to be the image of $M$ (resp., $M_1$) under $\Omega$. We consider two cases based on the properties of $M_1$.
		\begin{itemize}
			\item If $M_1=\emptyset$, then $P_{\text{I}}=\emptyset$. 
			Using the decompositions, we have  
			$M=\{[1,2]\}\oplus M_{2}$, and
			$P$ is obtained by adding a new unique minimal element to~$P_{\text{II}}$.
			Thus, it is clear that
			\begin{align}\label{sta-P1-case1}
				\mcr(M)= \mcr(M_2)+1, \hspace{0.3cm} \fcr(M)=1, \hspace{0.3cm}\bl(M)= \bl(M_2)+1,\notag \\ 
				\magg(P)=\magg(P_{\text{II}})+1, \hspace{0.3cm} \minn(P)=1, \hspace{0.3cm} \ssd(P)= \ssd(P_{\text{II}})+1.
			\end{align}
			
			\item If $M_1 \neq \emptyset$, then it is not difficult to see that $\bl(M^*)= 1$.
			By the proof of Theorem~\ref{thm-equi-cr-nr-P1} (see (i)), each maximal crossing in~$M_{1}$ corresponds to a maximal crossing in~$M^{*}$ whose size increases by~$1$ under~$\Theta$; hence
			$\fcr(M^{*}) = \fcr(M_{1}) + 1$.
			Since $\Theta$ preserves the number of maximal crossings, we also have
			$\mcr(M^{*}) = \mcr(M_{1})$.
			Therefore, we have
			\begin{align}\label{sta-P1-case2}
				\mcr(M)= \mcr(M_1)+\mcr(M_2), \hspace{0.3cm}  \fcr(M)=\fcr(M_1)+1, \hspace{0.3cm} \bl(M)=\bl(M_2)+1.
			\end{align}
			Moreover, by the construction of $P^{*}$, we obtain one extra minimal element compared to $P_{\text{I}}$. Hence
			$
			\minn(P^{*}) = \minn(P_{\text{I}}) + 1,
			$
			while
			$
			\magg(P^{*}) = \magg(P_{\text{I}})$ and
			$\ssd(P^{*}) = 1$.
			Thus, we have
			\begin{align*}
				\magg(P)= \magg(P_{\text{I}})+\magg(P_{\text{II}}), \hspace{0.3cm}  \minn(P)=\minn(P_{\text{I}})+1, \hspace{0.3cm} \ssd(P)=\ssd(P_{\text{II}})+1.
			\end{align*}
		\end{itemize}
		The desired properties in both cases follow by induction on $n$, which completes the proof.
	\end{proof}
	
	\begin{lem}\label{lem-P2-N-101-3142}
		The statistics $(\mcr, \fcr, \bl)$ over $\M_n(P_2)$, $(\magg(P), \minn, \smc)$ over $\PP_n\bf{(N)}$, $(\lmax, \zero, \rmin)$ over $\A_n(101)$, and $(\rmin, \idr, \lmax)$ over $\F_n(3142)$ have the same distribution.
	\end{lem} 
	\begin{proof} 
		Let $M= \V(M') \oplus M'' \in \M_n(P_2)$, where $M' \in \M_{k-1}(P_2)$ and $M'' \in \M_{n-k}(P_2)$, and $\widetilde{M}=\V(M')$. 
		We also define
		\[
		P=\widetilde{P}\oplus P''\in\PP_{n}\textbf{(N)},\quad
		\alpha=\widetilde{\alpha}\oplus\alpha'' \in \A_n(101),\quad
		\pi=\widetilde{\pi}\oplus\pi'' \in \F_n(3142),
		\]
		to be the respective images of $M$ under the bijections $\Omega$, $\Psi$, and $\Upsilon$. Similarly, $P'$, $\alpha'$, and $\pi'$ denote the corresponding images of $M'$ under these same bijections. We consider the following two cases.
		\begin{itemize}
			\item If $M' =\emptyset$, then $P'=\alpha'=\pi'=\emptyset$. 
			By the decompositions, we have  
			$M=\{[1,2]\}\oplus M''$,  
			$P$ is obtained by adding a new unique minimal element to~$P''$, 
			$\alpha$ by prepending~$0$ to~$\alpha''$, and  
			$\pi=1\oplus \pi''$. 
			Thus, this case follows the same recurrences as~\eqref{sta-P1-case1} for the statistics over $\M_n(P_{2})$, $\PP_{n}\bf{(N)}$, $\A_n(101)$, and $\F_n(3142)$.
			
			\item If $M' \neq \emptyset$, then it is not difficult to see that $\bl(\widetilde{M})= 1$.
			By the construction of~$\V$ in Corollary~\ref{coro-P2-irre}, the newly added reduction arc in~$M'$ yields $\fcr(\widetilde{M})=\fcr(M')+1$
			and leaves the number of maximal crossings unchanged; hence $\mcr(\widetilde{M})=\mcr(M')$.
			Therefore, we have
			\begin{align*}
				\mcr(M)= \mcr(M')+\mcr(M''), \hspace{0.3cm}  \fcr(M)=\fcr(M')+1, \hspace{0.3cm} \bl(M)=\bl(M'')+1.
			\end{align*}
			By the constructions of $\widetilde P$, $\widetilde \alpha$, and $\widetilde \pi$, it is clear that
			$\smc(\widetilde{P})=\rmin(\widetilde{\alpha})=\lmax(\widetilde{\pi})=1$,
			and the $\minn$, $\zero$, and $\idr$ statistics each increase by one in
			$\widetilde{P}$, $\widetilde{\alpha}$ and $\widetilde{\pi}$
			compared to $P'$, $\alpha'$ and $\pi'$, respectively.
			We also see that
			$\magg(\widetilde P)= \magg(P')$,
			$\lmax(\widetilde{\alpha})=\lmax(\alpha')$, and  $\rmin(\widetilde{\pi})=\rmin(\pi')$. 
			Furthermore, by the decompositions of $P$, $\alpha$, and $\pi$, we have
			\begin{align*}
				&\magg(P)=\magg(P')+\magg(P''),\hspace{0.3cm} \minn(P)=\minn(P')+1,\hspace{0.3cm}\smc(P)=\smc(P'')+1,\\ &\lmax(\alpha)=\lmax(\alpha')+\lmax(\alpha''), \hspace{0.3cm}\zero(\alpha)=\zero(\alpha')+1, \hspace{0.3cm} \rmin(\alpha)=\rmin(\alpha'')+1, \\ 
				&\rmin(\pi)=\rmin(\pi')+\rmin(\pi''), \hspace{0.3cm} \idr(\pi)=\idr(\pi')+1, \hspace{0.3cm} \lmax(\pi)=\lmax(\pi'')+1.
			\end{align*}
		\end{itemize}
		The proof now follows by induction on $n$.
	\end{proof}
	
	\begin{coro}\label{coro-equi-all}
		The statistics $(\mcr, \fcr, \bl)$ over $\M_n(P_1)$ and $\M_n(P_2)$, $(\magg(P), \minn, \ssd)$ over $\PP_n\bf{(3+1)}$, $(\magg(P), \minn, \smc)$ over $\PP_n\bf{(N)}$, $(\lmax, \zero, \rmin)$ over $\A_n(101)$ and $(\rmin, \idr, \lmax)$ over $\F_n(3142)$ all have the same distribution.
	\end{coro}
	\begin{proof}
		From the proofs of Lemmas~\ref{lem-P1-3+1} and~\ref{lem-P2-N-101-3142}, 
		the triples of statistics over each set satisfy the same recurrence relations, 
		which completes the proof by induction on the size of the objects.
	\end{proof}
	
	\begin{remark}
		Note that the bijection $\Omega$ provides a proof of the equidistribution of the statistics 
		$(\nr, \mcr, \fcr, \bl)$ over $\M_n$ and $(\h, \magg(P), \minn, \ssd)$ over $\PP_n$.
		Moreover, using the equivalent definition of~$\magg(P)$, the number of distinct strict down-sets, given by Fishburn~\cite{Fishburn-1983,Fishburn-1985}, we note that~$\mcr$, the number of maximal crossings of~$M\in\M_n$, equals the number of distinct sets, each of which is associated with an arc $[a,b]$ and consists of all arcs whose closers precede the opener~$a$.
	\end{remark}
	
	The following two corollaries follow from Lemma~\ref{lem-P2-N-101-3142}.
	
	\begin{coro}\label{coro-coin-MP2}
		For any $M \in \M_n(P_2)$, we have $\nr(M)=\mcr(M)$.
	\end{coro}
	\begin{proof}
		It is easy to verify that, in both cases of the proof of Lemma~\ref{lem-P2-N-101-3142}, 
		the statistic $\nr$ satisfies the same recurrence and initial conditions as $\mcr$.
	\end{proof}
	
	\begin{coro}\label{coro-coin-P-N}
		For any $P \in \PP_n\bf{(N)}$, we have
		$
		\h(P)=\magg(P)$ and $\ssd(P)=\smc(P)$.
	\end{coro}
	\begin{proof}
		It is not difficult to verify that the statistics $\h$ and $\ssd$ over $\PP_n\bf{(N)}$ satisfy the same recurrence and initial conditions as $\magg(P)$ and $\smc$ stated in Lemma~\ref{lem-P2-N-101-3142}. Hence, the proof follows by induction on~$n$.
	\end{proof}
	
	Let $T(x,y,z,t):=\displaystyle\sum_{M \in \M(P_1)} x^{{\mcr(M)}}y^{{\fcr(M)}} z^{{\bl(M)}}t^{|M|}$. 
	We next give a recurrence relation for these polynomials.
	
	\begin{lem}\label{lem-sta-P1}
		We have
		\begin{equation}\label{sta-P1}
			T(x,y,z,t)= 1+yzt\, T(x,1,z,t) \left(x-1+T(x,y,1,t)\right).
		\end{equation}
	\end{lem}
	\begin{proof}
		The proof is based on the proof of Lemma~\ref{lem-P1-3+1}. 
		For the case $M_1=\emptyset$, together with $|M|=|M_2|+1$, we have a contribution of $xyzt\, T(x,1,z,t)$ to $T(x,y,z,t)$. When $M_1 \neq \emptyset$ (so that $|M|=|M_1|+|M_2|+1$), it gives $yzt \left(T(x,y,1,t)-1 \right)\, T(x,1,z,t)$.
		Summarizing the two cases above, we obtain~\eqref{sta-P1}.
	\end{proof}
	
	\begin{proof}[Proof of Theorem~\ref{thm-P1-nara}]
		By setting $y=z=1$ in~\eqref{sta-P1}, we obtain
		$$
		T(x,1,1,t)=1+t(x-1)T(x,1,1,t)+tT^2(x,1,1,t).
		$$
		By solving this equation, we obtain
		\begin{align}\label{fomular-T(x,1,1,t)}
			T(x,1,1,t)&=\frac{1+t-xt-\sqrt{(1+t-xt)^2-4t}}{2t} \notag \\
			&= \frac{1+2t-t(1+x)-\sqrt{\bigl(1-t(1+x)\bigr)^{2}-4xt^{2}}}{2t}
			= N(x,t),
		\end{align}
		where the last equality follows from~\eqref{Narayana-formula}. We use formula \eqref{fomular-T(x,1,1,t)} and Corollaries~\ref{coro-equi-all}, \ref{coro-coin-MP2}, and \ref{coro-coin-P-N} to complete the proof.
	\end{proof}
	
	\begin{proof}[Proof of Theorem~\ref{thm-ballot}.] 
		Setting $x=1$ in~\eqref{sta-P1}, we obtain
		\begin{align}\label{fomular-T(1,y,z,t)}
			T(1,y,z,t) &=1+yztT(1,1,z,t)\,T(1,y,1,t).
		\end{align}
		Furthermore, by setting $y=1$ in~\eqref{fomular-T(1,y,z,t)}, we obtain
		\begin{align*}
			T(1,1,z,t)=1+ztT(1,1,z,t)\,T(1,1,1,t) = 1+ ztT(1,1,z,t)\,C(t).
		\end{align*}
		By solving the equation, we get $T(1,1,z,t)=\frac{1}{1-ztC(t)}$, 
		which coincides with $C(z,t)$. Similarly, by setting $z=1$ in~\eqref{fomular-T(1,y,z,t)}, we derive $T(1,y,1,t)= C(y,t)$. 
		Furthermore, by substituting $T(1,1,z,t)$ and $T(1,y,1,t)$ into~\eqref{fomular-T(1,y,z,t)}, we obtain the rightmost expression in~\eqref{formula-fcr-block-C(x,t)}. Together with Corollary~\ref{coro-equi-all}, we obtain the desired result.
	\end{proof}
	
	\begin{proof}[Proof of Theorem~\ref{thm-P1-P2-joint}]
		By setting $y=1$ in~\eqref{sta-P1}, we obtain
		\begin{equation}\label{T(x1zt)-form}
			T(x,1,z,t)= 1+zt T(x,1,z,t)(x-1+T(x,1,1,t)).
		\end{equation}
		Furthermore, by substituting~\eqref{fomular-T(x,1,1,t)} into (\ref{T(x1zt)-form}) and solving, we get 
		\begin{align}\label{fomular-T(x,1,z,t)}
			T(x,1,z,t)&=\frac{2}{2 + z\left(t-1 - t x + \sqrt{1 + t^{2}(x - 1)^{2} - 2 t (1 + x)}\right) } \notag \\
			&=\frac{1}{1 - zt\left(x-1+N(x,t)\right)}.
		\end{align}
		Moreover, by setting $z=1$ in~\eqref{sta-P1}, we obtain
		\begin{equation}\label{Txy1t-formula}
			T(x,y,1,t)= 1+yt T(x,1,1,t)(x-1+T(x,y,1,t)).
		\end{equation}
		By substituting~\eqref{fomular-T(x,1,1,t)} into (\ref{Txy1t-formula}) and solving, we have
		\begin{align}\label{fomular-T(x,y,1,t)}
			T(x,y,1,t) &=1 + \frac{x y \left(1 - xt- 2yt+t - \sqrt{1 + t^2 (1 - x)^2 - 2 t (1 + x)} \right)}{2 - 2 y \left( 1- xt-yt+t \right)} \notag \\
			&=\frac{1+yt(x-1) N(x,t)}{1-yt N(x,t)}.
		\end{align}
		Finally, by substituting~\eqref{fomular-T(x,1,z,t)} and \eqref{fomular-T(x,y,1,t)} 
		into~\eqref{sta-P1} and solving, we derive that $T(x,y,z,t)$ coincides with~\eqref{fomular-sta-P1-P2-joint}, 
		and~\eqref{main-formulas-dist} is obtained by substituting~\eqref{Narayana-formula} into~\eqref{fomular-sta-P1-P2-joint}.
		This, together with Corollary~\ref{coro-equi-all}, completes the proof.
	\end{proof}
	
	\section{Concluding remarks}
	
	The following conjecture extends the results in Theorem~\ref{thm-equi-cr-nr-P1}.
	\begin{conj}\label{conj-nr}
		The statistics $\nr$ over $\M_n(P_1)$ and  $\h$ over $\D_n$ have the same distribution, i.e.,
		\begin{align*}
			\sum_{M \in \mathcal{M}_n(P_1)} x^{\nr(M)} 
			= \sum_{\mu \in \mathcal{D}_n} x^{\h(\mu)}.
		\end{align*}
	\end{conj}
	Via the bijection $\Omega$, Conjecture~\ref{conj-nr} implies that the height statistics defined on $\mathcal{P}_n(\mathbf{3+1})$ and $\D_n$ are equidistributed.
	Dilworth's theorem~\cite{Dilworth} states that, in any finite poset, the width equals the minimum number of chains needed to cover it. Moreover, Mirsky~\cite{Mirsky} proved that the height of a poset equals the minimum number of antichains needed to cover it; see also~\cite{Greene-Kleitman}.
	If Conjecture~\ref{conj-nr} holds, then it, together with Theorem~\ref{thm-equi-cr-nr-P1}, implies that the height over $\mathcal{P}_n(\mathbf{3+1})$, the width over $\mathcal{P}_n(\mathbf{3+1})$, the maximal crossing size over $\M_n(P_1)$, and the maximal noncrossing size over $\M_n(P_1)$ all have the same distribution.
	
	\section*{Acknowledgements} 
	The second author is grateful to the SUSTech International Center for Mathematics for its hospitality during his visit in April 2026.
	The work of the third author was supported by the National Natural Science Foundation of China (No.\ 12171362) and the Tianjin Municipal Natural Science Foundation (No.\ 25JCYBJC00430).

\end{document}